\documentclass{article}

\newcommand{\paperTitle}{Operator Splitting Methods with Online Scaling}
\newcommand{\paperAbstract}{This paper develops a principled framework for automatically tuning preconditioners in operator splitting methods, including forward-backward splitting, Douglas-Rachford splitting, and the alternating direction method of multipliers. We interpret these splitting methods as preconditioned gradient descent applied to their corresponding envelope functions and use online learning to adaptively learn the preconditioner during optimization. Our framework achieves asymptotically faster convergence rates and accelerates the solution of important optimization problems, including Lasso, $\ell_1$-regularized logistic regression, and quadratic programming. Extensive numerical experiments validate our theoretical results and demonstrate the practical effectiveness of the proposed approach.
}

\usepackage[utf8]{inputenc}
\usepackage[T1]{fontenc}
\usepackage[english]{babel}
\usepackage{amsmath,amssymb,graphicx,xcolor,tabularx,booktabs,multirow}
\usepackage{amsthm}
\usepackage[shortlabels]{enumitem}
\usepackage[ruled,linesnumbered]{algorithm2e}
\SetKwInput{KwInit}{Init}
\SetKwInput{KwOutput}{Output}
\usepackage{subcaption}
\usepackage{placeins}
\usepackage{aliascnt}
\usepackage{hyperref}
\usepackage{geometry}
\usepackage{cleveref}
\crefname{appendix}{Appendix}{Appendices}
\Crefname{appendix}{Appendix}{Appendices}

\graphicspath{{figs/}}

\makeatletter
\newcommand{\appendixtableofcontents}{\noindent{\huge\bfseries Appendix\par}\vspace{1.5em}\noindent{\Large\bfseries Table of Contents\par}\vspace{0.25em}\hrule
  \begingroup
  \setlength{\parskip}{0pt}\setcounter{tocdepth}{2}\@starttoc{atc}\endgroup
  \hrule
}
\newcommand{\appendixsection}[1]{\section{#1}\addcontentsline{atc}{section}{\protect\numberline{\thesection}#1}}
\newcommand{\appendixsubsection}[1]{\subsection{#1}\addcontentsline{atc}{subsection}{\protect\numberline{\thesubsection}#1}}
\makeatother

\setlist[itemize]{leftmargin=1.4em,itemsep=2pt,topsep=2pt}
\setlist[enumerate]{leftmargin=1.6em,itemsep=2pt,topsep=2pt}

\theoremstyle{plain}
  \newtheorem{theorem}{Theorem}[section]
  \newaliascnt{lemma}{theorem}
  \newtheorem{lemma}[lemma]{Lemma}
  \aliascntresetthe{lemma}
  \newaliascnt{corollary}{theorem}
  
  \aliascntresetthe{corollary}
  \newaliascnt{proposition}{theorem}
  \newtheorem{proposition}[proposition]{Proposition}
  \aliascntresetthe{proposition}
  \theoremstyle{definition}
  \newaliascnt{definition}{theorem}
  
  \aliascntresetthe{definition}
  \newaliascnt{assumption}{theorem}
  \newtheorem{assumption}[assumption]{Assumption}
  \aliascntresetthe{assumption}
  \newaliascnt{example}{theorem}
  \newtheorem{example}[example]{Example}
  \aliascntresetthe{example}
  \newaliascnt{remark}{theorem}
  \newtheorem{remark}[remark]{Remark}
  \aliascntresetthe{remark}

\crefname{theorem}{theorem}{theorems}
\Crefname{theorem}{Theorem}{Theorems}
\crefname{lemma}{lemma}{lemmas}
\Crefname{lemma}{Lemma}{Lemmas}
\crefname{corollary}{corollary}{corollaries}
\Crefname{corollary}{Corollary}{Corollaries}
\crefname{proposition}{proposition}{propositions}
\Crefname{proposition}{Proposition}{Propositions}
\crefname{definition}{definition}{definitions}
\Crefname{definition}{Definition}{Definitions}
\crefname{assumption}{assumption}{assumptions}
\Crefname{assumption}{Assumption}{Assumptions}
\crefname{example}{example}{examples}
\Crefname{example}{Example}{Examples}
\crefname{remark}{remark}{remarks}
\Crefname{remark}{Remark}{Remarks}
\newcommand{\nonconverted}[1]{\mbox{}}

\newcommand{\assign}{:=}
\newcommand{\backassign}{=:}
\newcommand{\ie}{\emph{i.e.}}

\newcommand{\tmop}[1]{\ensuremath{\operatorname{#1}}}

\newcommand{\tmstrong}[1]{\textbf{#1}}
\newcommand{\tmtextbf}[1]{\text{{\bfseries{#1}}}}
\newcommand{\tmtextit}[1]{\text{{\itshape{#1}}}}

\newenvironment{itemizedot}{\begin{itemize}}{\end{itemize}}

\begin{document}

\date{}
\title{\paperTitle}
\author{Wanyu Zhang\footnotemark[1]\qquad 
Wenzhi Gao\footnotemark[1] \qquad
Madeleine Udell\footnotemark[1]
}

\renewcommand{\thefootnote}{\fnsymbol{footnote}}
\footnotetext[1]{Stanford University. \texttt{\{zwanyu,gwz,udell\}@stanford.edu}}
\renewcommand{\thefootnote}{\arabic{footnote}}
\maketitle

\begin{abstract}
  \paperAbstract
\end{abstract}

\section{Introduction}

Operator splitting methods are standard tools for large-scale convex optimization problems, such as regularized statistical learning and constrained optimization.  Well-known operator splitting methods include 
the forward--backward splitting (FBS, or the proximal gradient method), Douglas--Rachford splitting (DRS), 
and the alternating direction method of multipliers (ADMM).
The practical efficiency of these methods is highly sensitive to the
hyperparameters: for example, the proximal
stepsize in FBS and DRS, relaxation coefficients, and the penalty parameter in
ADMM.
The choice of these parameters substantially affects the geometry and scale of the algorithmic iterations and can greatly affect practical performance.

Despite their importance, hyperparameters in operator splitting methods are notoriously difficult to tune, and existing approaches broadly fall into two categories. On the one hand, variable-metric splitting methods \cite{combettes2014variable,chouzenoux2014variable} replace the Euclidean metric with one induced by a positive definite matrix and establish convergence for prescribed or structure-aware choices of the metric. Related approaches use randomized linear algebra to construct scalable, structured metrics \cite{sun2025sapphire,frangella2023geniadmm}. However, the design of these metrics often relies heavily on specific problem structure, making such approaches difficult to generalize across problem classes.
On the other hand, adaptive methods, particularly those developed for ADMM, tune scalar penalty parameters based on information observed during the algorithmic iterations, such as primal and dual residuals \cite{he2000alternating,xu2017adaptive,stellato2020osqp} or augmented-Lagrangian sufficient decrease \cite{themelis2022douglas}. While these methods incorporate feedback from the optimization trajectory and often perform well in practice, they typically either lack convergence guarantees or are restricted to scalar hyperparameters. In particular, existing approaches do not provide a general framework for dynamically learning matrix-valued preconditioners from the observed optimization trajectory.

To address this challenge, we develop a learning-based framework for adaptively selecting these hyperparameters with strong convergence guarantees. Our starting point is a recent advance in smooth unconstrained minimization, the online scaled gradient method (OSGM) \cite{pmlr-v291-gao25a}. OSGM casts preconditioner selection for smooth convex optimization as an online decision-making problem and provides both global and local convergence guarantees, making it a natural candidate here. However, OSGM cannot be applied directly to operator-splitting methods. First, these methods are fixed-point iterations rather than gradient methods, and OSGM theory is not immediately applicable. Second, operator-splitting methods are designed to handle nonsmooth objectives, whereas OSGM heavily relies on the descent properties of smooth functions.
\subsection{Contributions}
This paper develops a principled framework for adaptively learning matrix-valued preconditioners in operator splitting methods. Our key idea is to view operator splitting as preconditioned gradient descent on an associated envelope function \cite{patrinos2013proximal}, which allows us to bring online preconditioner learning to splitting methods.

\textbf{Smooth and weakly convex envelopes.}
We establish sufficient conditions under which the forward--backward envelope (FBE) and Douglas--Rachford envelope (DRE) are smooth and (weakly) convex for a broad class of problems, including those with binary and multinomial logistic losses. These results extend previous guarantees, which were limited to convex quadratic smooth terms \cite{patrinos2014forward,patrinos2014douglas}, and yield sharper bounds on the smoothness constants even in the quadratic setting.

\textbf{Online-scaled operator splitting (OSOP).}
We introduce OSOP, a principled method for learning matrix-valued preconditioners directly from the optimization trajectory. At iteration $k$, OSOP performs the preconditioned splitting update
$$
x^{k+1}=x^k-P_k(x^k-T(x^k)),
$$
where $T$ is the underlying splitting operator, and updates the preconditioner $P_k$ using online gradient feedback derived from the corresponding envelope. The framework applies directly to FBS and DRS and extends to ADMM through its equivalence with DRS.

\subsection{Related work}

\textbf{Envelope-based methods.} Existing algorithms that use envelopes fall
into two regimes.

\emph{(i) Envelope as a merit function.} 
PANOC {\cite{stella2017simple}} combines forward-backward steps with quasi-Newton
directions and uses the FBE as a \emph{merit function} for line-search safeguards. ZeroFPR {\cite{themelis2018forward}}
extends this framework to nonconvex composite problems. These
methods retain global convergence guarantees and can achieve local superlinear
convergence when the search directions satisfy suitable conditions.

\emph{(ii) Envelope directly as a smooth objective.} Patrinos et al.
{\cite{patrinos2014douglas}} apply Nesterov acceleration to the DRE, while
Patrinos et al. {\cite{patrinos2014forward}} develop truncated Newton methods based on the FBE. 
Their global accelerated analyses
require a convex and smooth envelope, for which sufficient conditions have primarily been established when the smooth term is convex quadratic \cite{patrinos2014forward,patrinos2014douglas}. 
We identify a general sufficient condition for smoothness and convexity beyond the quadratic case.

Compared to prior envelope methods, OSOP uses the envelope to provide online feedback for learning the residual scaling itself, rather than only as a merit function or a fixed smooth objective.

\textbf{Accelerated operator splitting.} 
Momentum-based methods modify the splitting trajectory through extrapolation.
FISTA {\cite{beck2009fast}} attains
an $\mathcal{O} (1 / k^2)$ rate for composite convex problems. Goldstein et
al. {\cite{goldstein2014fast}} develop fast ADMM, while inertial methods
such as iPiano {\cite{ochs2014ipiano}} extend momentum-based splitting to
nonconvex settings. 
Sun et al. {\cite{sun2025accelerating}} accelerate preconditioned ADMM through degenerate proximal point mappings and obtain nonergodic convergence rates for the KKT residual.
Kim {\cite{kim2021accelerated}} develops an accelerated
proximal point method with an
$\mathcal{O} (1 / k^2)$ rate with respect to the fixed-point residual.  
Anderson acceleration instead combines previous fixed-point residuals and has been applied to fixed-point iterations {\cite{zhang2020globally}}, proximal gradient methods
{\cite{mai2020anderson}}, and DRS {\cite{fu2020anderson}}.

\textbf{Variable metrics and preconditioning.}
A closely related line of work accelerates splitting methods by changing their metric or scale.

\emph{(i) Scalar penalty adaptation for ADMM.} 
Many methods adapt the augmented-Lagrangian penalty $\rho$ using observed iterates.
He et al. {\cite{he2000alternating}} propose a residual-balancing update rule
based on the observed primal and dual residuals. OSQP
{\cite{stellato2020osqp}} employs a multiplicative update rule for $\rho$ in
quadratic programs. Xu et al. {\cite{xu2017adaptive}} propose a spectral
penalty selection rule inspired by the Barzilai-Borwein steps, 
which was extended to jointly adapt the penalty and relaxation parameters {\cite{xu2017aradmm}}.
Themelis et al. \cite{themelis2022douglas} adapt the penalty by augmented-Lagrangian backtracking.
Lorenz and Tran-Dinh {\cite{lorenz2019nonstationary}} adapt the DRS stepsize and the corresponding ADMM penalty with a convergence guarantee. Tang and Toh {\cite{tang2024selfadaptive}} establish convergence for a self-adaptive ADMM when one objective block is strongly convex.

These rules are inexpensive, but only tune a scalar rather than a general transformation of the residual.
OSOP instead learns a matrix-valued scaling. The two adaptations are orthogonal,
and our experiments apply online scaling on top of both residual-balancing and spectral-penalty ADMM.

\emph{(ii) Matrix metrics.} Variable-metric FBS was developed by Combettes and
Vu {\cite{combettes2014variable}} and Chouzenoux et al.
{\cite{chouzenoux2014variable}}, 
while Raguet and Landrieu
{\cite{raguet2015preconditioning}} study preconditioned generalized FBS for monotone
inclusions. These works establish convergence for prescribed or
problem-structured metrics. Giselsson and Boyd {\cite{giselsson2017linear}}
characterize the metric selection problem for DRS and ADMM as an optimization problem.
For ADMM, Li et al. {\cite{li2016majorized}} allow indefinite proximal terms in a majorized scheme, while Gu and Yamashita {\cite{gu2021broyden}} generate variable proximal metrics through BFGS and Broyden updates.

These methods place the metric inside the proximal subproblems, so changing the metric changes the associated proximal operators. OSOP instead applies the learned preconditioner outside the splitting oracle, directly to the fixed-point residual, so the underlying proximal oracles remain unchanged.

\emph{(iii) Structured and randomized preconditioners.} This line of work
computes \emph{explicit and scalable} preconditioners from randomized
linear algebra. NysADMM {\cite{zhao2022nysadmm}} accelerates inexact ADMM with
a randomized low-rank Nystr\"om approximation. GeNI-ADMM
{\cite{frangella2023geniadmm}} establishes a general framework for inexact
ADMM with subproblem approximations, and is implemented
at scale in the GeNIOS solver {\cite{diamandis2023genios}}. In the stochastic
regime, SAPPHIRE {\cite{sun2025sapphire}} combines sketch-based
preconditioning with a scaled proximal mapping.

\subsection{Notation and definitions}\label{subsec:notation}
We work in finite-dimensional Euclidean spaces. For vectors, $\langle x, y
\rangle = x^{\top} y$ denotes the standard inner product, and $\|x\| =
\sqrt{\langle x, x \rangle}$ denotes the induced Euclidean norm. For matrices,
$\|A\|$ denotes the induced operator norm. When matrices are viewed as
variables, their ambient inner product is the trace inner product
$\langle A, B \rangle = \tmop{tr} (A^{\top} B)$, with $\|A\|_F =
\sqrt{\langle A, A \rangle}$ denoting the induced Frobenius norm. We write $I$ for
the identity matrix of the appropriate dimension and $\lambda_{\min} (A)$
and $\lambda_{\max} (A)$ for the smallest and largest eigenvalues of a
symmetric matrix. The notation $A \preceq B$ means that $B - A$ is positive
semidefinite.

Functions may be extended-real-valued, $h : \mathbb{R}^n \to \mathbb{R} \cup
\{ + \infty \}$. For a set $C \subseteq \mathbb{R}^n$, its indicator
function $\delta_C$ is defined by $\delta_C (x) \assign 0$ if $x \in C$
and $\delta_C (x) \assign + \infty$ otherwise.
The effective domain is $\tmop{dom} h \assign \{ x \mid h (x) < + \infty
\}$. A function is proper if $\tmop{dom} h \neq \varnothing$ and it never
takes the value $- \infty$. It is closed if it is lower
semicontinuous, equivalently, if its epigraph is closed. For a convex function
$h$, $\partial h (x)$ denotes the convex subdifferential.
A twice-differentiable function $h$ has $H$-Lipschitz Hessian on $\Omega$ if
\[ \| \nabla^2 h (x) - \nabla^2 h (y) \| \leq H \| x - y \|, \quad \forall x, y
   \in \Omega . \]
Let $\Omega$ be convex. A proper function $h$ is $\mu$-convex on $\Omega$ if
the function $x \mapsto h(x) - \tfrac{\mu}{2}\|x\|^2$ is convex on $\Omega$.
Thus, $\mu>0$, $\mu=0$, and $\mu<0$ correspond to strong convexity,
convexity, and $(-\mu)$-weak convexity, respectively.

\paragraph{Proximal operator and Moreau envelope.}\label{subsec:moreau}
Let $h : \mathbb{R}^n \to \mathbb{R} \cup \{ + \infty \}$ be proper, closed,
and convex, and let $\gamma > 0$. The proximal operator of $h$ with stepsize
$\gamma$ is
\[ \tmop{prox}_{\gamma h} (z) \assign \arg \min_x \left\{ h (x) +
   \tfrac{1}{2 \gamma} \| x - z \|^2 \right\} . \]
We call $h$ proximable when this proximal operator can be
evaluated efficiently. The proximal
optimality condition gives
\[ \tfrac{1}{\gamma} (z - \tmop{prox}_{\gamma h} (z)) \in \partial h (\tmop{prox}_{\gamma h} (z)) . \]

The Moreau envelope of $h$ with parameter $\gamma$ is
\begin{equation}\label{eq:env-compute}
    h_{\gamma} (z) \assign \min_x \left\{ h (x) + \tfrac{1}{2 \gamma} \| x - z
   \|^2 \right\} = h (\tmop{prox}_{\gamma h} (z)) + \tfrac{1}{2 \gamma} \| z -
   \tmop{prox}_{\gamma h} (z) \|^2 .
\end{equation}
For a proper, closed, and convex function $h$, the envelope $h_{\gamma}$ is convex and
continuously differentiable, with
\[ \nabla h_{\gamma} (z) = \tfrac{1}{\gamma} (z - \tmop{prox}_{\gamma h} (z))
   \]
and $\nabla h_{\gamma}$ is $\tfrac{1}{\gamma}$-Lipschitz. Thus, we can view a proximal step as a gradient step on the envelope $h_{\gamma}$:
\begin{equation}\label{eq:prox-moreau}
   \tmop{prox}_{\gamma h} (z)= z - \gamma\nabla h_{\gamma} (z).
\end{equation}
If $h$ is
differentiable, then the proximal optimality condition also gives
\[ \nabla h_{\gamma} (z) = \nabla h (\tmop{prox}_{\gamma h} (z)) . \]

\paragraph{Structure of the paper.}
\Cref{sec:components} introduces two algorithmic components used by the main algorithms: OSGM for smooth minimization problems and envelopes for splitting methods.
\Cref{sec:osop} presents online-scaled operator splitting (OSOP) for accelerating FBS, DRS, and ADMM.
\Cref{sec:analysis} first establishes the smoothness and convexity of envelopes under a generalized sufficient condition and then gives the global and local convergence analyses for OSOP.
\Cref{sec:exp} demonstrates the performance of the proposed methods through numerical experiments on statistical learning problems and QPs.

\section{Algorithmic components} \label{sec:components}

This section reviews two components used by the main algorithms of this paper: (1) the envelopes for splitting methods, on which FBS and DRS can be viewed as variable-metric gradient methods, and (2) the online-scaled gradient method for smooth optimization problems.

We start with the basics of FBS and DRS.
FBS and DRS solve the composite convex problem
\begin{equation}
  \min_x F (x) \assign f (x) + g (x).
  \label{eq:composite-obj}
\end{equation}
Throughout the paper we make the following assumption.
\begin{assumption}\label{ass:standing}
  The functions $f$ and $g$ are proper, closed, and convex. The function $f$ is twice differentiable and $L$-smooth and is
  additionally proximable in the DR case.
The function $g$ is proximable.
\end{assumption}

Given the current iterate $x$, 
both methods update via the fixed point map $x^+ = T (x)$, where
\begin{equation}
  \begin{aligned}
    T^{\mathrm{FB}} (x) &\assign \tmop{prox}_{\gamma g} (x - \gamma \nabla f
    (x)), \\
    T^{\mathrm{DR}} (x) &\assign x + \tmop{prox}_{\gamma g} (2
    \tmop{prox}_{\gamma f} (x) - x) - \tmop{prox}_{\gamma f} (x) .
  \end{aligned}
\end{equation}

{\textbf{Unified notation.}} 
For brevity, we use the following unified notation to describe both FBS and DRS and drop the FB/DR superscript whenever a
statement holds for both cases.
Each splitting step $T(x)$ generates two points
$p (x)$ and $q (x)$:
\begin{equation}
  \begin{aligned}
    \text{FB:} \quad & p (x) \assign x, \quad q (x) \assign
    T^{\mathrm{FB}} (x) = \tmop{prox}_{\gamma g} (x - \gamma \nabla f (x)),\\
    \text{DR:} \quad & p (x) \assign \tmop{prox}_{\gamma f}
    (x), \quad q (x) \assign \tmop{prox}_{\gamma g} (2 p (x) - x) .
  \end{aligned} \label{eq:pq-maps}
\end{equation}
In both cases, the splitting update moves along the difference of the two
points, $T (x) = x - (p (x) - q (x))$. Equivalently, the fixed-point
residual is $x - T (x) = p (x) - q (x)$.

\subsection{Envelopes for FBS and DRS} \label{subsec:envelopes}

As shown by \eqref{eq:prox-moreau}, a proximal step can be interpreted as a gradient step on the smooth Moreau envelope. 
Similarly, previous work \cite{patrinos2014forward,patrinos2014douglas} shows that
FBS and DRS are variable-metric gradient methods on functions called
the forward-backward envelope (FBE)  and the
Douglas-Rachford envelope (DRE), respectively:
\begin{align}
  F^{\mathrm{FB}}_{\gamma} (x) & = f (x) - \tfrac{\gamma}{2} \| \nabla f (x)
  \|^2 + g_{\gamma} (x - \gamma \nabla f (x)),  \label{eq:fbe}\\
  F^{\mathrm{DR}}_{\gamma} (x) & = f_{\gamma} (x) - \gamma \| \nabla f_{\gamma}
  (x) \|^2 + g_{\gamma} (x - 2 \gamma \nabla f_{\gamma} (x)) .  \label{eq:dre}
\end{align}
Under the standing \Cref{ass:standing}, the literature establishes two properties of the FBE and DRE.
First, the envelope $F_{\gamma}$ and the original objective $F$ share the minimum value, and their minimizers have a closed-form correspondence.
Second, FBS and DRS are variable-metric gradient methods on the FBE and DRE, respectively. The envelope value $F_{\gamma}(x)$ and gradient $\nabla F_{\gamma}(x)$ are computable from the algorithm trajectories without explicitly forming the Moreau envelopes $f_{\gamma}$ and $g_{\gamma}$.

{\textbf{1) Solution equivalence.}} 
The following lemma first gives a sandwich bound between the envelope $F_{\gamma}$ and the original objective $F$.
It then shows that minimizing $F_\gamma$ is equivalent to minimizing $F$ and that the map $p$ recovers solutions of~\eqref{eq:composite-obj} from minimizers of the envelope.
\begin{lemma}[{\cite{patrinos2014forward,patrinos2014douglas}}]
  For any $x$ and $\gamma > 0$,
  \begin{equation}
    F (q (x)) + \tfrac{1 - \gamma L}{2 \gamma} \| p (x) - q (x) \|^2 \leq
    F_{\gamma} (x) \leq F (p (x)) - \tfrac{1}{2 \gamma} \| p (x) - q (x) \|^2.
    \label{eq:sandwich}
  \end{equation}
  In particular, if $0 < \gamma \leq \tfrac{1}{L}$ then $F (q (x)) \leq
  F_{\gamma} (x)$, and
  \begin{equation}\label{eq:env-eqiv}
   \min F = \min F_{\gamma}, \quad \arg \min F = p (\arg \min F_{\gamma}) .
  \end{equation}
\end{lemma}

{\textbf{2) Computable envelope oracle.}} 
Given the equivalence, we convert the original problem~\eqref{eq:composite-obj} into envelope minimization. In addition, we show that (1) each splitting step is
a variable-metric gradient step on its own envelope, and 
(2) the function value and gradient are computable. 
These properties enable minimization of the envelope by gradient methods.
Define the symmetric matrix
\begin{equation}
  D^{\mathrm{FB}}(x) \assign I - \gamma \nabla^2 f(x), \qquad
  D^{\mathrm{DR}}(x) \assign I - 2\gamma \nabla^2 f_{\gamma}(x).
\end{equation}
If $0 < \gamma < 1/L$, then $0 \prec D(x) \preceq I$,
and the splitting step is a variable-metric gradient step on the envelope,
\begin{equation}
  \begin{aligned}
    x^+ = T(x) &= x - \gamma D(x)^{-1} \nabla F_{\gamma}(x), \\
    \nabla F_{\gamma}(x) &= \tfrac{1}{\gamma} D(x)\,(x - T(x)).
  \end{aligned}
  \label{eq:vm-grad}
\end{equation}
The right-hand identity expresses $\nabla F_{\gamma}$ through the FBS/DRS iterates and a Hessian-vector product.

The function value is computable from the proximal quantities by the
Moreau-envelope identity \eqref{eq:env-compute}.
For the FBE, let $x^+ = T^{\mathrm{FB}}(x) =
\tmop{prox}_{\gamma g}(x-\gamma\nabla f(x))$. Substituting
\eqref{eq:env-compute} into \eqref{eq:fbe} gives
\begin{align}
  F^{\mathrm{FB}}_{\gamma}(x)
  &= f(x) - \tfrac{\gamma}{2}\|\nabla f(x)\|^2
     + g(x^+) + \tfrac{1}{2\gamma}
       \|x^+ - x + \gamma\nabla f(x)\|^2 \notag\\
  &= f(x) + g(x^+) + \langle \nabla f(x), x^+-x\rangle
     + \tfrac{1}{2\gamma}\|x^+-x\|^2 .
  \label{eq:fbe-computable-value}
\end{align}
Thus evaluating the FBE requires only $f(x)$, $\nabla f(x)$, the FBS update
$x^+$, and $g(x^+)$.
For the DRE, recall
\(
p(x) = \tmop{prox}_{\gamma f}(x),
  q(x) = \tmop{prox}_{\gamma g}(2p(x)-x),
  x^+ = T^{\mathrm{DR}}(x) = x+q(x)-p(x).
\)
Using $f_{\gamma}(x) = f(p(x))+\tfrac{1}{2\gamma}\|p(x)-x\|^2$ and
$\nabla f_{\gamma}(x) = \tfrac{1}{\gamma}(x-p(x))$ in \eqref{eq:dre}, we obtain the following formula, which uses only the two prox points $p(x),q(x)$ already computed by DRS, together with $f(p(x))$ and $g(q(x))$.
\begin{align}
  F^{\mathrm{DR}}_{\gamma}(x)
  &= f(p(x))+g(q(x))
     + \tfrac{1}{2\gamma}\|p(x)-x\|^2
     - \tfrac{1}{\gamma}\|x-p(x)\|^2 \notag\\
  &\quad
     + \tfrac{1}{2\gamma}\|q(x)-2p(x)+x\|^2 \notag\\
  &= f(p(x))+g(q(x))
     + \tfrac{1}{2\gamma}
       \left(\|x^+-p(x)\|^2-\|x-p(x)\|^2\right).
  \label{eq:dre-computable-value}
\end{align}

\subsection{Online scaled gradient methods}\label{subsec:osgm}

The preceding subsection shows that FBS and DRS are variable-metric gradient methods for minimizing the envelopes, which have computable function values and gradient oracles. This result enables us to use adaptive gradient methods on the envelopes if they are additionally smooth and convex.

This section briefly reviews how OSGM applies to gradient methods for smooth $\mu$-convex objectives, where $\mu$ may be negative.
Given an objective function $h$, the preconditioned gradient step is
\begin{equation}
    x^{k + 1 / 2} = x^k - P_k \nabla h (x^k) .
\end{equation}
OSGM views the problem of choosing $P_k$ as an online decision-making
problem: we define a feedback function to evaluate the performance of a candidate preconditioner, and OSGM optimizes the choice of $P_k$ by online
gradient descent on this feedback. 

Motivated by the descent lemma, OSGM uses the hypergradient feedback
$\ell_{x^k} (P) \assign \tfrac{h (x^k - P \nabla h (x^k)) - h (x^k)}{\| \nabla h (x^k) \|^2}$. Online learning requires the feedback to be convex in $P$ to attain sublinear regret
{\cite{zinkevich2003online,hazan2016introduction}} and thereby establish the global convergence of OSGM.
However, the feedback $\ell_{x^k}$ has the same convexity modulus as the objective $h$, which means the feedback is only weakly convex when $h$ is $(-\mu)$-weakly convex. In this case,
OSGM adds a quadratic regularizer to ensure convex feedback. The regularized hypergradient feedback is
\[ \ell_{x^k}^{\beta} (P) \assign \tfrac{h (x^k - P \nabla h (x^k)) - h (x^k)}{\| \nabla h
   (x^k) \|^2} + \tfrac{\beta}{2} \left\| P - \hat P \right\|_F^2 . \]
Here $\beta \geq 0$ is a regularization parameter and the fixed matrix $\hat
P$ is the \emph{anchor}. 
When $h$ is convex, we take $\beta = 0$. When $h$ is only weakly convex ($\mu < 0$),
taking $\beta \geq - \mu > 0$ makes the
regularized feedback $\ell_{x^k}^{\beta}$ convex {\cite{chu2025osgm2}}.
The anchor $\hat P$ is chosen to retain a baseline with a descent certificate. The
choice $\hat P = \tfrac{1}{L} I$ provides this property since the descent lemma gives $\ell_{x^k} (\hat P)
\leq - \tfrac{1}{2 L}$ whenever $\nabla h (x^k) \neq 0$.

After defining the feedback $\ell_{x^k}^{\beta}$, OSGM updates $P_k$ by online gradient descent. To ensure convergence and improve performance, OSGM uses a null step, a simple form of line search that guarantees a nonincreasing function value. Overall, the OSGM iteration is stated as follows.
\begin{align}
  \begin{split}\label{eq:osgm}
    x^{k + 1 / 2}  & = x^k - P_k \nabla h (x^k)\\
    P_{k + 1}  & = \tmop{proj}_{\mathcal{P}} [P_k - \eta \nabla \ell_{x^k}^{\beta} (P_k)]\\
    x^{k + 1} & = \arg \min \{ h (x) \mid x \in \{ x^k, x^{k + 1 / 2} \} \}.
  \end{split}
\end{align}

In particular, all $x$-iterates
remain in the initial sublevel set $\mathcal{L}_0$, so the
envelope regularity established in \Cref{sec:envelope-regularity} is
only required on this bounded region. With
this algorithm, OSGM converges as follows.

\begin{lemma}[{\cite{pmlr-v291-gao25a,chu2025osgm2}}]
\label{lem:osgm-global}
Suppose $h$ is $L$-smooth and $\mu$-convex, and $\mathcal{P}$ is a closed, bounded, convex set containing $\tfrac{1}{L} I$. Let $\bar{\ell}_K^{\star} \assign \min_{P \in \mathcal{P}} \tfrac{1}{K} \sum_{k=1}^K \ell_{x^k}^{\beta} (P)$ denote the best-in-hindsight average feedback, and let $R_K = \mathcal{O} (\sqrt{K})$ denote the regret of online gradient descent. Running OSGM for $K$ steps guarantees
\begin{enumerate}
  \item If $h$ is $\mu$-strongly convex, then
    \[
      h (x^K) - h^{\star} \leq (h (x^0) - h^{\star})
      \left( 1 - 2 \mu \max \left\{ - \bar{\ell}_K^{\star}
      - \tfrac{R_K}{K}, 0 \right\} \right)^K.
    \]
  \item If $h$ is convex, then
    \[
      h (x^K) - h^{\star} \leq {\cal{O}}\left(
      \tfrac{1 / K}{\max \{ - \bar{\ell}_K^{\star} - R_K / K, 0 \}}
      \right).
    \]
  \item If $h$ is $(-\mu)$-weakly convex, then
    \[
      \min\limits_{1 \leq k \leq K} \| \nabla h (x^k) \|^2
      \leq \tfrac{(h (x^0) - h^{\star}) / K}
      {\max \{ - \bar{\ell}_K^{\star} - R_K / K, 0 \}}.
    \]
\end{enumerate}
\end{lemma}

When $h$ is strongly convex and twice differentiable near $x^{\star}$ with $\nabla^2 h (x^{\star})^{-1} \in \mathcal{P}$, OSGM additionally attains local superlinear convergence: once the iterates enter a neighborhood of $x^{\star}$, the learned preconditioner tracks $\nabla^2 h (x^{\star})^{-1}$ and the $\mu$-strongly-convex rate in \Cref{lem:osgm-global} improves to superlinear {\cite{pmlr-v291-gao25a,gao2025osgm1}}.

\subsection{Challenges}\label{subsec:challenges}

Combining the above two components suggests a natural strategy: apply OSGM \eqref{eq:osgm} to the envelope by taking $h = F_{\gamma}$.
However, two gaps remain between this natural strategy and a working method.

\textbf{Smoothness and convexity of envelopes.}
Running OSGM on the envelope requires $F_{\gamma}$ to be smooth and
(weakly) convex by \Cref{lem:osgm-global}.
However, existing results (\Cref{lem:quad-regularity}) certify both
properties only when the smooth term $f$ is a convex quadratic.
\Cref{sec:envelope-regularity} closes this gap with a general
sufficient condition under which the FBE and DRE are smooth and (weakly) convex.

\textbf{Mismatch with the splitting geometry.}
The probe step, $x^{k + 1 / 2} = x^k - P_k \nabla F_{\gamma} (x^k)$, is poorly matched to the geometry of the
underlying splitting method.
By~\eqref{eq:vm-grad}, the base splitting step is a gradient step in the
iterate-dependent metric $\gamma D (x)^{- 1}$, so no fixed preconditioner
$P$ reproduces FBS or DRS.
Even the best fixed $P$ in hindsight therefore need not perform as well as
the base method.
\Cref{sec:osop} resolves this mismatch by building both the probe and the
feedback on the fixed-point residual $x - T (x)$. The probe becomes
$x - P (x - T (x))$, and $P = I$ recovers the base splitting method.

\section{Algorithms}
\label{sec:osop}

This section presents the online-scaled operator splitting (OSOP) framework, which provides a unified method for accelerating FBS, DRS, and ADMM.
\Cref{subsec:osop-alg} introduces online-scaled FBS and DRS. \Cref{sec:os-admm} then derives online-scaled ADMM (OS-ADMM)
using the equivalence between ADMM and DRS.

\subsection{Online-scaled FBS and DRS} \label{subsec:osop-alg}

\paragraph{Feedback design.}
As discussed in \Cref{subsec:challenges},
directly applying OSGM with a probing step $x - P \nabla F_{\gamma} (x)$ does not reproduce FBS/DRS with a fixed preconditioner $P$ when the underlying metric $\gamma D(x)^{-1}$ is iterate-dependent. 
We instead precondition the fixed-point residual
\( x^k - T (x^k)\)
and probe
\begin{equation}
  x^{k + 1 / 2} = x^k - P_k (x^k - T (x^k)) = x^k - \gamma P_k D (x^k)^{- 1}
  \nabla F_{\gamma} (x^k), \label{eq:osop-probe}
\end{equation}
where the second identity follows from \eqref{eq:vm-grad}. 
The probe is a
gradient step on $F_{\gamma}$ in the variable metric $\gamma P_k D (x^k)^{-
1}$, and $P_k = I$ reproduces the base splitting method exactly.

Accordingly, we introduce a new feedback
$\ell_x^{\beta}$ in \eqref{eq:osop-feedback} to measure the quality of
$P$ and let $\ell_x \assign \ell_x^0$:
\begin{equation}
  \ell_x^{\beta} (P) \assign \tfrac{F_{\gamma} (x - P (x - T (x))) -
  F_{\gamma} (x)}{\| x - T (x) \|^2} + \tfrac{\beta}{2} \| P - I \|_F^2.
  \label{eq:osop-feedback}
\end{equation}
The feedback $\ell_x^{\beta}$ adapts the OSGM hypergradient feedback in two ways.
\textbf{(1)} The residual direction $x - T (x)$ and its squared norm
$\| x - T (x) \|^2$ replace the gradient direction and its squared-norm
normalization, respectively. 
\textbf{(2)} We use the same quadratic
regularization term $\tfrac{\beta}{2} \| P - \hat P \|_F^2$ as OSGM but
center it at $\hat P = I$. The identity matrix is the natural anchor because it recovers the base splitting step and satisfies the
sufficient-decrease guarantee (\Cref{lem:hindsight-i}).

We choose the regularization parameter $\beta$ to ensure the feedback is convex.
Since $F_{\gamma}$ is $\mu_{\gamma}$-convex (\Cref{thm:envelope-regularity}), $\ell_x$ is also $\mu_{\gamma}$-convex.
When $\mu_{\gamma} \geq 0$, no regularization is needed.
When $\mu_{\gamma} < 0$, we take $\beta \geq - \mu_{\gamma}$.

The preconditioner is updated by projected online gradient descent on the
feedback, $P_{k + 1} = \Pi_{\mathcal{P}} [P_k - \eta \nabla \ell_{x^k}^{\beta}
(P_k)]$. Applying the chain rule and expanding the envelope gradient at the probe $\nabla F_{\gamma} (x^{k + 1 / 2})$ via~\eqref{eq:vm-grad}, we obtain
\begin{equation}
  \begin{split}
    \nabla \ell_{x^k}^{\beta} (P_k)
    &= - \tfrac{\nabla F_{\gamma} (x^{k + 1 / 2}) \, (x^k - T (x^k))^{\top}}{\|
       x^k - T (x^k) \|^2} + \beta (P_k - I) \\
    &= - \tfrac{D (x^{k + 1 / 2}) \, (x^{k + 1 / 2} - T (x^{k + 1 / 2})) \,
       (x^k - T (x^k))^{\top}}{\gamma \| x^k - T (x^k) \|^2} + \beta (P_k - I) .
  \end{split}
  \label{eq:osop-hgrad}
\end{equation}
The resulting method is summarized in
\Cref{alg:osop}. Each iteration performs a base splitting step, a probe step with preconditioner $P_k$, a preconditioner update with online gradient descent, and a null step to ensure nonincreasing envelope values.
Finally, after $K$ iterations, the output is $\hat{x}^K = q (x^K)$ with the map $q$ defined in \Cref{sec:components}.
By the sandwich bound~\eqref{eq:sandwich}, $F (\hat{x}^K) \leq
F_{\gamma} (x^K)$. This bound converts guarantees on
$F_{\gamma} (x^K)$ into guarantees on $F (\hat{x}^K)$.
We call the FBS and DRS instantiations of OSOP OS-FBS and OS-DRS.

\begin{algorithm}[H]
  \DontPrintSemicolon
  \caption{Online-scaled operator splitting (OSOP)}
  \label{alg:osop}
  \KwInit{$x^0$; stepsize $\gamma$; splitting operator $T \in \{
  T^{\mathrm{FB}}, T^{\mathrm{DR}} \}$; compact convex $\mathcal{P} \ni I$; $P_0
  = I$; online stepsize $\eta > 0$; regularization $\beta = \max \{ -
  \mu_{\gamma}, 0 \}$}
  \For{$k = 0, 1, \ldots, K - 1$}{
    \textbf{Splitting step:} evaluate $T (x^k)$; \textbf{stop} if $x^k =
    T (x^k)$\;
    \textbf{Probe step:} $x^{k + 1 / 2} = x^k - P_k (x^k - T (x^k))$\;
    \textbf{Preconditioner update:} $P_{k + 1} = \Pi_{\mathcal{P}}
    [P_k - \eta \nabla \ell_{x^k}^{\beta} (P_k)]$ with the hypergradient
    \eqref{eq:osop-hgrad}\;
    \textbf{Null step:} $x^{k + 1} = \arg \min \{ F_{\gamma} (x) \mid x \in
    \{ x^k, x^{k + 1 / 2} \} \}$\;
  }
  \KwOutput{$\hat{x}^K = T^{\mathrm{FB}} (x^K)$ for FBS;
  $\hat{x}^K = \tmop{prox}_{\gamma g} (2 \tmop{prox}_{\gamma f} (x^K) -
  x^K)$ for DRS}
\end{algorithm}

\begin{remark}[Per-iteration cost]\label{rem:osop-cost}
  With caching, each iteration of \Cref{alg:osop} performs only a single
  evaluation of the splitting step $T$.
  Specifically, $T (x^k)$ is
  cached from iteration $k - 1$ (the accepted probe $T (x^{k - 1 / 2})$, or
  the unchanged anchor after a rejection). Hence the only new $T$-evaluation is
  at the probe $T (x^{k + 1 / 2})$. 
  This quantity is also reused to compute the two remaining
  quantities: the null-step value $F_{\gamma} (x^{k + 1 / 2})$ and the
  envelope gradient $\nabla F_{\gamma} (x^{k + 1 / 2})$ that appears in the
  hypergradient~\eqref{eq:osop-hgrad}.
\end{remark}

\subsection{Online-scaled ADMM}\label{sec:os-admm}

\paragraph{ADMM.}
We next translate the OSOP method to ADMM through
the equivalence between ADMM and DRS
\cite{eckstein1992douglas,shipu,lessard}.
ADMM minimizes a separable objective in variables $x \in \mathbb{R}^{n_1}$ and $z\in \mathbb{R}^{n_2}$ subject to a linear constraint:
\begin{equation}\label{eq:admm}
  \min_{x,z} \quad \phi (x) + \psi (z) \quad \text{subject to} \quad A x + B z = b
\end{equation}
with problem data $A \in \mathbb{R}^{m \times n_1}$, $B \in \mathbb{R}^{m \times n_2}$, and $b \in \mathbb{R}^m$, as well as the objective functions $\phi: \mathbb{R}^{n_1} \to \mathbb{R} \cup \{ + \infty \}$ and $\psi: \mathbb{R}^{n_2} \to \mathbb{R} \cup \{ + \infty \}$.
We introduce a dual multiplier $\nu \in \mathbb{R}^m$ corresponding to the constraint. The
augmented Lagrangian with penalty parameter $\gamma > 0$ is
\[ \mathcal{L}_{\gamma} (x, z, \nu) = \phi (x) + \psi (z) + \langle \nu, A x + B z -
   b \rangle + \tfrac{\gamma}{2} \| A x + B z - b \|^2 . \]
The ADMM algorithm minimizes the augmented Lagrangian alternately over the
variables $x$ and $z$, followed by a dual ascent step:
\begin{align}
  x^{k + 1} \in & \arg \min_x  \left\{ \phi (x) + \langle \nu^k, A x \rangle +
  \tfrac{\gamma}{2} \| A x + B z^k - b \|^2 \right\} \nonumber\\
  z^{k + 1} \in & \arg \min_z  \left\{ \psi (z) + \langle \nu^k, B z \rangle +
  \tfrac{\gamma}{2} \| A x^{k + 1} + B z - b \|^2 \right\} \nonumber\\
  \nu^{k + 1} = & \nu^k + \gamma (A x^{k + 1} + B z^{k + 1} - b) . \nonumber
\end{align}
ADMM is equivalent to DRS applied to the Fenchel dual of~\eqref{eq:admm},
written as a minimization over the dual variable $\lambda$ \cite{eckstein1992douglas,shipu,lessard}:
\[ \min_{\lambda} \quad F^{\tmop{ADMM}} (\lambda) = \underbrace{\psi^{\ast} (-
   B^{\top} \lambda)}_{f (\lambda)} + \underbrace{\phi^{\ast} (- A^{\top}
   \lambda) + b^{\top} \lambda}_{g (\lambda)} . \]
This Fenchel dual problem is
an instance of the DRS problem~\eqref{eq:composite-obj}, and the
DRS variable $x$ corresponds to $\lambda$.

\paragraph{OS-ADMM.}\label{subsec:os-admm-alg}

To derive online-scaled ADMM (OS-ADMM), stated in \Cref{alg:os-admm}, we start by applying OS-DRS to the Fenchel dual problem $\min_{\lambda} f (\lambda) + g (\lambda)$.

Let $T$ denote the DRS operator. To run OSOP
on the dual envelope, four quantities must be computed at each iteration:
\textbf{1)} the fixed-point residual $\lambda^k - T
(\lambda^k)$, \textbf{2)} the probe iterate $\lambda^{k + 1 / 2} =
\lambda^k - P_k (\lambda^k - T (\lambda^k))$,
\textbf{3)} the
envelope gradient $\nabla F^{\tmop{ADMM}}_{\gamma}$ used in the hypergradient, and \textbf{4)} the
envelope value $F^{\tmop{ADMM}}_{\gamma}$ needed by the null
step. 
All four are computable from an ADMM pair
$(z, \nu)$ without explicitly recovering $\lambda$. 
The next proposition makes the translation explicit.

For a pair $(z, \nu)$, define $\lambda (z, \nu) \assign \nu - \gamma B z$ as the
corresponding DRS variable. Define the ADMM $x$-update and the associated primal residual by
\[
  \begin{aligned}
    x^+ (z, \nu) &\in \arg \min_x  \left\{ \phi (x) + \langle \nu, A x \rangle +
    \tfrac{\gamma}{2} \| A x + B z - b \|^2 \right\}, \\
    r_x (z, \nu) &\assign A x^+ (z, \nu) + B z - b .
  \end{aligned}
\]

\begin{proposition}[DRS--ADMM translation]
  \label{lem:drs-admm}
  {Let $\phi$ be proper, closed, and convex, and let $\psi \in
  C^2 (\mathbb{R}^{n_2})$ be $\mu_\psi$-strongly convex for some $\mu_\psi >
  0$. Suppose $g$ is proper and every $x$-subproblem invoked by
  \Cref{alg:os-admm} attains a minimizer.}
  Let $\gamma > 0$.
  Run \Cref{alg:os-admm} and associate each iterate $(z^k, \nu^k)$ and
  probe pair $(z^{k + 1 / 2}, \nu^{k + 1 / 2})$ with its corresponding DRS variable through $\lambda (z, \nu)$.
  Fix $k$, let
  $(\lambda, z, \nu)$ denote either $(\lambda^k, z^k, \nu^k)$ or $(\lambda^{k + 1 / 2}, z^{k + 1 / 2},
  \nu^{k + 1 / 2})$. Write $\lambda = \lambda (z, \nu)$ and $\lambda^k = \lambda (z^k,
  \nu^k)$. Then the following hold.
  \begin{enumerate}
    \item (Fixed-point residual) $\lambda^k - T (\lambda^k) = - \gamma \; r_x (z^k, \nu^k)$. 

    \item (Probe iterate) 
    $\lambda^k - P_k (\lambda^k - T (\lambda^k))=\lambda (z^{k + 1 / 2}, \nu^{k + 1 / 2})$.  \item (Envelope gradient)  $\nabla F_{\gamma}^{\tmop{ADMM}}
    (\lambda) = - (2 J_z - I) \; r_x(z, \nu)$, where $J_z \assign (I + \gamma B [\nabla^2
    \psi (z)]^{- 1} B^{\top})^{- 1}$. 

    \item (Envelope value)  $F_{\gamma}^{\tmop{ADMM}} (\lambda) = -
    \mathcal{L}_{\gamma} (x^+(z, \nu), z, \nu)$. \end{enumerate}
\end{proposition}

Part (1) follows from the standard DRS--ADMM correspondence
\cite{eckstein1992douglas,shipu,lessard}. Part (2) follows by applying this
correspondence to the probe pair in \Cref{alg:os-admm}. Part (3) follows from
the DRE gradient identity \cite[Eq.~(14)]{patrinos2014douglas} and its ADMM
specialization \cite[Sections~4.4--4.5]{giselsson2018envelope}. Part (4)
follows from the DRE--augmented-Lagrangian identity
\cite[Eq.~(3.5) and Theorem~5.5(vi)]{themelis2020douglas} and the sign relation
between the primal DRE and the dual ADMM envelope
\cite[Proposition~4.8]{giselsson2018envelope}. For completeness, we provide a full
proof in \Cref{app:drs-admm}.

\begin{remark}[Per-iteration cost]
  \label{rem:cost} With caching, each iteration performs one fresh probe
  $x$-solve. If the probe is accepted, this solve becomes the current
  $x$-update of the next iteration. If the probe is rejected, it is discarded while
  the current $x^{k +}$ is reused, so a rejected probe incurs one non-reusable
  $x$-solve. However, computing $\nabla F_{\gamma}^{\tmop{ADMM}}
    (\lambda)$ to update the hypergradient costs one linear solve with $I + \gamma B [\nabla^2 \psi
  (z^{k + 1 / 2})]^{- 1} B^{\top}$.
\end{remark}

\begin{algorithm}[htbp]
  \DontPrintSemicolon
  \caption{Online-scaled ADMM (OS-ADMM)}
  \label{alg:os-admm}
  \KwInit{$(z^0, \nu^0)$ with $- B^{\top} \nu^0 \in \partial \psi
  (z^0)$; penalty $\gamma > 0$; compact
  convex $\mathcal{P} \ni I$; $P_0 = I$; online stepsize $\eta > 0$;
  regularization $\beta = \max \{ - \mu^{\tmop{ADMM}}_{\gamma}, 0 \}$}
  \For{$k = 0, 1, \ldots, K - 1$}{
    \textbf{Current $x$-update:}
    $x^{k +} \in \arg \min_x  \left\{ \phi (x) + \langle \nu^k, A x \rangle +
    \tfrac{\gamma}{2} \| A x + B z^k - b \|^2 \right\}$,
    $r_x^k = A x^{k +} + B z^k - b$; \textbf{stop} if $r_x^k = 0$\;
    \textbf{Probe step:}
    $c_{P_k} = b - B z^k + P_k r_x^k$\;
    \qquad $\begin{aligned}
      z^{k + 1 / 2} \in & \arg \min_z  \left\{ \psi (z) + \langle \nu^k, B z
      \rangle + \tfrac{\gamma}{2} \| c_{P_k} + B z - b \|^2 \right\}\\
      \nu^{k + 1 / 2} = & \nu^k + \gamma (c_{P_k} + B z^{k + 1 / 2} -
      b)\\
      x^{k + 1 / 2} \in & \arg \min_x  \left\{ \phi (x) + \langle
      \nu^{k + 1 / 2}, A x \rangle + \tfrac{\gamma}{2} \| A x + B
      z^{k + 1 / 2} - b \|^2 \right\}\\
      r_x^{k + 1 / 2} = & A x^{k + 1 / 2} + B z^{k + 1 / 2} - b
    \end{aligned}$\;
    \textbf{Preconditioner update:}
    $P_{k + 1} = \Pi_{\mathcal{P}} \big[ P_k + \tfrac{\eta}{\gamma}
    \tfrac{G^{k + 1} (r_x^k)^{\top}}{\| r_x^k \|^2} - \eta \beta (P_k - I)
    \big]$ with
    $G^{k + 1} = (2 J_{z^{k + 1 / 2}} - I)  r_x^{k + 1 / 2}$,
    $J_{z^{k + 1 / 2}} = (I + \gamma B [\nabla^2 \psi (z^{k + 1 / 2})]^{- 1} B^{\top})^{- 1}$\;
    \textbf{Null step:}
    $(z^{k + 1}, \nu^{k + 1}) = \arg \max \{ \mathcal{L}_{\gamma} (x^+ (z,
    \nu), z, \nu) \mid (z, \nu) \in \{ (z^k, \nu^k), (z^{k + 1 / 2},
    \nu^{k + 1 / 2}) \} \}$\;
  }
  \KwOutput{$(x^{K +}, z^K, q^K)$ with $x^{K +} = x^+ (z^K, \nu^K)$, $q^K =
  \nu^K + \gamma r_x^K$}
\end{algorithm}

\section{Theory}\label{sec:analysis}

To analyze the convergence of OSOP and OS-ADMM, we first establish the smoothness and convexity of envelope functions under a sufficient condition that extends beyond the known quadratic case.
We then establish the global convergence of OSOP in the residual-PL, convex, and weakly convex regimes.
We also prove local superlinear convergence when the splitting iteration becomes locally affine, for example, after active-set identification for Lasso.
Finally, because the preceding convergence results are stated in terms of the function-value gap or the gradient norm of an envelope, which are not intuitive measures of ADMM convergence, we translate these metrics into bounds on the KKT residual.

\subsection{Smoothness and convexity of envelopes}\label{sec:envelope-regularity}

Previous work shows that the FBE and DRE are smooth and convex when the smooth term in~\eqref{eq:composite-obj} is a convex quadratic. We summarize this result in the following lemma.

\begin{lemma}[{\cite{patrinos2014forward,patrinos2014douglas}}]
  \label{lem:quad-regularity}
  Suppose that the smooth term $f$
is quadratic with Hessian $Q\succeq 0$. Let $\mu = \lambda_{\min}
  (Q) \geq 0$, $L = \lambda_{\max} (Q) > 0$, and $\gamma \in (0, \tfrac{1}{L}]$. Then $F_{\gamma}^{\mathrm{FB}}$ and $F_{\gamma}^{\mathrm{DR}}$ are smooth and
convex on $\mathbb{R}^n$ with constants
  \[ L_{\gamma}^{\mathrm{FB}} = \tfrac{2 (1 - \gamma \mu)}{\gamma}, \qquad
     \mu_{\gamma}^{\mathrm{FB}} = \min \{ \mu (1 - \gamma \mu), L (1 - \gamma L)
     \} , \]
  \[ L_{\gamma}^{\mathrm{DR}} = \tfrac{1 - \gamma \mu}{\gamma (1 + \gamma \mu)},
     \qquad \mu_{\gamma}^{\mathrm{DR}} = \min \left\{ \tfrac{\mu (1 - \gamma
     \mu)}{(1 + \gamma \mu)^2}, \tfrac{L (1 - \gamma L)}{(1 + \gamma L)^2}
     \right\} . \]
\end{lemma}

This section generalizes \Cref{lem:quad-regularity} from
quadratic smooth terms to a broader class of smooth functions.

Our results rest on \emph{either} of two assumptions.
\Cref{ass:smooth} imposes general regularity conditions on the
smooth term $f$ and the proximable component $g$.
\Cref{ass:quad} instead requires the smooth term to be quadratic.
The definitions of $L$-smoothness,
$\mu$-convexity, and $H$-Lipschitz Hessian appear in
\Cref{subsec:notation}.

\begin{assumption}\label{ass:smooth}
  There exist a convex region $\Omega_f \subseteq
  \mathbb{R}^n$ and constants $0 \le \mu \le L$ with $L > 0$ and $H, D \ge 0$ such
  that $\mu I \preceq \nabla^2 f (x) \preceq L I$ and
  $\| \nabla f (x) \| \le D$ for every $x \in \Omega_f$, and the function
  $f$ has an $H$-Lipschitz Hessian on $\Omega_f$.
  The proximable component $g$ is $G$-Lipschitz continuous on
  $\mathbb{R}^n$.
\end{assumption}

\begin{assumption}\label{ass:quad}
  The smooth term $f$ is a quadratic whose Hessian $Q$ satisfies $\mu I
  \preceq Q \preceq L I$ for some constants $0 \le \mu \le L$ with $L > 0$.
\end{assumption}

\begin{example}
  Smooth terms satisfying \Cref{ass:smooth} include the binary and multinomial logistic losses.
  Proximable components satisfying the Lipschitz condition in \Cref{ass:smooth} include the
  $\ell_1$-norm, the $\ell_2$-norm, and the total variation regularizer.
  In particular, $\ell_1$-regularized logistic
  regression satisfies \Cref{ass:smooth}.
  Examples satisfying \Cref{ass:quad} include box-constrained quadratic
  programs and Lasso.
\end{example}

\begin{remark}\label{rem:localization}
  In \Cref{ass:smooth}, strong convexity forces
  $\Omega_f$ to be bounded. If $\mu > 0$, then by
  $\| \nabla f (x) - \nabla f (y) \| \ge \mu \| x - y \|$ and the
  gradient bound $\| \nabla f \| \le D$ on $\Omega_f$, we have
  $\tmop{diam} ({\Omega_f}) \le 2 D / \mu$.

  When $\Omega_f$ is bounded, the smooth-term conditions in \Cref{ass:smooth}
  are mild: the gradient bound holds with
  $D = \sup_{x \in \Omega_f} \| \nabla f (x) \|$,
  and the Lipschitz-Hessian condition holds for any
  $f \in C^3$. Bounded regions already suffice for our analysis.
  The algorithm
  we develop in \Cref{sec:osop} is a descent method on $F_{\gamma}$, so
  its iterates remain in the bounded initial sublevel set
  $\mathcal{L}_0 = \{ x : F_{\gamma} (x) \le F_{\gamma} (x^0) \}$. 
  Then it suffices to verify \Cref{ass:smooth} with bounded regions
  tied to the iterates and probes.
\end{remark}

\subsubsection{FBE and DRE}
The next theorem gives explicit smoothness and convexity constants for the FBE and DRE under either of two sets of assumptions.
Under \Cref{ass:smooth}, the result is localized to a region \(\Omega_f\). Under \Cref{ass:quad}, the smooth term is quadratic and the result is global.

\begin{theorem}
  \label{thm:envelope-regularity}
  Suppose $\gamma \in (0, \tfrac{1}{L}]$, $g$ is proper, closed, and convex, and one of the following holds:
  \begin{itemize}
    \item \Cref{ass:smooth} holds and $\Omega$ is convex, with $\Omega
    \subseteq \Omega_f$ in the FB case and $\tmop{prox}_{\gamma f}
    (\Omega) \subseteq \Omega_f$ in the DR case;

    \item \Cref{ass:quad} holds.
    Then $\Omega = \mathbb{R}^n$, and the constants below apply with $H
    (D + G)$ replaced by $0$.
  \end{itemize}
  Then the corresponding envelope $F_{\gamma}$ is $L_{\gamma}$-smooth and
  $\mu_{\gamma}$-convex on $\Omega$, where
  \[ L_{\gamma}^{\mathrm{FB}} = \tfrac{1 - \gamma \mu}{\gamma} + \gamma H
     (D + G), \qquad
     \mu_{\gamma}^{\mathrm{FB}} = \min \{\mu
     (1 - \gamma \mu), L (1 - \gamma L)\} - \gamma H (D + G) , \]
  \[ L_{\gamma}^{\mathrm{DR}} = \tfrac{1 - \gamma \mu}{\gamma (1 + \gamma
     \mu)^2} + \tfrac{2 \gamma H (D + G)}{(1 + \gamma \mu)^3}, \qquad
     \mu_{\gamma}^{\mathrm{DR}} = \min \left\{ \tfrac{\mu (1 - \gamma
     \mu)}{(1 + \gamma \mu)^2}, \tfrac{L (1 - \gamma L)}{(1 +
     \gamma L)^2} \right\} - \tfrac{2 \gamma H (D + G)}{(1 + \gamma
     \mu)^3} . \]
\end{theorem}

We prove \Cref{thm:envelope-regularity} by first identifying a general template for both the FBE and DRE and then deriving the smoothness and convexity of this template.
Both the FBE and DRE have the following form for a function $s$ and a stepsize
$\alpha > 0$:
\begin{equation}
  E (x) = s (x) - \tfrac{\alpha}{2} \| \nabla s (x) \|^2 + g_{\gamma}
  (M (x)), \quad M (x) \assign x - \alpha \nabla s (x) .
  \label{eq:env-template}
\end{equation}
For the FBE, take $(s, \alpha) = (f, \gamma)$. For the DRE, take
$(s, \alpha) = (f_{\gamma}, 2 \gamma)$.
The following lemma establishes the smoothness and convexity of the template $E$, which together imply \Cref{thm:envelope-regularity}.

\begin{lemma}
  \label{lem:env-template}
  Let $\Omega \subseteq \mathbb{R}^n$ be convex, and let
  $\gamma > 0$, $0 \le \mu \le L$ with $L > 0$, and $\alpha \in (0, 1 / L]$.
  Suppose $s$ is twice differentiable on $\mathbb{R}^n$ and $g$ is proper,
  closed, and convex. Suppose either that the conditions in
  \Cref{ass:smooth} hold with $f$ and $\Omega_f$ replaced by $s$ and
  $\Omega$, respectively, or that the conditions in \Cref{ass:quad} hold
  with $f$ replaced by $s$. 
  Then $E$ in~\eqref{eq:env-template} is $L_E$-smooth and $\mu_E$-convex on
  $\Omega$ with
  \[
    \begin{aligned}
      L_E &= \sup_{\lambda \in [\mu, L]} [\lambda (1 - \alpha \lambda) +
      \tfrac{1}{\gamma} (1 - \alpha \lambda)^2] + \alpha H (D + G), \\
      \mu_E &= \inf_{\lambda \in [\mu, L]} \lambda (1 - \alpha
      \lambda) - \alpha H (D + G) .
    \end{aligned}
  \]
\end{lemma}

We first provide a simple proof of \Cref{lem:env-template} under stronger
differentiability assumptions and defer the complete proof to \Cref{app:proof-env-template}. 
In this paragraph, we additionally assume that the function $s$ is three times continuously
differentiable, and $g_{\gamma}$ is twice continuously
differentiable.

For $z \in \Omega$, write $S (z) \assign \nabla^2 s (z)$ and $J (z) \assign I - \alpha S (z)$.
Differentiating the function $E$ in~\eqref{eq:env-template} gives
\[
  \nabla E (z) = J \bigl(\nabla s (z) + \nabla g_{\gamma} (M (z))\bigr).
\]
For a direction $v \in \mathbb{R}^n$, differentiate $\nabla E (z)$ once
more along $v$ and write $B (z) \assign
\nabla^2 g_{\gamma} (M (z))$:
\[
  \nabla^2 E (z) v
  = (JS + JBJ) v
  - \alpha \nabla^3 s (z) [v]
  \bigl(\nabla s (z) + \nabla g_{\gamma} (M (z))\bigr).
\]
We first bound the principal term. Convexity and $\tfrac{1}{\gamma}$-smoothness of $g_{\gamma}$ give
\(
  0 \preceq B \preceq \tfrac{1}{\gamma} I
\). Moreover, $J \succeq 0$, and the matrices $S$ and $J$ commute. Therefore,
\[
  \left(\inf_{\lambda \in [\mu,L]} \lambda (1 - \alpha \lambda)\right) I
  \preceq JS + JBJ
  \preceq \left(\sup_{\lambda \in [\mu,L]}
  \left[\lambda (1 - \alpha \lambda) + \tfrac{1}{\gamma}
  (1 - \alpha \lambda)^2\right]\right) I.
\]
To bound the second term, fix $x, y \in \Omega$, write $d
\assign y - x$, and set $x_t \assign x + t d$ for $t \in (0,1)$. For
sufficiently small $|h|$, convexity gives $x_t + h d \in \Omega$, so the
Lipschitz-Hessian condition gives
$\| \nabla^2 s (x_t + h d) - \nabla^2 s (x_t) \| \le H |h| \| d \|$.
Dividing by $|h|$ and letting $h \to 0$ yields
$\| \nabla^3 s (x_t) [d] \| \le H \| d \|$. 
Since $g$ is $G$-Lipschitz
continuous, proximal optimality gives
$\| \nabla g_{\gamma} (M (x_t)) \| \le G$. Therefore, for every $x_t$, the second term satisfies
\[
  \left| \alpha \left\langle d, \nabla^3 s (x_t) [d]
  \bigl(\nabla s (x_t) + \nabla g_{\gamma} (M (x_t))\bigr) \right\rangle \right|
  \le \alpha H (D + G) \| d \|^2.
\]

Combining the two bounds and integrating along
the line segment from $x$ to $y$
gives the stated $L_E$-smoothness and $\mu_E$-convexity constants. When $s$
is quadratic, the same argument holds with $H=0$.

We then instantiate \Cref{lem:env-template} for the FBE and DRE to prove
\Cref{thm:envelope-regularity}.

\begin{proof}[Proof of \Cref{thm:envelope-regularity}]
For the FBE instantiation, the inclusion $\Omega \subseteq \Omega_f$
verifies the first case of \Cref{lem:env-template} under
\Cref{ass:smooth}. Under \Cref{ass:quad}, the second case applies on
$\mathbb{R}^n$. The scalar expression inside the supremum defining $L_E$ becomes
\[
  \lambda (1 - \gamma \lambda) + \tfrac{1}{\gamma}
  (1 - \gamma \lambda)^2 = \tfrac{1 - \gamma \lambda}{\gamma},
\]
which is decreasing on $[\mu,L]$. The expression inside the infimum defining $\mu_E$,
$\lambda (1-\gamma\lambda)$, is concave on $[0,1/\gamma]$, so its minimum
on $[\mu,L]$ is attained at an endpoint. Substitution in
\Cref{lem:env-template} gives the stated FBE constants.

The DRE instantiation also uses the following localized Moreau-envelope
regularity result. The proof is in \Cref{app:proof-s-env}.

\begin{lemma}
  \label{lem:s-env}
  Let $\gamma > 0$, and suppose the smooth term $f$ satisfies the conditions
  in \Cref{ass:smooth}. Then for every
  convex $\Omega \subseteq \mathbb{R}^n$ with
  $\tmop{prox}_{\gamma f} (\Omega) \subseteq \Omega_f$, the Moreau envelope
  $f_{\gamma}$ is twice differentiable on $\mathbb{R}^n$ and satisfies the
  same bounds on $\Omega$ with constants
  \[
    \tilde{\mu} = \tfrac{\mu}{1 + \gamma \mu}, \qquad
    \tilde{L} = \tfrac{L}{1 + \gamma L}, \qquad
    \tilde{H} = \tfrac{H}{(1 + \gamma \mu)^3}, \qquad
    \tilde{D} = D .
  \]
  If instead the smooth term $f$ satisfies \Cref{ass:quad}, then
  $f_{\gamma}$ is quadratic and the same Hessian bounds hold on
  $\mathbb{R}^n$, with $\tilde{H}=0$.
\end{lemma}

For the DRE instantiation, \Cref{lem:s-env} gives the constants
$(\tilde{\mu},\tilde{L},\tilde{H},\tilde{D})$ on $\Omega$ under
\Cref{ass:smooth}. Under \Cref{ass:quad}, the function $f_{\gamma}$ is
quadratic with Hessian
$Q (I + \gamma Q)^{-1}$, whose spectrum lies in
$[\tilde{\mu},\tilde{L}]$. In both cases,
\(
  2 \gamma \tilde{L} = \tfrac{2 \gamma L}{1 + \gamma L} \le 1,
\)
so \Cref{lem:env-template} applies. With
$\tilde{\lambda} = \lambda/(1+\gamma\lambda)$, the scalar expressions inside
the supremum and infimum defining $L_E$ and $\mu_E$, respectively, become
\[
  \tfrac{1 - \gamma \lambda}{\gamma (1 + \gamma \lambda)^2}, \qquad
  \tfrac{\lambda (1 - \gamma \lambda)}{(1 + \gamma \lambda)^2}.
\]
The first expression is decreasing on $[0,1/\gamma]$. The second has only
an interior maximum, so its minimum on $[\mu,L]$ is attained at an
endpoint. Finally, under \Cref{ass:smooth},
\(
  2 \gamma \tilde{H} (\tilde{D}+G)
  = \tfrac{2 \gamma H (D+G)}{(1+\gamma\mu)^3},
\)
which gives the stated DRE constants. 

The same argument holds under \Cref{ass:quad} with $H=0$.
\end{proof}

The following proposition shows that the quadratic-case constants of \Cref{thm:envelope-regularity} are tight. The proof is in \Cref{app:proof-envelope-tightness}.

\begin{proposition}[Tightness in the quadratic case]
  \label{prop:quad}
  Suppose \Cref{ass:quad} holds, $\mu$ and $L$ are eigenvalues of $Q$,
  and $\gamma \in (0, \tfrac{1}{L}]$. Then the quadratic-case constants
  of \Cref{thm:envelope-regularity} are tight: $g = 0$ attains the
  convexity constants, and $g = \delta_{\{0\}}$ attains the smoothness
  constants.
\end{proposition}

In the quadratic case (\Cref{ass:quad}), the envelopes are convex: $\mu_{\gamma}
\geq 0$ for every $\gamma \in (0, \tfrac{1}{L}]$. 
By contrast, under
\Cref{ass:smooth}, the envelopes are convex on
$\Omega$ only for small enough $\gamma$.

\subsubsection{ADMM envelope}\label{subsec:admm-envelope}

Define the ADMM envelope $F^{\tmop{ADMM}}_{\gamma}$ as the DRE~\eqref{eq:dre} of the dual objective $F^{\tmop{ADMM}}$.
According to \Cref{thm:envelope-regularity}, the ADMM envelope is smooth and convex if
$f$ and $g$ satisfy either \Cref{ass:smooth} or \Cref{ass:quad}.
\Cref{ass:admm,ass:admm-quad} translate the two assumptions into
conditions on $\phi$ and $\psi$.

\begin{assumption}\label{ass:admm}
  $\psi$ is proper, closed, $L_\psi$-smooth, $\mu_\psi$-strongly convex
  ($\mu_\psi > 0$), and twice differentiable with an $H_\psi$-Lipschitz
  Hessian. There exists a convex region $\Lambda$ and a
  constant $D$ such that $\| B \nabla \psi^{\ast} (- B^{\top}
  \lambda) \| \leq D$ for all $\lambda \in \Lambda$. The function
  $\phi$ is proper, closed, and convex, and $\tmop{dom} \phi \subseteq \{
  x : \| x \| \leq R_\phi \}$.
\end{assumption}

\begin{assumption}\label{ass:admm-quad}
  $\psi$ is a quadratic $\psi (z) = \tfrac{1}{2} z^{\top} Q z + c^{\top} z$ with
  $\mu_\psi I \preceq Q \preceq L_\psi I$ and $\mu_\psi > 0$. The function $\phi$ is proper,
  closed, and convex. The $x$-subproblem $\min_x \, \{ \phi (x)
  + \langle \nu, A x \rangle + \tfrac{\gamma}{2} \| A x + B z - b \|^2
  \}$ attains its minimum for every $(z, \nu)$.
\end{assumption}

\begin{example}\label{ex:qp-split}
  The standard-form quadratic programming (QP) $\min_x \, \{ \tfrac{1}{2} x^{\top} Q x + q^{\top} x
  : A x = b, \, x \geq 0 \}$ with $Q \succ 0$ satisfies \Cref{ass:admm-quad}
  when cast as a splitting of the form
  \[ \min_{x, z} \quad \underbrace{\delta_{\mathbb{R}_+^n} (x)}_{\phi (x)}
     + \underbrace{\tfrac{1}{2} z^{\top} Q z + q^{\top} z}_{\psi (z)}
     \quad \text{subject to} \quad \left[ \begin{array}{c} 0 \\ I \end{array} \right] x
     + \left[ \begin{array}{c} A \\ - I \end{array} \right] z
     = \left[ \begin{array}{c} b \\ 0 \end{array} \right] . \]
\end{example}

The following lemma translates the conditions in \Cref{ass:admm,ass:admm-quad} for the ADMM problem into those in \Cref{ass:smooth,ass:quad} for the Fenchel dual problem.

\begin{lemma}
  \label{lem:admm-regularity} Under \Cref{ass:admm}, the dual pair $(f,
  g)$ satisfies \Cref{ass:smooth} with $\Omega_f = \Lambda$
  and $D$ as in \Cref{ass:admm}, where
  \[ \begin{aligned}
       \mu &= \tfrac{\lambda_{\min} (B B^{\top})}{L_\psi}, &
       L &= \tfrac{\| B \|^2}{\mu_\psi},\\
       H &= \tfrac{H_\psi \| B \|^3}{\mu_\psi^3}, &
       G &= \| A \| R_\phi + \| b \| .
     \end{aligned} \]
  Under \Cref{ass:admm-quad}, the dual pair $(f,
  g)$ satisfies \Cref{ass:quad} with the same
  $\mu$ and $L$.
\end{lemma}

The proof is in \Cref{app:proof-admm-regularity}.

Under the translation, we apply \Cref{thm:envelope-regularity} in the DR case with the constants defined in \Cref{lem:admm-regularity} and obtain the smoothness and convexity constants of the ADMM envelope, $L^{\tmop{ADMM}}_{\gamma}$ and $\mu^{\tmop{ADMM}}_{\gamma}$.

\subsection{Global convergence analysis}\label{subsec:osop-global}

The analysis needs
the envelope to be smooth and (weakly) convex at every point where $F_{\gamma}$ is evaluated: the iterates $x^k$ and the probes $x^{k + 1 / 2}$. 
Because of the null step, the iterates $x^k$ never leave the
initial sublevel set
$\mathcal{L}_0 \assign \{ x : F_{\gamma} (x) \leq F_{\gamma} (x^0) \}$.
The probes $x - P (x - T (x))$ then range over all $x \in \mathcal{L}_0$ and $P \in \mathcal{P}$. Hence both iterates and probes lie in the following region:
\begin{equation}
  \Omega \assign \tmop{conv} \left( \mathcal{L}_0 \cup \{ x - P (x - T
  (x)) : x \in \mathcal{L}_0, P \in \mathcal{P} \} \right) .
  \label{eq:analysis-region}
\end{equation}
Our results use \Cref{thm:envelope-regularity}
on this region $\Omega$~\eqref{eq:analysis-region}, with the stepsize restricted to $\gamma \in (0, 1 / L)$. We also make the following two assumptions.

\begin{assumption}[Candidate set]\label{ass:candidate}
  The candidate set $\mathcal{P}$ is compact and convex with $I \in
  \mathcal{P}$, and $\Delta_{\mathcal{P}} \assign \max_{P, P' \in
  \mathcal{P}} \| P - P' \|_F$ denotes its diameter.
\end{assumption}

\begin{assumption}[Bounded sublevel sets]\label{ass:level}
  The objective $F$ has bounded sublevel sets.
\end{assumption}

We first show that the region $\Omega$ is
compact and convex under \Cref{ass:candidate,ass:level}. 
\textbf{(1)} The sublevel set $\mathcal{L}_0$ is closed because $F_{\gamma}$ is continuous. 
The sandwich bound~\eqref{eq:sandwich} shows that \Cref{ass:level} implies that $F_{\gamma}$ has bounded sublevel sets. Hence the sublevel set $\mathcal{L}_0$ is compact.
\textbf{(2)} The probe set is the image of
$\mathcal{L}_0 \times \mathcal{P}$ under the continuous map $(x, P)
\mapsto x - P (x - T (x))$, so it is compact as well. 
Since $\Omega$ is
the convex hull of $\mathcal{L}_0$ and the probe set, it is compact and convex. 

In particular,
both $\Omega$ and its proximal image $\tmop{prox}_{\gamma f} (\Omega)$
are bounded. \Cref{thm:envelope-regularity}
requires $\Omega \subseteq \Omega_f$ in the FB case and
$\tmop{prox}_{\gamma f} (\Omega) \subseteq \Omega_f$ in the DR case.
Both relations can be met by choosing a bounded $\Omega_f$, on which \Cref{ass:smooth} is mild (\Cref{rem:localization}). Therefore,
the envelope $F_{\gamma}$ is $L_{\gamma}$-smooth and $\mu_{\gamma}$-convex on
$\Omega$.

By~\eqref{eq:vm-grad} and $D (x) \preceq I$, the envelope gradient is
controlled by the fixed-point residual:
\begin{equation}
  \gamma \| \nabla F_{\gamma} (x) \| \leq \| x - T (x) \|, \quad \forall x
  . \label{eq:res-grad}
\end{equation}

\paragraph{Feedback reduction.}

We now reduce the global convergence analysis of \Cref{alg:osop} to an online
learning problem. As an example, we demonstrate how to use this reduction to prove the global linear convergence rate.

By the definition of the
feedback~\eqref{eq:osop-feedback} and the null step, each iteration satisfies the descent
inequality
\begin{equation}
  F_{\gamma} (x^{k + 1}) - F_{\gamma} (x^k) \leq \min \{
  \ell_{x^k}^{\beta} (P_k), 0 \} \| x^k - T (x^k) \|^2 .
  \label{eq:null-descent}
\end{equation}
The descent inequality~\eqref{eq:null-descent} controls the envelope
decrease in units of the squared residual $\| x^k - T (x^k) \|^2$. 
To derive linear convergence of the function-value gap $F_{\gamma} (x^k) - F^{\star}$, we need a lower bound on the squared residual in terms of the gap.
Let $F^{\star} \assign \min F_{\gamma} = \min F$ and $\mathcal{X}^{\star}
\assign \arg \min F_{\gamma}$. We define
$\tilde{\mu}_{\gamma}$ as the \emph{fixed-point residual
Polyak--{\L}ojasiewicz (PL) constant} of $F_{\gamma}$ on $\mathcal{L}_0$,
namely, the largest constant satisfying
\begin{equation}
  2 \tilde{\mu}_{\gamma} (F_{\gamma} (x) - F^{\star}) \leq \| x - T (x)
  \|^2 \qquad \forall x \in \mathcal{L}_0 . \label{eq:residual-pl}
\end{equation}
For brevity, we refer to $\tilde{\mu}_{\gamma}$ as the residual PL constant because \(\nabla F_\gamma(x)\) in the usual PL inequality is replaced by the fixed-point residual \(x-T(x)\).

In \Cref{prop:os-lower}, we first give a global lower bound on the residual PL constant $\tilde{\mu}_{\gamma}$ in terms of the strong convexity constant $\mu_{\gamma}$ of the envelope.
We then give sharper lower bounds in two settings: when the smooth term is quadratic and when an error bound holds. The proof is in \Cref{app:proof-os-lower}.

\begin{proposition}[Lower bounds on the residual PL constant]
  \label{prop:os-lower} In both the FB and DR cases:
  \begin{enumerate}
    \item (strong convexity) Suppose $0 < \gamma < \tfrac{1}{L}$ and the
      envelope $F_{\gamma}$ is $\mu_{\gamma}$-strongly convex on
      $\mathcal{L}_0$. Then $\tilde{\mu}_{\gamma} \geq \gamma^2
      \mu_{\gamma}$.
    \item (quadratic smooth term) Suppose \Cref{ass:quad} holds with
      $\mu = \lambda_{\min} (Q) > 0$ and $L = \lambda_{\max} (Q)$. If
      $0 < \gamma < \tfrac{1}{L}$, then $\tilde{\mu}_{\gamma} \geq
      \tfrac{\gamma^2 \mu}{1 - \gamma \mu}$.
    \item (error bound) Suppose $0 < \gamma < \tfrac{1}{L}$, the envelope
      $F_{\gamma}$ is $L_{\gamma}$-smooth on
      $\Omega$~\eqref{eq:analysis-region}, and the error bound holds:
      $\tmop{dist} (x, \mathcal{X}^{\star}) \leq \kappa \| x - T (x) \|$
      for all $x \in \mathcal{L}_0$. Then $\tilde{\mu}_{\gamma} \geq
      \tfrac{1}{L_{\gamma} \kappa^2}$.
  \end{enumerate}
\end{proposition}

In the quadratic case, both convexity moduli
$\mu_{\gamma}^{\mathrm{FB}}$ and $\mu_{\gamma}^{\mathrm{DR}}$ tend to zero as
$\gamma \uparrow \tfrac{1}{L}$, so the first bound vanishes while the second
bound $\tfrac{\gamma^2 \mu}{1 - \gamma \mu}$ stays positive.
For the third bound, a finite $\kappa$ does not require the
envelope to be strongly convex. The simplest instance is
rank-deficient least squares, $f (x) = \tfrac{1}{2} \| A x - b \|^2$
with $g = 0$ and $Q \assign A^{\top} A$ singular.

Substituting~\eqref{eq:residual-pl} into~\eqref{eq:null-descent}
shows that each iteration contracts the gap. Whenever the current gap is
positive, nonnegativity of the next gap forces the contraction factor to be
nonnegative. If the current gap is zero, the null step keeps it zero.
Therefore, multiplying the $K$ contraction factors and applying the AM--GM
inequality, followed by $\min \{ \ell_{x^k}^{\beta} (P_k), 0 \} \leq
\ell_{x^k}^{\beta} (P_k)$, bounds the final gap by the average feedback:
\begin{equation}
\begin{split}
  F_{\gamma} (x^K) - F^{\star} & \leq (F_{\gamma} (x^0) - F^{\star})
  \textstyle \prod_{k = 0}^{K - 1} \left( 1 + 2 \tilde{\mu}_{\gamma} \min \{
  \ell_{x^k}^{\beta} (P_k), 0 \} \right) \\
  & \leq (F_{\gamma} (x^0) - F^{\star}) \left( 1 + \tfrac{2
  \tilde{\mu}_{\gamma}}{K} \textstyle \sum_{k = 0}^{K - 1} \ell_{x^k}^{\beta}
  (P_k) \right)^K.
\end{split}
\label{eq:feedback-decomp}
\end{equation}
The regret is
\( R_K \assign
   \max_{{P} \in \mathcal{P}} \sum_{k = 0}^{K - 1} \left[
   \ell_{x^k}^{\beta} (P_k) - \ell_{x^k}^{\beta} ({P}) \right]
\). We define the average feedback at the best fixed preconditioner as the
\emph{hindsight feedback}:
\[ \bar{\ell}_K^{\star} \assign \min_{P \in \mathcal{P}} \tfrac{1}{K}
   \textstyle \sum_{k = 0}^{K - 1} \ell_{x^k}^{\beta} (P) . \]
We can then rewrite~\eqref{eq:feedback-decomp} as
\begin{equation}\label{eq:feedback-decomp2}
  \tfrac{F_{\gamma} (x^K) - F^{\star}}{F_{\gamma} (x^0) - F^{\star}} \leq
 \Big( 1 + 2 \tilde{\mu}_{\gamma}
  \big( \underbrace{ \min_{P \in \mathcal{P}} \tfrac{1}{K}
   \textstyle \sum_{k = 0}^{K - 1} \ell_{x^k}^{\beta} (P)}_{\text{(1) $\bar{\ell}_K^{\star}$, hindsight feedback}} +
  \underbrace{\tfrac{R_K}{K}}_{\text{(2) average regret}} \big)
  \Big)^K.
\end{equation}

\paragraph{Analysis.}

We must bound both terms to show that \eqref{eq:feedback-decomp2} provides a reduction.
These bounds are also used to prove the convergence of OSOP even when the residual PL constant is non-positive.
\textbf{(1)} The hindsight feedback $\bar{\ell}_K^{\star}$ is negative and bounded away
from zero when the base splitting method (\Cref{lem:hindsight-i}) satisfies a sufficient decrease condition.
\textbf{(2)} The average regret $R_K / K$ vanishes as $K$ grows when the online learning method is chosen appropriately depending on properties of the feedback. 
For example, projected online gradient descent guarantees $R_K = \mathcal{O} (\sqrt{K})$ when the
feedback is convex and Lipschitz (\Cref{lem:regret}).

To bound the hindsight feedback $\bar{\ell}_K^{\star}$, 
it suffices to bound the average feedback at preconditioner $P = I$.
The following lemma proves a uniform negative upper bound on $\ell_x (I)$ for all $x$, using the fact that $\ell_x (I)$ measures
the envelope decrease of the base splitting step per unit of squared residual:
$\tfrac{F_{\gamma}^{\mathrm{FB}} (T (x)) - F_{\gamma}^{\mathrm{FB}} (x)}{\| x - T (x) \|^2}$.

\begin{lemma}[Sufficient decrease]
  \label{lem:hindsight-i}
  Suppose $\gamma \in (0, 1 / L)$. For all $x$, $F^{\mathrm{FB}}_{\gamma}$ and $F_{\gamma}^{\mathrm{DR}}$ satisfy the sufficient decrease condition:
  \[
    \begin{aligned}
      F_{\gamma}^{\mathrm{FB}} (T (x)) - F_{\gamma}^{\mathrm{FB}} (x)
      &\leq - \tfrac{1 - \gamma L}{2 \gamma} \| x - T (x) \|^2, \\
      F_{\gamma}^{\mathrm{DR}} (T (x)) - F_{\gamma}^{\mathrm{DR}} (x)
      &\leq - \tfrac{\min \{ 1, (1 - \gamma L) (1 + 2 \gamma L) \}}
      {2 \gamma (1 + \gamma L)^2} \| x - T (x) \|^2 .
    \end{aligned}
  \]
  As a result, the feedback $\ell_x^{\beta} (I)$ is negative and bounded away from zero:
  \[
    \ell_x^{\beta} (I) = \ell_x (I) \leq - c_0 < 0,
    \quad \text{where} \quad
    \begin{aligned}
      c_0^{\mathrm{FB}} &\assign \tfrac{1 - \gamma L}{2 \gamma}, \\
      c_0^{\mathrm{DR}} &\assign
      \tfrac{\min \{ 1, (1 - \gamma L) (1 + 2 \gamma L) \}}
      {2 \gamma (1 + \gamma L)^2}.
    \end{aligned}
  \]
  For $\gamma \leq \tfrac{1}{2 L}$, the DR constant simplifies to
  $c_0^{\mathrm{DR}} = \tfrac{1}{2 \gamma (1 + \gamma L)^2}$.
\end{lemma}

\begin{proof}
  The FB inequality follows from two applications of the sandwich bound~\eqref{eq:sandwich}, in which $p (x) = x$ and $q (x) = T (x)$: the
  upper bound at $T (x)$ gives $F_{\gamma}^{\mathrm{FB}} (T (x)) \leq F (T
  (x))$, and the lower bound at $x$ gives $F (T (x)) \leq
  F_{\gamma}^{\mathrm{FB}} (x) - \tfrac{1 - \gamma L}{2 \gamma} \| x - T
  (x) \|^2$.

  The DR inequality is due to {\cite[Theorem~4.1]{themelis2020douglas}}, and
  we restate the argument in our setting and notation.
  Write $x^+ = T (x)$ and abbreviate the prox-point and gradient
  differences $a \assign p (x^+) - p (x)$, $e \assign \nabla f (p (x^+))
  - \nabla f (p (x))$. Substituting the Moreau identities of
  \Cref{subsec:moreau}, $\gamma \nabla f_{\gamma} (x) = x - p (x) =
  \gamma \nabla f (p (x))$ and~\eqref{eq:env-compute}, into~\eqref{eq:dre} expresses the DRE as a partial minimization,
  \[
    \begin{aligned}
      F_{\gamma}^{\mathrm{DR}} (x) &= \min_w Q (w ; p (x)), \\
      Q (w ; u) &\assign f (u) + \langle \nabla f (u), w - u \rangle
      + g (w) + \tfrac{1}{2 \gamma} \| w - u \|^2,
    \end{aligned}
  \]
  with minimizer $w = \tmop{prox}_{\gamma g} (2 p (x) - x) = q (x)$.
  Hence $F_{\gamma}^{\mathrm{DR}} (x) = Q (q (x) ; p (x))$, while the
  suboptimal choice $w = q (x)$ gives $F_{\gamma}^{\mathrm{DR}}
  (x^+) \leq Q (q (x) ; p (x^+))$, so it suffices to bound the difference
  \begin{align*}
    Q (q (x) ; p (x^+)) - Q (q (x) ; p (x)) & = f (p (x^+)) - f (p (x)) +
    \langle \nabla f (p (x^+)), p (x) - p (x^+) \rangle\\
    & \quad + \langle e, q (x) - p (x) \rangle - \tfrac{1}{\gamma}
    \langle q (x) - p (x), a \rangle + \tfrac{1}{2 \gamma} \| a \|^2 .
  \end{align*}
  The first line is at most $- \tfrac{1}{2 L} \| e \|^2$ by the
  descent lemma for the convex $L$-smooth function $f$. The
  definition of $T^{\mathrm{DR}}$ and the gradient identity above, applied
  at $x$ and $x^+$, give $q (x) - p (x) = x^+ - x = a + \gamma e$, so the
  second line equals $\gamma \| e \|^2 - \tfrac{1}{2 \gamma} \| a \|^2$.
  Altogether,
  \[ F_{\gamma}^{\mathrm{DR}} (x^+) \leq F_{\gamma}^{\mathrm{DR}} (x) -
     \tfrac{1}{2 \gamma} \| a \|^2 + \left( \gamma - \tfrac{1}{2 L}
     \right) \| e \|^2 . \]
  If $\gamma L \leq \tfrac{1}{2}$, drop the last term. Otherwise bound
  it via $\| e \| \leq L \| a \|$ and factor $1 + \gamma L - 2
  \gamma^2 L^2 = (1 - \gamma L)  (1 + 2 \gamma L)$: in both cases
  the decrease is at least $\tfrac{\min \{ 1, (1 - \gamma L) (1 + 2
  \gamma L) \}}{2 \gamma} \| a \|^2$. Finally, $x - x^+ = - (a + \gamma
  e)$ and $\| e \| \leq L \| a \|$ give $\| x - T (x) \| \leq (1 +
  \gamma L) \| a \|$, which converts the bound to the stated constant.
\end{proof}

\Cref{lem:hindsight-i} gives a uniform bound on the hindsight feedback $\bar{\ell}_K^{\star}$ along any trajectory:
\begin{equation}
  \bar{\ell}_K^{\star} \leq \tfrac{1}{K} \textstyle \sum_{k = 0}^{K - 1}
  \ell_{x^k}^{\beta} (I) \leq - c_0 . \label{eq:hindsight-bound}
\end{equation}
The benchmark $\bar{\ell}_K^{\star}$ is itself trajectory-dependent and can be
far smaller than $- c_0$ when a good fixed preconditioner exists. The
local analysis in \Cref{subsec:osop-local} exploits exactly this gap.

We now bound the regret. The next lemma shows that the feedback
$\ell_x^{\beta}$ is convex and Lipschitz. Therefore, the projected online
gradient descent in \Cref{alg:osop} attains sublinear regret. The proof is in \Cref{app:proof-regret}.

\begin{lemma}[Regret bound]
  \label{lem:regret} Assume the conditions of
  \Cref{thm:envelope-regularity} with $\gamma \in (0, 1 / L)$, and
  let \Cref{ass:candidate,ass:level} hold as well. Let
  $\beta \geq \max \{ - \mu_{\gamma}, 0 \}$. For every $x \in
  \mathcal{L}_0$, the feedback $\ell_x^{\beta}$ is
  convex and $G_{\ell}$-Lipschitz on $\mathcal{P}$, where $G_{\ell}
  \assign \tfrac{1}{\gamma} + L_{\gamma} + (L_{\gamma} + \beta)
  \Delta_{\mathcal{P}}$. Then \Cref{alg:osop} with $\eta =
  \tfrac{\Delta_{\mathcal{P}}}{G_{\ell} \sqrt{K}}$ has regret 
  \[R_K \leq
  \Delta_{\mathcal{P}} G_{\ell} \sqrt{K}.\]
\end{lemma}

We have shown above that when the residual PL constant $\tilde{\mu}_{\gamma}$ is positive, 
  combining the feedback reduction~\eqref{eq:feedback-decomp2} with the
hindsight and regret bounds (\Cref{lem:hindsight-i,lem:regret}) gives linear convergence of
the gap $F_{\gamma} (x^k) - F^{\star}$.
We now show that the convex and weakly convex regimes admit
$\mathcal{O} (1 / K)$ rates on the function value gap and on
stationarity, respectively. 
The following theorem
provides guarantees in the residual PL, convex, and
weakly convex regimes, respectively.

\begin{theorem}[Global convergence]
  \label{thm:os-global}
  Suppose the assumptions of \Cref{thm:envelope-regularity} hold for
  the region $\Omega$~\eqref{eq:analysis-region} with $\gamma
  \in (0, 1 / L)$, and \Cref{ass:candidate,ass:level} hold. Run
  \Cref{alg:osop} with online stepsize $\eta =
  \tfrac{\Delta_{\mathcal{P}}}{G_{\ell} \sqrt{K}}$, and let $\delta_K
  \assign \max \{ - \bar{\ell}_K^{\star} - \tfrac{R_K}{K}, 0 \}$. The
  iterates and output $\hat{x}^K$ satisfy
  the following bounds:
  \begin{enumerate}
    \item If $F_{\gamma}$ is residual PL~\eqref{eq:residual-pl} with
      $\tilde{\mu}_{\gamma} > 0$, then
      \[
        F (\hat{x}^K) - F^{\star} \leq (F_{\gamma} (x^0) - F^{\star})
        \left(1 - 2 \tilde{\mu}_{\gamma} \delta_K \right)^K.
      \]
    \item If $F_{\gamma}$ is convex on $\Omega$, then
      \[
        F (\hat{x}^K) - F^{\star} \leq F_{\gamma} (x^K) - F^{\star}
        \leq \tfrac{\max_{x \in \mathcal{L}_0} \tmop{dist} (x,
        \mathcal{X}^{\star})^2}{K \gamma^2 \delta_K}.
      \]
    \item If $F_{\gamma}$ is only weakly convex on $\Omega$, then 
      \[
        \min\limits_{0 \leq k < K} \| \nabla F_{\gamma} (x^k) \|^2
        \leq \tfrac{F_{\gamma} (x^0) - F^{\star}}{K \gamma^2 \delta_K},
      \]
      \[
        \min\limits_{0 \leq k < K} \tmop{dist} (0, \partial F (q
        (x^k)))^2 \leq \tfrac{F_{\gamma} (x^0) - F^{\star}}{K \gamma^2
        \delta_K}.
      \]
  \end{enumerate}
\end{theorem}

\begin{proof}
  The descent inequality~\eqref{eq:null-descent} makes $F_{\gamma}$
  nonincreasing along the iterates, so $x^k \in \mathcal{L}_0$ for all $k$.
  Write $d^k \assign x^k - T (x^k)$.
  We also record a consequence of the
  decomposition~\eqref{eq:feedback-decomp2} and the definition of
  $\delta_K$ in \Cref{thm:os-global}: since $\min \{
  \ell_{x^k}^{\beta} (P_k), 0 \} \leq \ell_{x^k}^{\beta} (P_k)$ and
  $\min \{ \ell_{x^k}^{\beta} (P_k), 0 \} \leq 0$,
  \begin{equation}
    \sum_{k = 0}^{K - 1} \min \{ \ell_{x^k}^{\beta} (P_k), 0 \}
    \leq \min \left\{ \sum_{k = 0}^{K - 1} \ell_{x^k}^{\beta} (P_k),
    0 \right\} \leq \min \{ K \bar{\ell}_K^{\star} + R_K, 0 \} = - K
    \delta_K . \label{eq:pf-regret}
  \end{equation}
  For part~1, if $\bar{\ell}_K^{\star} + \tfrac{R_K}{K} \leq 0$, then the
  feedback reduction~\eqref{eq:feedback-decomp2} and the definition of
  $\delta_K$ give the stated envelope-gap bound. Otherwise the definition
  gives $\delta_K = 0$, and the same bound follows from monotonicity. The sandwich bound~\eqref{eq:sandwich} then gives $F (\hat{x}^K) \leq
  F_{\gamma} (x^K)$, which proves part~1.

  For part~2, assume $F_{\gamma} (x^k) > F^{\star}$ for $k = 0, \ldots, K$:
  otherwise some iterate is optimal, the null step keeps the gap at zero,
  and the second bound holds. From~\eqref{eq:null-descent} and $F_{\gamma}
  (x^{k + 1}) \leq F_{\gamma}
  (x^k)$,
  \[ \tfrac{1}{F_{\gamma} (x^{k + 1}) - F^{\star}} - \tfrac{1}{F_{\gamma}
     (x^k) - F^{\star}} = \tfrac{F_{\gamma} (x^k) - F_{\gamma} (x^{k +
     1})}{(F_{\gamma} (x^{k + 1}) - F^{\star})  (F_{\gamma} (x^k) -
     F^{\star})} \geq \tfrac{- \min \{ \ell_{x^k}^{\beta} (P_k), 0 \}
     \| d^k \|^2}{(F_{\gamma} (x^k) - F^{\star})^2} . \]
  Convexity of $F_{\gamma}$ on $\Omega$ and the residual
  bound~\eqref{eq:res-grad} give
  \[ F_{\gamma} (x^k) - F^{\star} \leq \| \nabla F_{\gamma} (x^k) \|
     \tmop{dist} (x^k, \mathcal{X}^{\star}) \leq \tfrac{1}{\gamma} \|
     d^k \| \tmop{dist} (x^k, \mathcal{X}^{\star}), \]
  hence
  \[ \tfrac{\| d^k \|^2}{(F_{\gamma} (x^k) - F^{\star})^2} \geq
     \tfrac{\gamma^2}{\tmop{dist} (x^k,
     \mathcal{X}^{\star})^2} \geq
     \tfrac{\gamma^2}{(\max_{x \in \mathcal{L}_0}
     \tmop{dist} (x, \mathcal{X}^{\star}))^2} . \]
  Telescoping, dropping the nonnegative $\tfrac{1}{F_{\gamma} (x^0) -
  F^{\star}}$ term, and applying~\eqref{eq:pf-regret},
  \[ \tfrac{1}{F_{\gamma} (x^K) - F^{\star}} \geq
     \tfrac{\gamma^2}{(\max_{x \in \mathcal{L}_0} \tmop{dist}
     (x, \mathcal{X}^{\star}))^2} \sum_{k = 0}^{K - 1} (- \min \{
     \ell_{x^k}^{\beta} (P_k), 0 \}) \geq \tfrac{K \gamma^2
     \delta_K}{(\max_{x \in \mathcal{L}_0}
     \tmop{dist} (x, \mathcal{X}^{\star}))^2}. \]

  For part~3, the identity $d^k = p (x^k) - q (x^k)$ and the proximal
  optimality conditions give
  \[ v^k \assign \tfrac{1}{\gamma} d^k + \nabla f (q (x^k)) -
     \nabla f (p (x^k)) \in \partial F (q (x^k)). \]
  Since $f$ is convex and $L$-smooth, $\nabla f (p (x^k)) -
  \nabla f (q (x^k)) = H_k d^k$ for a symmetric matrix $0 \preceq H_k
  \preceq L I$. The stepsize condition $\gamma < \tfrac{1}{L}$
  therefore gives
  \[ \tmop{dist} (0, \partial F (q (x^k))) \leq \| v^k \| \leq
     \tfrac{1}{\gamma} \| d^k \| . \]
  Let the scalar $r_k$ denote either $\| \nabla F_{\gamma} (x^k) \|$ or
  $\tmop{dist} (0, \partial F (q (x^k)))$. The residual
  bound~\eqref{eq:res-grad} and the preceding estimate give $\gamma r_k
  \leq \| d^k \|$. Combining this inequality with~\eqref{eq:null-descent}
  and $\min \{ \ell_{x^k}^{\beta} (P_k), 0 \} \leq 0$ gives
  \[ \begin{aligned}
       F_{\gamma} (x^{k + 1}) - F_{\gamma} (x^k)
       &\leq \gamma^2 \min \{ \ell_{x^k}^{\beta} (P_k), 0 \}
       r_k^2 \\
       &\leq \gamma^2 \min \{ \ell_{x^k}^{\beta} (P_k), 0 \}
       \min_{0 \leq j < K} r_j^2 .
     \end{aligned} \]
  Summing over $k = 0, \ldots, K - 1$ and using $F_{\gamma} (x^K) \geq
  F^{\star}$ (the envelope shares the minimum value of $F$ by~\eqref{eq:env-eqiv}),
  \[ F^{\star} - F_{\gamma} (x^0) \leq \gamma^2
     \min_{0 \leq j < K} r_j^2
     \sum_{k = 0}^{K - 1} \min \{ \ell_{x^k}^{\beta} (P_k), 0 \}, \]
  and rearranging with~\eqref{eq:pf-regret} proves both stationarity bounds.
\end{proof}

\begin{remark}[Non-vacuous results]
  \label{rem:os-nonvacuous}
  The bounds of \Cref{thm:os-global} are non-vacuous only when
  $\delta_K > 0$. Combining the hindsight
  bound~\eqref{eq:hindsight-bound} with the regret bound (\Cref{lem:regret}) gives
  $\delta_K \geq c_0 - \tfrac{\Delta_{\mathcal{P}} G_{\ell}}{\sqrt{K}}$,
  so $\delta_K$ is positive whenever $K > (\tfrac{\Delta_{\mathcal{P}}
  G_{\ell}}{c_0})^2$.
\end{remark}

\begin{remark}[Comparison with the base method]
  \label{rem:os-base}
  The base splitting method satisfies all three bounds of
  \Cref{thm:os-global} with $\delta_K$ replaced by the constant $c_0$
  of \Cref{lem:hindsight-i}. 
  As a result, OSOP provides a sharper guarantee
  when $- \bar{\ell}_K^{\star} - \tfrac{R_K}{K} > c_0$. Since $- \bar{\ell}_K^{\star} \geq c_0$
  by~\eqref{eq:hindsight-bound} and $\tfrac{R_K}{K} = \mathcal{O} (1 / \sqrt{K})$
  by \Cref{lem:regret}, OSOP asymptotically recovers the guarantees of
  the base method. Moreover, if some fixed $\hat{P} \in \mathcal{P}$
  keeps the hindsight feedback bounded away from $- c_0$,
  then OSOP is strictly sharper for large enough $K$.
\end{remark}

\subsection{Local convergence analysis}\label{subsec:osop-local}

In this section, we prove the local superlinear convergence of OSOP. The key is that
an optimal preconditioner exists when the splitting operator $T$ becomes affine near the solution $x^{\star}$. Suppose $T$ is locally affine:
\begin{equation}
  T (x) = x^{\star} + M (x - x^{\star}), \quad \forall x \in \mathcal{B} (x^{\star}, \rho) \label{eq:local-affine}
\end{equation}
where $\mathcal{B}
(x^{\star}, \rho) \assign \{ x : \| x - x^{\star} \| \leq \rho \}$. 
In this case, if $I - M$ is nonsingular, the splitting step in~\eqref{eq:local-affine} has an exact optimal preconditioner $P_{\star} \assign (I - M)^{- 1}$ that sends any nearby iterate $x$ to
$x^{\star}$ in one step.

The local superlinear convergence of OSOP requires the following assumption.

\begin{assumption}[Locally affine]
  \label{ass:local} Suppose $T$ satisfies~\eqref{eq:local-affine}, $I - M$ is nonsingular, and $(I - M)^{- 1} \in \mathcal{P}$.
\end{assumption}

\begin{example}
  \label{ex:lasso-box} The local conditions are standard for
  FBS in the following two applications with $Q \succeq 0$:
    \begin{alignat}{3}
    &\text{Lasso:}\qquad
    && f(x) = \tfrac{1}{2}x^{\top}Qx + c^{\top}x,
    \qquad && g(x) = \tau\|x\|_1, \notag\\
    &\text{box-constrained QP:}\qquad
    && f(x) = \tfrac{1}{2}x^{\top}Qx + c^{\top}x,
    \qquad && g(x) = \delta_{[l,u]}(x). \notag
  \end{alignat}
  Under strict complementarity, soft-thresholding and box
  projection are affine near $x^{\star} - \gamma \nabla f (x^{\star})$.
  Hence $T^{\mathrm{FB}}$ is locally affine with $M = \Pi_S (I - \gamma
  Q)$, where $\Pi_S$ is the coordinate projector onto the support of the
  solution for Lasso and onto the inactive coordinates for box-constrained
  QP. Positive definiteness of the reduced Hessian $Q_{S S}$ makes $I - M$
  nonsingular and the FBE locally strongly convex. Therefore,
  \Cref{ass:local} holds whenever the remaining condition $P_{\star} = (I -
  M)^{- 1} \in \mathcal{P}$ is satisfied.
\end{example}

Define the
\emph{identification index} as the first iteration after which the iterates stay in the ball $\mathcal{B} (x^{\star}, \rho)$:
\[ K_0 \assign \min \{ k : x^j \in \mathcal{B} (x^{\star}, \rho) \text{
   for all } j \geq k \}. \]
After $K_0$, the splitting step is affine and thus gives the local superlinear rate.

\begin{theorem}[Local superlinear convergence]
  \label{thm:os-local}
  Assume the conditions of \Cref{thm:os-global}
  as well as \Cref{ass:local}. Assume $F_{\gamma}$ is globally convex and $\mu_{\gamma}$-strongly convex on $\mathcal{B}(x^{\star}, \rho)$.
  Run \Cref{alg:osop} with regularizer $\beta = 0$ and stepsize $\eta =
  \tfrac{\Delta_{\mathcal{P}}}{G_{\ell} \sqrt{K}}$ for $K$ iterations.
  Then the identification index satisfies $K_0
  \leq C \sqrt{K} + C'$ for constants $C, C'$ independent
  of $K$. Whenever $K > K_0$, the iterates satisfy
  \begin{equation}
    \| x^K - x^{\star} \|^2 \leq \| x^{K_0} - x^{\star} \|^2 \left(
    \tfrac{2 \| I - M \|^2}{\mu_{\gamma}} \cdot
    \tfrac{\Delta_{\mathcal{P}} G_{\ell} \sqrt{K}}{K - K_0} \right)^{K -
    K_0} .
    \label{eq:os-local-rate}
  \end{equation}
\end{theorem}

\begin{proof}
  \Cref{ass:local} makes $x^{\star}$ the unique minimizer of
  $F_{\gamma}$: $\mathcal{X}^{\star} = \{ x^{\star} \}$.

  We first prove the bound on the identification index. Since $\mu_{\gamma} \geq 0$, the proof of the second bound of
  \Cref{thm:os-global}, applied to the prefix of length $k$ and using the same
  regret argument as \Cref{lem:regret}, gives
  \[
    \begin{aligned}
      F_{\gamma} (x^k) - F^{\star}
      &\leq \tfrac{(\max_{x \in \mathcal{L}_0} \tmop{dist} (x,
      \mathcal{X}^{\star}))^2}{k \gamma^2 \max \left\{ -
      \bar{\ell}_k^{\star} - \Delta_{\mathcal{P}} G_{\ell} \sqrt{K} / k, 0
      \right\}} \\
      &\leq \tfrac{2 (\max_{x \in \mathcal{L}_0} \tmop{dist} (x,
      \mathcal{X}^{\star}))^2}{c_0 \gamma^2 \, k}.
    \end{aligned}
  \]
  once $k \geq k_1 \assign \tfrac{2 \Delta_{\mathcal{P}} G_{\ell}
     \sqrt{K}}{c_0}$, because~\eqref{eq:hindsight-bound} shows that the bracket is at least $c_0 -
  \Delta_{\mathcal{P}} G_{\ell} \sqrt{K} / k \geq c_0 / 2$ for such $k$.
  Because $\mathcal{L}_0$ is compact, $F_{\gamma}$ is continuous, and
  $\mathcal{X}^{\star} = \{ x^{\star} \}$, there is $\varepsilon_{\rho} > 0$
  with $\{ x \in \mathcal{L}_0 : F_{\gamma} (x) \leq F^{\star} +
  \varepsilon_{\rho} \} \subseteq \mathcal{B} (x^{\star}, \rho)$: otherwise
  a sequence in $\mathcal{L}_0 \setminus \mathcal{B} (x^{\star}, \rho)$
  with vanishing gap would accumulate at a minimizer other than
  $x^{\star}$. The envelope value is nonincreasing along the iterates~\eqref{eq:null-descent}, so every iterate from the first entry into
  this sublevel set onward remains in $\mathcal{B} (x^{\star}, \rho)$, whence
  \[ K_0 \leq \max \left\{ k_1, \tfrac{2 (\max_{x \in \mathcal{L}_0}
     \tmop{dist} (x, \mathcal{X}^{\star}))^2}{c_0 \gamma^2
     \varepsilon_{\rho}} \right\} + 1 = \mathcal{O} (\sqrt{K}) . \]
  We next prove the local convergence rate~\eqref{eq:os-local-rate}. Write $A
  \assign I - M$, $x^i \assign x^{K_0 + i}$, and $P_{\star} \assign A^{- 1}$. For every $i \geq 0$ we have $x^i \in
  \mathcal{B} (x^{\star}, \rho)$, so $x^i - T (x^i) = A (x^i - x^{\star})$
  and the probe at $P_{\star}$ lands exactly on the solution, $x^i -
  P_{\star} (x^i - T (x^i)) = x^{\star}$. Hence
  \[ \ell_{x^i} (P_{\star}) = \tfrac{F^{\star} - F_{\gamma} (x^i)}{\| x^i -
     T (x^i) \|^2} \leq 0. \]
  With $\beta = 0$, the descent inequality~\eqref{eq:null-descent} gives,
  for $a_i \assign \min \{ \ell_{x^i} (P_{K_0 + i}), 0 \} - \ell_{x^i}
  (P_{\star})$,
  \[ F_{\gamma} (x^{i + 1}) - F^{\star} = \big( F_{\gamma} (x^{i + 1}) -
     F_{\gamma} (x^i) \big) + \big( F_{\gamma} (x^i) - F^{\star} \big)
     \leq a_i \| x^i - T (x^i) \|^2, \]
  and $a_i \geq 0$ because the left-hand side is nonnegative. Since
  $\nabla F_{\gamma} (x^{\star}) = 0$, strong convexity on the ball
  gives $F_{\gamma} (x^{i + 1}) - F^{\star} \geq
  \tfrac{\mu_{\gamma}}{2} \| x^{i + 1} - x^{\star} \|^2$, which with
  $\| x^i - T (x^i) \| \leq \| A \| \| x^i - x^{\star} \|$ yields the
  recursion
  \[ \| x^{i + 1} - x^{\star} \|^2 \leq \tfrac{2 \| A
     \|^2}{\mu_{\gamma}} \, a_i \, \| x^i - x^{\star} \|^2 . \]
  Multiplying the
  recursion over $i = 0, \ldots, N - 1$ and applying the AM--GM inequality
  to the nonnegative $a_i$,
  \[ \tfrac{\| x^N - x^{\star} \|^2}{\| x^0 - x^{\star} \|^2} \leq \left(
     \tfrac{2 \| A \|^2}{\mu_{\gamma}} \right)^N \prod_{i = 0}^{N -
     1} a_i \leq \left( \tfrac{2 \| A \|^2}{\mu_{\gamma}} \cdot
     \tfrac{1}{N} \sum_{i = 0}^{N - 1} a_i \right)^N . \]
  Since $\min \{ \ell_{x^i} (P_{K_0 + i}), 0 \} \leq \ell_{x^i} (P_{K_0 +
  i})$ and $P_{\star} \in \mathcal{P}$, the sum is bounded by the same
  projected-gradient regret argument, re-indexed from the initial point
  $P_{K_0}$:
  \[
    \begin{aligned}
      \sum_{i = 0}^{N - 1} a_i
      &\leq \sum_{k = K_0}^{K_0 + N - 1} \left[
      \ell_{x^k} (P_k) - \ell_{x^k} (P_{\star}) \right] \\
      &\leq \tfrac{\| P_{K_0} - P_{\star} \|_F^2}{2 \eta}
      + \tfrac{\eta}{2} \sum_{k = K_0}^{K_0 + N - 1}
      \| \nabla \ell_{x^k} (P_k) \|_F^2 \leq \Delta_{\mathcal{P}} G_{\ell} \sqrt{K}.
    \end{aligned}
  \]
  Setting $N = K - K_0$, so that $x^N = x^K$ and $x^0 = x^{K_0}$,
  proves~\eqref{eq:os-local-rate}. Finally, $K_0 = \mathcal{O}
  (\sqrt{K})$ gives $K - K_0 \geq K -\mathcal{O} (\sqrt{K})$, so the
  base of~\eqref{eq:os-local-rate} is $\mathcal{O} (\sqrt{K} / (K - K_0))
  = \mathcal{O} (1 / \sqrt{K})$.
\end{proof}

\subsection{Convergence analysis of OS-ADMM}\label{subsec:os-admm-conv}

Since OS-ADMM (\Cref{alg:os-admm}) is equivalent to OS-DRS applied to the dual problem, 
the convergence rates of \Cref{thm:os-global}
apply with $(L_{\gamma}, \mu_{\gamma})$ replaced by the constants
$(L^{\tmop{ADMM}}_{\gamma}, \mu^{\tmop{ADMM}}_{\gamma})$ of
\Cref{lem:admm-regularity}.
When the envelope is convex, the envelope gap
$F^{\tmop{ADMM}}_{\gamma} (\lambda^k) - F^{\star}$ converges at rate $\mathcal{O} (1 / K)$. It converges linearly when the residual PL
constant~\eqref{eq:residual-pl} is positive.
In the weakly convex case, we have a guarantee on the gradient norm: $\min_{0 \leq k < K} \| \nabla
F^{\tmop{ADMM}}_{\gamma} (\lambda^k) \|^2 = \mathcal{O} (1 / (K
\delta_K))$.

It remains to convert
these dual guarantees into optimality guarantees for
problem~\eqref{eq:admm}.
The KKT conditions of the ADMM problem~\eqref{eq:admm} are
\begin{equation}
  A x + B z - b = 0, \quad - A^{\top} \nu \in \partial \phi (x), \quad - B^{\top}
  \nu \in \partial \psi (z) . \label{eq:kkt}
\end{equation}
The following proposition bounds the KKT residual of an OS-ADMM iterate by
the envelope gap and the envelope gradient norm.

\begin{proposition}[Envelope gap and KKT residual]
  \label{lem:admm-kkt}
  Suppose either \Cref{ass:admm} or \Cref{ass:admm-quad} holds, 
  and $0 < \gamma <
\tfrac{1}{L}$, with $L = \tfrac{\| B \|^2}{\mu_\psi}$.
  Let $(z^k, \nu^k)$ be an iterate of \Cref{alg:os-admm}.
  Define $\lambda^k = \nu^k - \gamma B z^k$, $x^{k +} = x^+ (z^k, \nu^k)$,
  $r_x^k = A x^{k +} + B z^k - b$, and $q^k \assign \nu^k + \gamma r_x^k$.
  Then $(x^{k +}, z^k, q^k)$ satisfies
  the KKT conditions~\eqref{eq:kkt} up to residuals controlled by $\|
  r_x^k \|$:
  \begin{equation}
   \begin{aligned}
     & A x^{k +} + B z^k - b = r_x^k, \\
     & 0 \in \partial \phi (x^{k +}) + A^{\top} q^k, \\
     & \tmop{dist} (0, \partial \psi (z^k) + B^{\top} q^k)
     \leq \gamma \| B^{\top} r_x^k \| .
   \end{aligned}
   \label{eq:admm-kkt-residual}
  \end{equation}
  Moreover, the envelope gap and the envelope
  gradient both control $\|
  r_x^k \|$:
  \[
    \begin{aligned}
      \| r_x^k \| \leq \sqrt{\tfrac{2 (F_{\gamma}^{\tmop{ADMM}} (\lambda^k)
      - F^{\star})}{\gamma (1 - \gamma L)}}, \;\;
      \| r_x^k \| \leq \tfrac{1 + \gamma L}{1 - \gamma L}
      \| \nabla F_{\gamma}^{\tmop{ADMM}} (\lambda^k) \| .
    \end{aligned}
  \]
\end{proposition}

\begin{proof}
  We first prove the three relations in~\eqref{eq:admm-kkt-residual}. The
  proximal identities~\eqref{eq:admm-prox-identities} give
  $p (\lambda^k) = \nu^k$ and $q (\lambda^k) = \nu^k + \gamma r_x^k = q^k$.
  Hence $p (\lambda^k) - q (\lambda^k) = - \gamma r_x^k$. The first relation
  is the definition of $r_x^k$.

  We next prove $0 \in \partial \phi (x^{k +}) + A^{\top} q^k$. The
  optimality condition of the $x^{k +}$-update gives
  \[ 0 \in \partial \phi (x^{k +}) + A^{\top} \nu^k + \gamma A^{\top} r_x^k =
     \partial \phi (x^{k +}) + A^{\top} q^k . \]

  We then prove $\tmop{dist} (0, \partial \psi (z^k) + B^{\top} q^k)
  \leq \gamma \| B^{\top} r_x^k \|$. The initialization in
  \Cref{alg:os-admm}, the optimality condition of each probe $z$-update, and
  the null step show inductively that every OS-ADMM iterate satisfies
  $- B^{\top} \nu^k \in \partial \psi (z^k)$. Therefore
  \[ \gamma B^{\top} r_x^k = B^{\top} q^k - B^{\top} \nu^k \in \partial \psi
     (z^k) + B^{\top} q^k, \]
  so $\tmop{dist} (0, \partial \psi (z^k) + B^{\top} q^k) \leq \gamma \|
  B^{\top} r_x^k \|$.

  We next prove that the envelope gap and envelope gradient norm both control
  $\| r_x^k \|$. First, for the envelope gap, the proximal
  identities~\eqref{eq:admm-prox-identities} give $q (\lambda^k) = q^k$ and
  $p (\lambda^k) - q (\lambda^k) = - \gamma r_x^k$. Applying the sandwich
  bound~\eqref{eq:sandwich} to the dual envelope with $L_{f} = L$ gives
  \[ F^{\tmop{ADMM}} (q^k) + \tfrac{1 - \gamma L}{2 \gamma} \|
     \gamma r_x^k \|^2 \leq F_{\gamma}^{\tmop{ADMM}} (\lambda^k) . \]
  Define the envelope gap $\varepsilon^k \assign
  F_{\gamma}^{\tmop{ADMM}} (\lambda^k) - F^{\star}$.
  Subtracting $F^{\star}$ from both sides,
  \[ F^{\tmop{ADMM}} (q^k) - F^{\star} + \tfrac{\gamma (1 - \gamma
     L)}{2} \| r_x^k \|^2 \leq \varepsilon^k . \]
  Since $F^{\tmop{ADMM}} (q^k) - F^{\star} \geq 0$, this yields $\| r_x^k
  \|^2 \leq \tfrac{2 \varepsilon^k}{\gamma (1 - \gamma L)}$.

  For the envelope gradient norm, the envelope-gradient formula of
  \Cref{lem:drs-admm} gives $\nabla F_{\gamma}^{\tmop{ADMM}} (\lambda^k)
  = - (2 J_{z^k} - I) r_x^k$. Since $\nabla^2 \psi (z^k) \succeq
  \mu_{\psi} I$, the matrix $\gamma B [\nabla^2 \psi (z^k)]^{- 1}
  B^{\top}$ has spectrum in $[0, \gamma L]$, so $\tfrac{1}{1 + \gamma L}
  I \preceq J_{z^k} \preceq I$ and $2 J_{z^k} - I \succeq \tfrac{1 -
  \gamma L}{1 + \gamma L} I \succ 0$. Therefore
  \[ \| \nabla F_{\gamma}^{\tmop{ADMM}} (\lambda^k) \| \geq \tfrac{1 -
     \gamma L}{1 + \gamma L} \| r_x^k \|, \]
  which yields the stated bounds.
\end{proof}

Combining \Cref{lem:admm-kkt} with \Cref{thm:os-global} gives guarantees for the
iterates of \Cref{alg:os-admm}. In the convex case, the envelope gap
satisfies $F^{\tmop{ADMM}}_{\gamma} (\lambda^K) - F^{\star} =
\mathcal{O} (1 / (K \delta_K))$. In the weakly convex case, the
envelope gradient satisfies $\min_{0 \leq k < K} \| \nabla
F^{\tmop{ADMM}}_{\gamma} (\lambda^k) \| = \mathcal{O} (1 / \sqrt{K
\delta_K})$. 
The two bounds of \Cref{lem:admm-kkt} then give convergence rates for the KKT residual in the last-iterate and best-iterate senses, respectively.

The convergence rates of OS-ADMM are consistent with the results in the literature.
Chen et al. \cite{chen2026admm} recently construct a counterexample showing that
the lower bound on the last-iterate and equal-weight ergodic convergence rates of ADMM is $\Omega (1 / \sqrt{K})$.

\section{Numerical experiments} \label{sec:exp}

In this section, we benchmark the OS-FBS algorithm on statistical learning problems with nonsmooth regularizers
and the OS-ADMM algorithm on quadratic programs.

\subsection{Statistical learning problems}

We first evaluate OS-FBS on Lasso and $\ell_1$-regularized logistic regression. 
Let $A \in \mathbb{R}^{m \times n}$
denote the data matrix, with rows $a_i^{\top}$, and let $x \in
\mathbb{R}^n$ denote the model parameter. 
The Lasso problem is
\[
  \min_{x \in \mathbb{R}^n}
  \frac{1}{2} \| A x - b \|^2 + \tau \| x \|_1 .
\]
The $\ell_1$-regularized logistic regression problem with binary labels
$b_i \in \{ -1, 1 \}$ is
\[
  \min_{x \in \mathbb{R}^n}
  \frac{1}{m} \sum_{i = 1}^m
  \log (1 + \exp (- b_i a_i^{\top} x)) + \tau \| x \|_1 .
\]

\textbf{Benchmark algorithms.} We use the following benchmark algorithms.
For all methods, we tune the parameters on a small holdout set and then
fix them for all runs. We set $\gamma = 0.9 / L$ for all methods.

\begin{itemizedot}
  \item {\tmstrong{FBS}}. Vanilla forward-backward splitting (proximal gradient descent).
  
  \item {\tmstrong{FISTA}} {\cite{beck2009fast}}. The fast iterative shrinkage-thresholding algorithm, which is Nesterov's acceleration of FBS.
  
  \item {\tmstrong{PANOC}} {\cite{stella2017simple}}. An FBE line-search algorithm that globally safeguards quasi-Newton acceleration of proximal-gradient iterations. We use a memory of 10 and tune the line-search-shrinkage and target-decrease parameters.
  
  \item {\tmstrong{Anderson}} {\cite{mai2020anderson,zhang2020globally}}. Anderson-accelerated FBS with a memory of 5 and tuned regularization and damping parameters.

  \item \tmtextbf{OS-FBS}. \Cref{alg:osop} with diagonal
  preconditioners, an AdaGrad scheduler, and a tuned learning rate.
\end{itemizedot}
\begin{figure}[htbp]
  {\noindent}\begin{tabularx}{1.0\textwidth}{@{}X@{}@{}X@{}}
    \begin{center}
      \raisebox{0.0\height}{\includegraphics[width=0.8\linewidth]{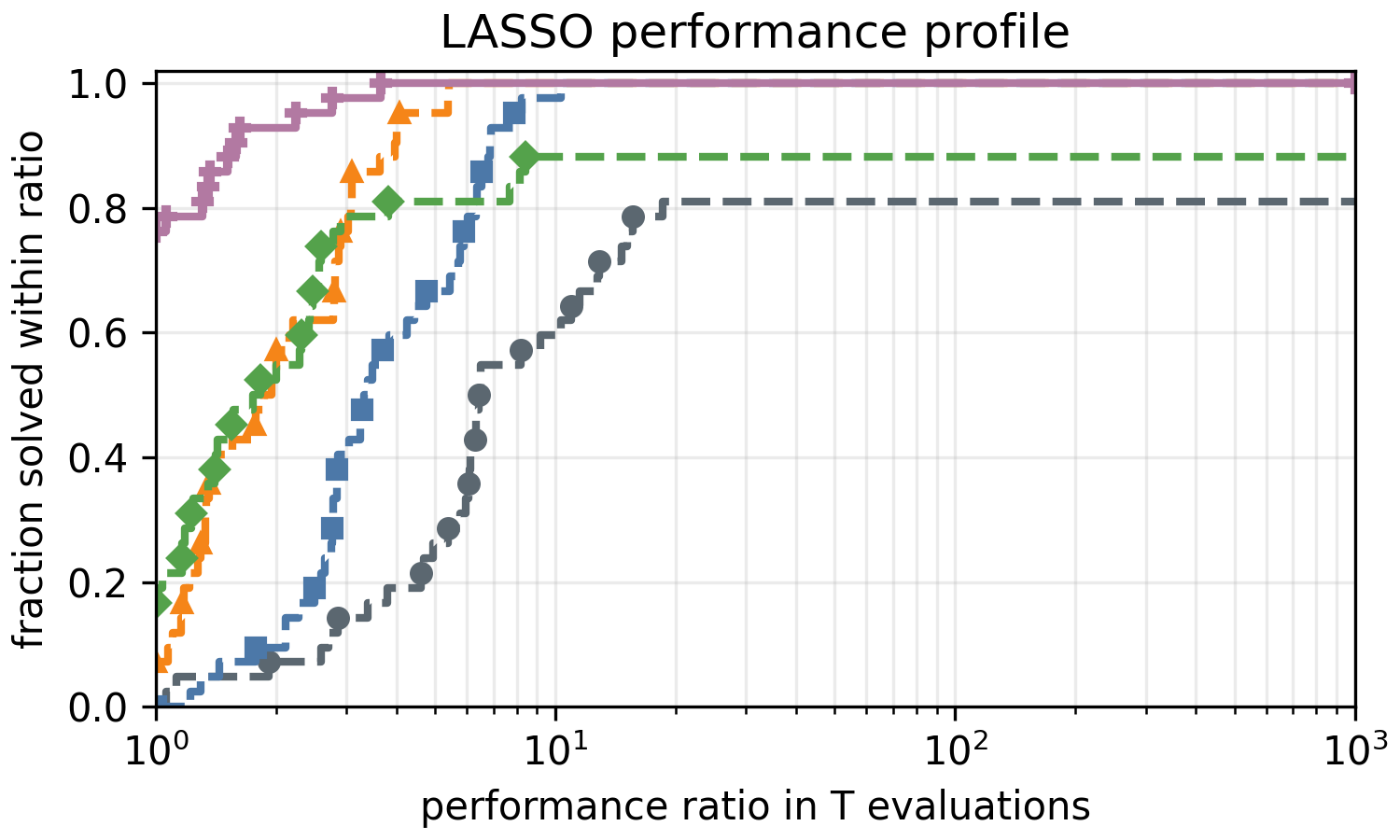}}
    \end{center} & \begin{center}
      \raisebox{0.0\height}{\includegraphics[width=0.8\linewidth]{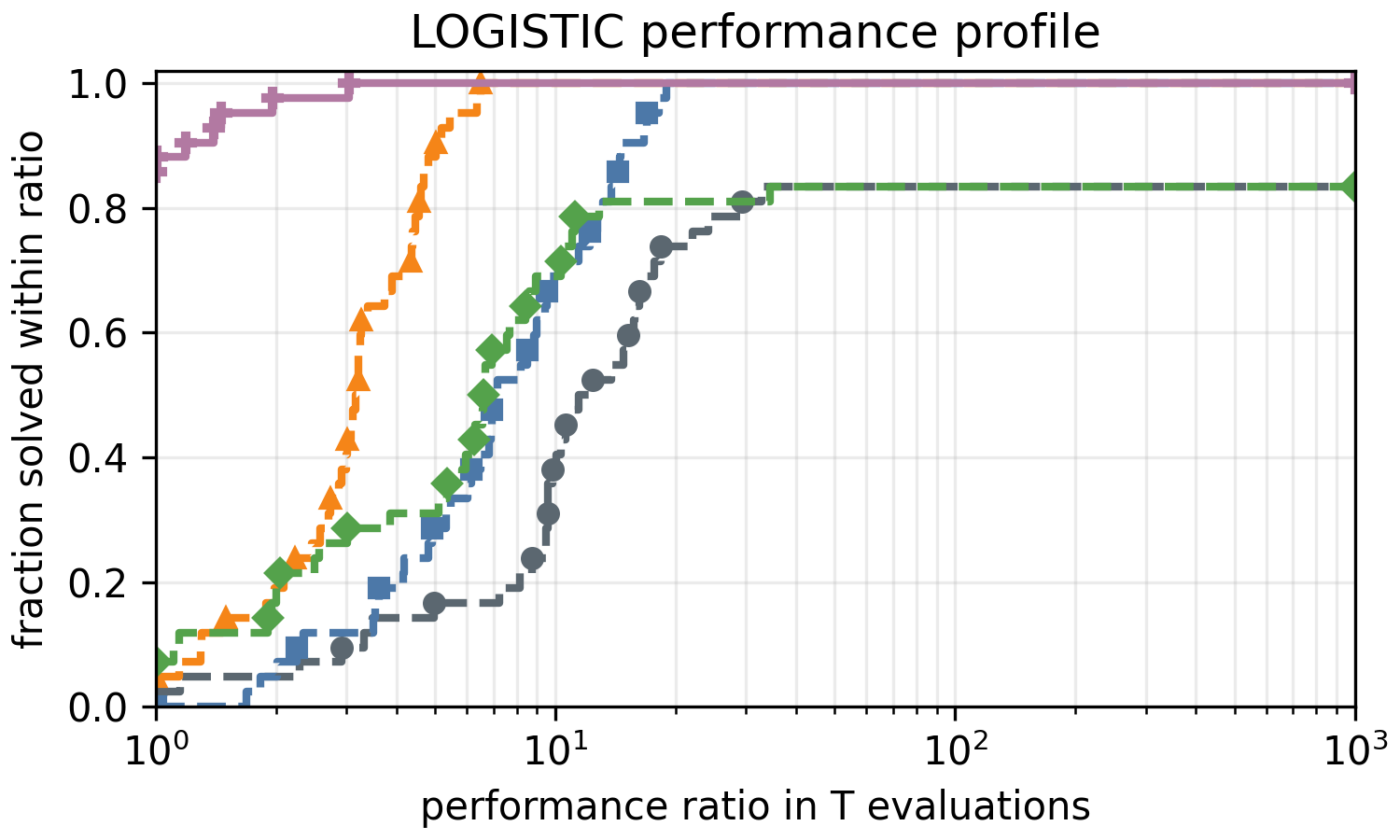}}
    \end{center}\\
    \multicolumn{2}{c}{\raisebox{0.0\height}{\includegraphics[width=0.6\textwidth]{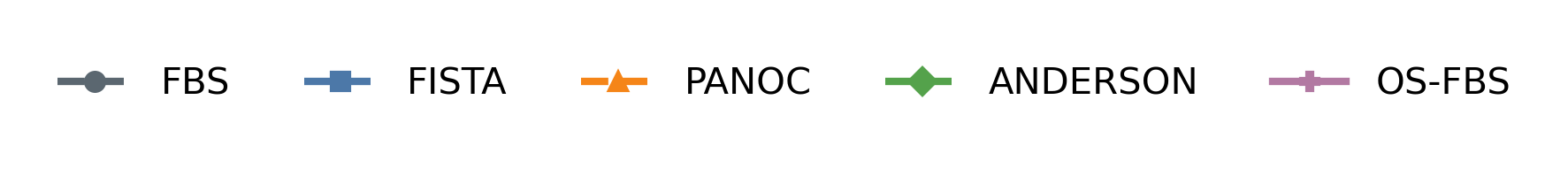}}}
  \end{tabularx}
  \caption{\label{fig:lasso-logistic-profile}Performance profiles for Lasso and
  $\ell_1$-regularized logistic regression at tolerance $10^{-6}$}
\end{figure}
\textbf{Datasets.} We use LIBSVM benchmark datasets and normalize the data in
each dataset.
The benchmark suite consists of the following 42 instances:
\texttt{duke}, \texttt{leukemia}, \texttt{colon-\allowbreak cancer}, \texttt{madelon},
\texttt{liver-\allowbreak disorders}, \texttt{sonar}, \texttt{heart},
\texttt{ionosphere}, \texttt{breast-\allowbreak cancer}, \texttt{cod-\allowbreak rna},  \texttt{australian},
\texttt{diabetes}, \texttt{fourclass}, \texttt{splice},
\texttt{german.\allowbreak numer}, \texttt{svmguide3},
\texttt{svmguide1},
\texttt{mushrooms}, \texttt{phishing}, \texttt{a1a}, \texttt{a2a},
\texttt{a3a}, \texttt{a4a}, \texttt{a5a}, \texttt{a6a}, \texttt{a7a}, \texttt{a8a}, \texttt{a9a},
\texttt{w1a}, \texttt{w2a}, \texttt{w3a}, \texttt{w4a}, \texttt{w5a},
\texttt{w6a}, \texttt{w7a}, \texttt{w8a}, \texttt{ijcnn1}, \texttt{rcv1},
\texttt{news20}, \texttt{gisette}, \texttt{real-\allowbreak sim}, and
\texttt{covtype.\allowbreak binary}.

\textbf{Results.} 
\Cref{fig:lasso-logistic-profile} summarizes the performance profiles for the full LIBSVM suite at
tolerance $10^{-6}$.
The profiles show that OS-FBS is consistently among the most robust methods.
For both Lasso and $\ell_1$-regularized logistic regression, OS-FBS reaches the target tolerance on all profiled
instances with fewer evaluations of $T$ than
PANOC and FISTA.
OS-FBS remains effective compared with methods that use stronger acceleration
mechanisms.

\Cref{fig:lasso,fig:logistic} show representative convergence plots of the
objective gap $F (x^k) - F^{\star}$ against the number of evaluations of the
fixed-point map $T (\cdot)$ for Lasso and logistic regression, respectively.
OS-FBS is the fastest method on many instances and shows substantial speedups
over the other benchmark methods.
PANOC uses quasi-Newton directions to accelerate FBS, but its line search makes
it less competitive with OS-FBS in the number of $T$ evaluations.

\begin{figure}[htbp]
  {\noindent}\begin{tabularx}{1.0\textwidth}{@{}X@{}@{}X@{}@{}X@{}}
    \begin{center}
      \raisebox{0.0\height}{\includegraphics[width=\linewidth]{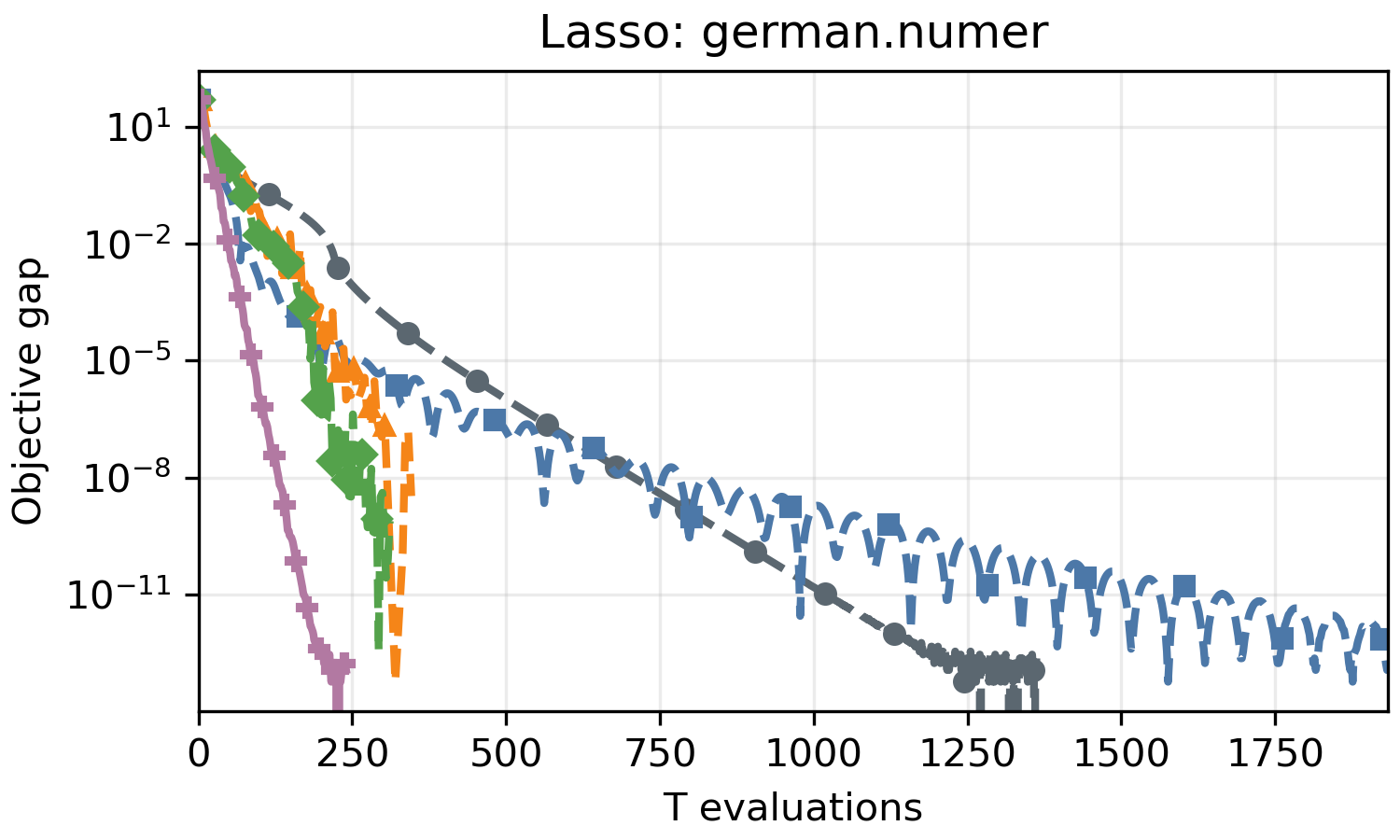}}
    \end{center} & \begin{center}
      \raisebox{0.0\height}{\includegraphics[width=\linewidth]{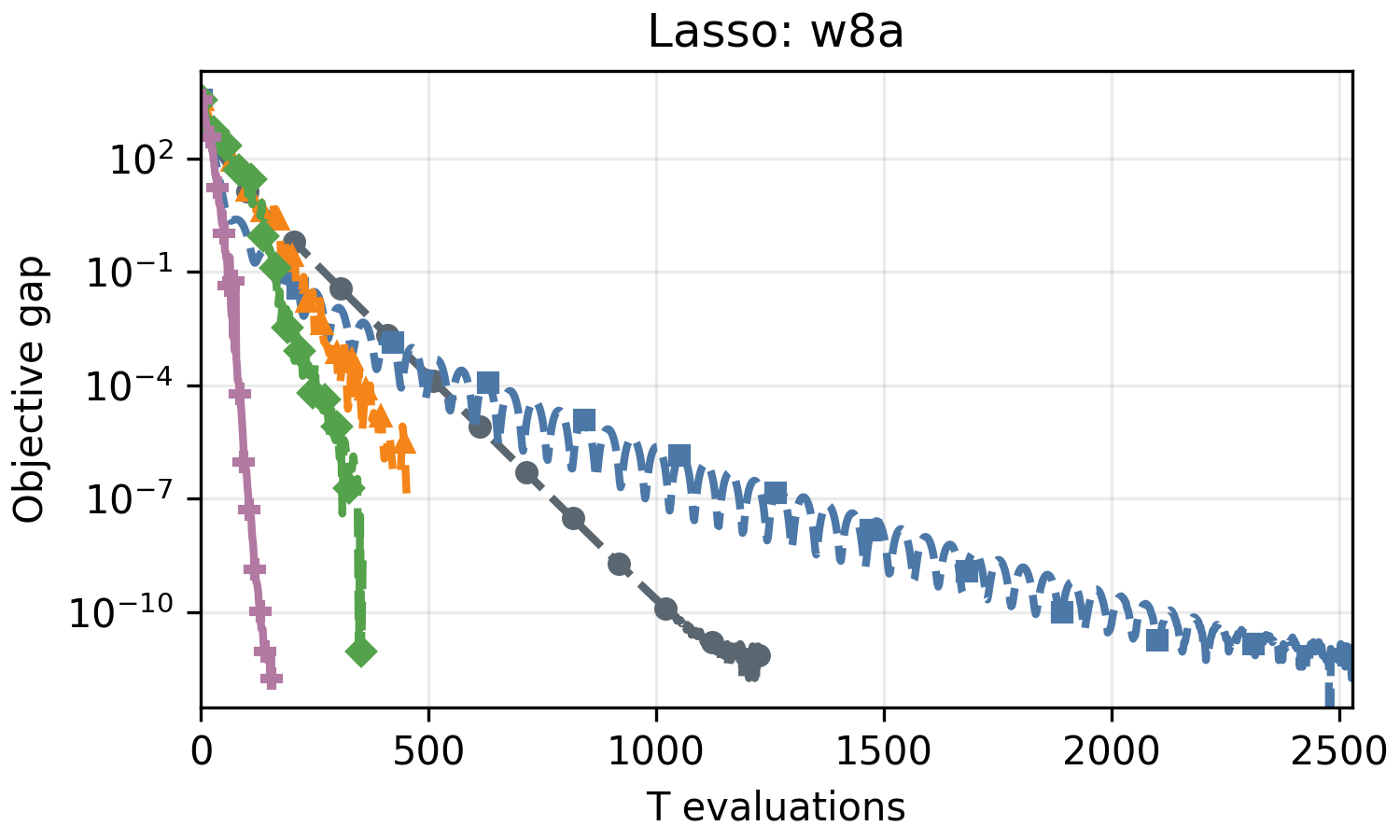}}
    \end{center} & \begin{center}
      \raisebox{0.0\height}{\includegraphics[width=\linewidth]{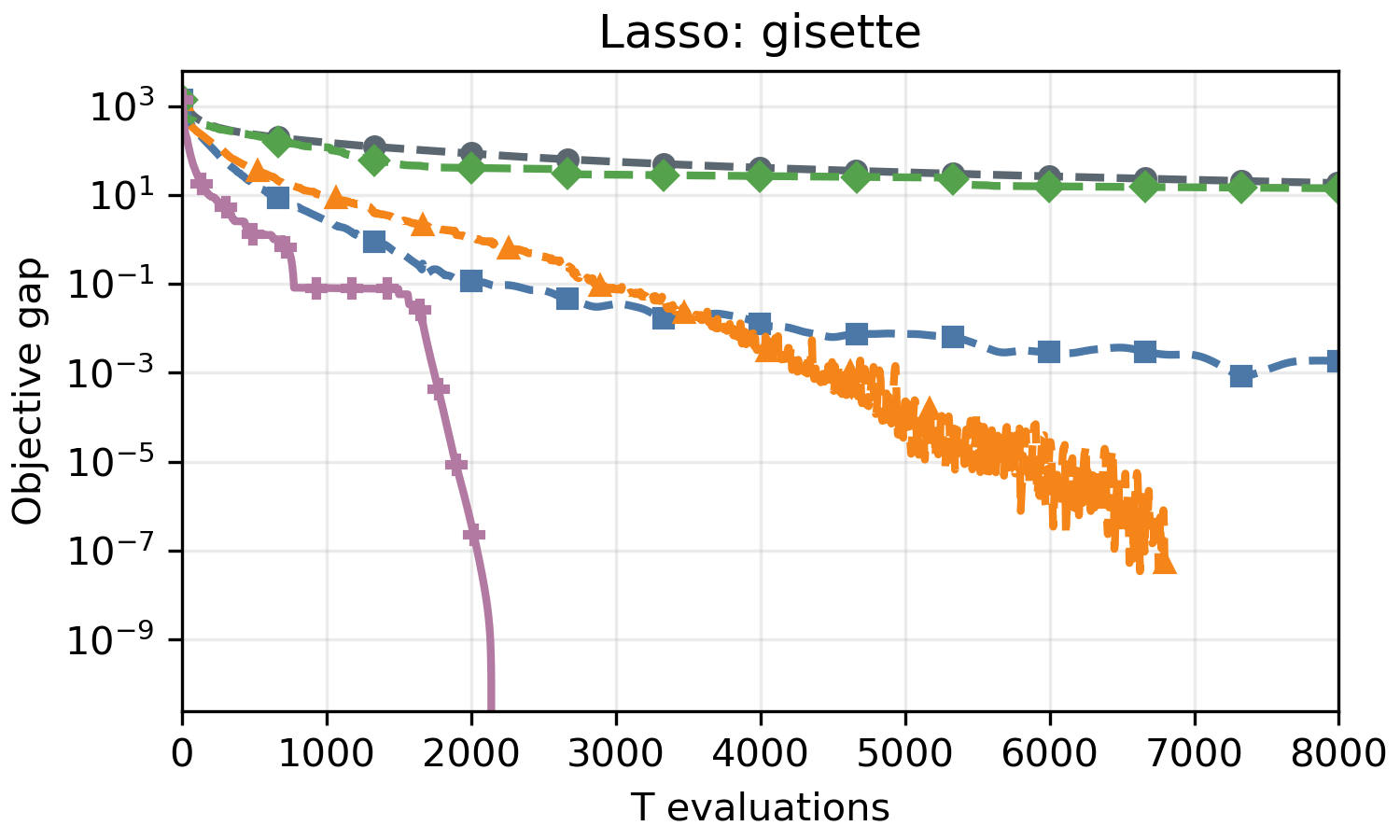}}
    \end{center}\\
    \begin{center}
      \raisebox{0.0\height}{\includegraphics[width=\linewidth]{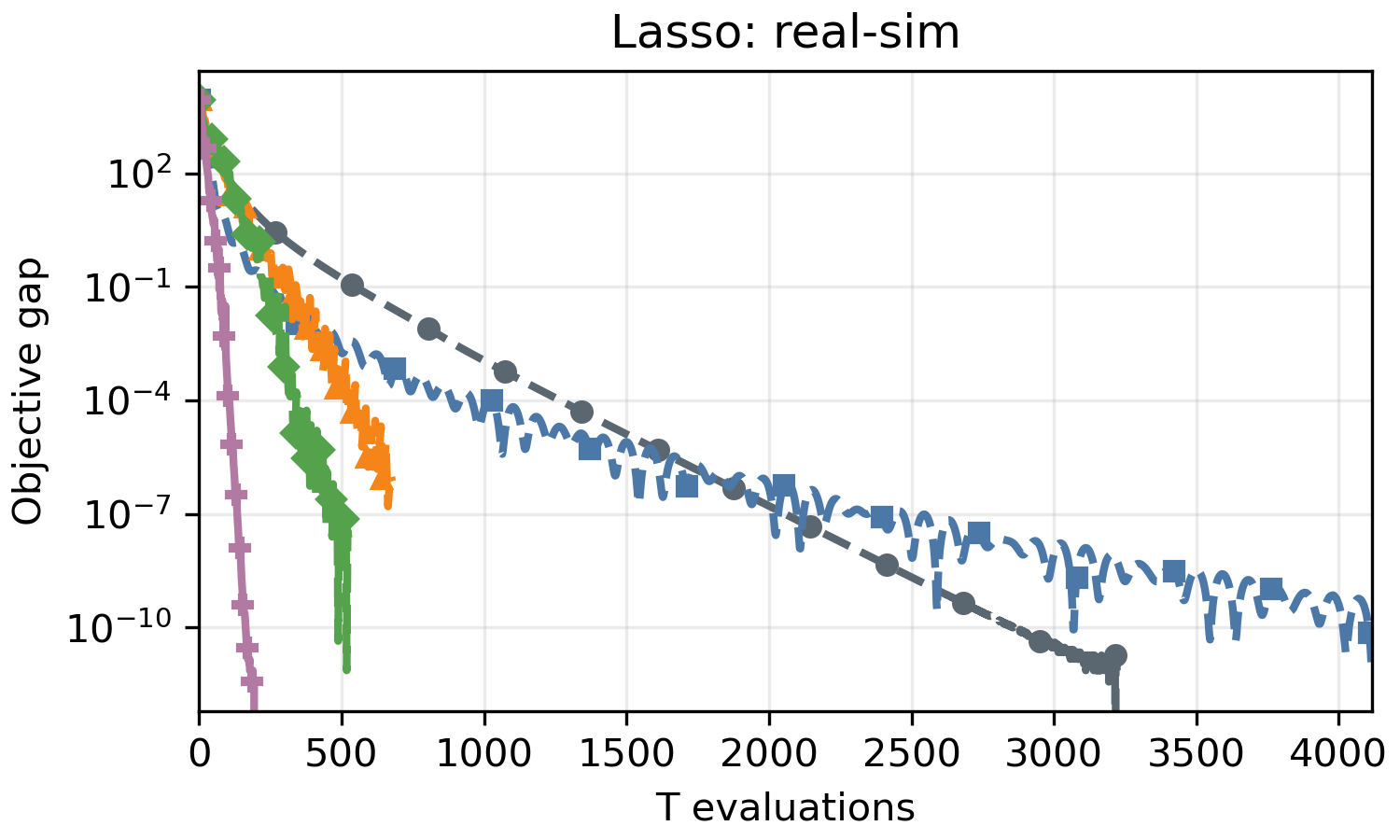}}
    \end{center} & \begin{center}
      \raisebox{0.0\height}{\includegraphics[width=\linewidth]{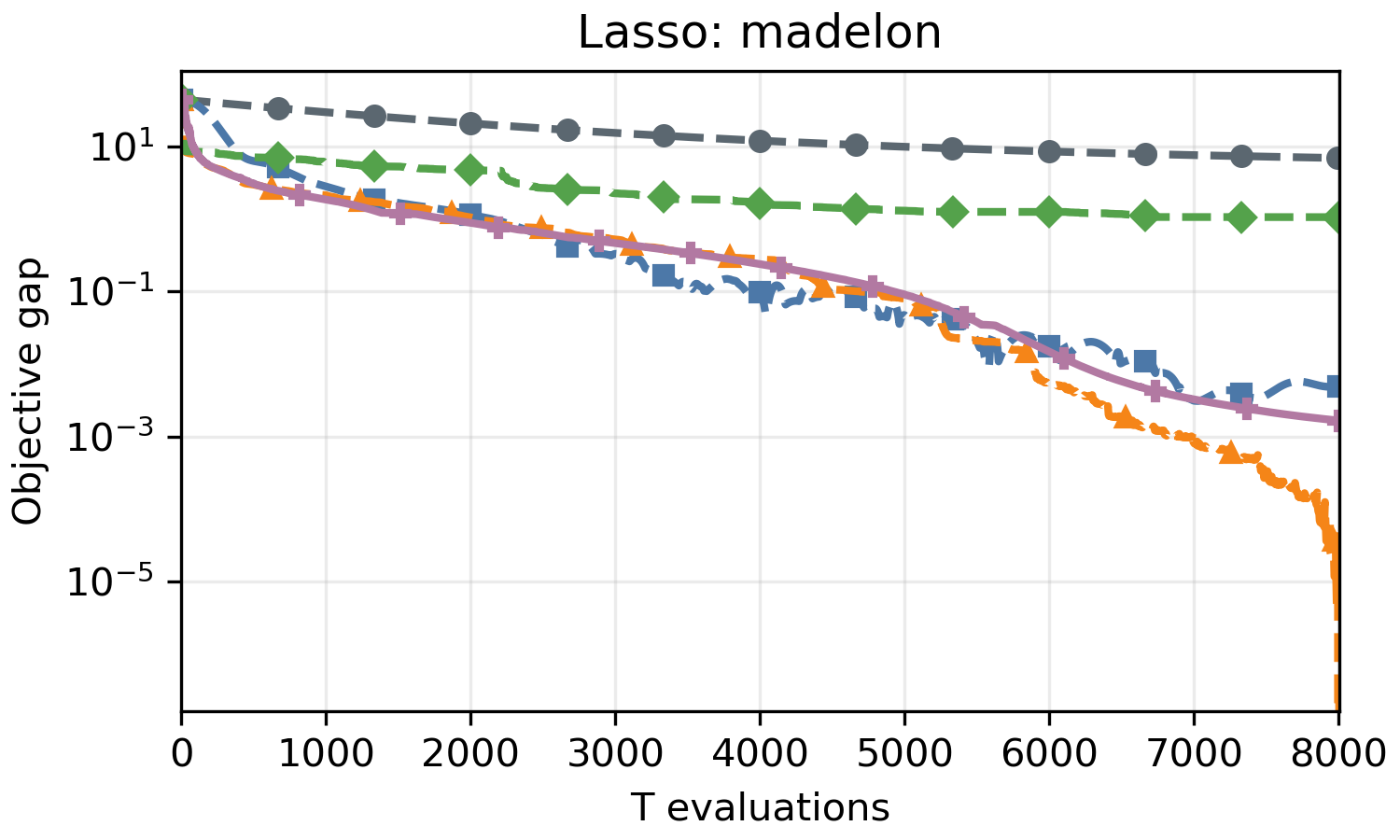}}
    \end{center} & \begin{center}
      \raisebox{0.0\height}{\includegraphics[width=\linewidth]{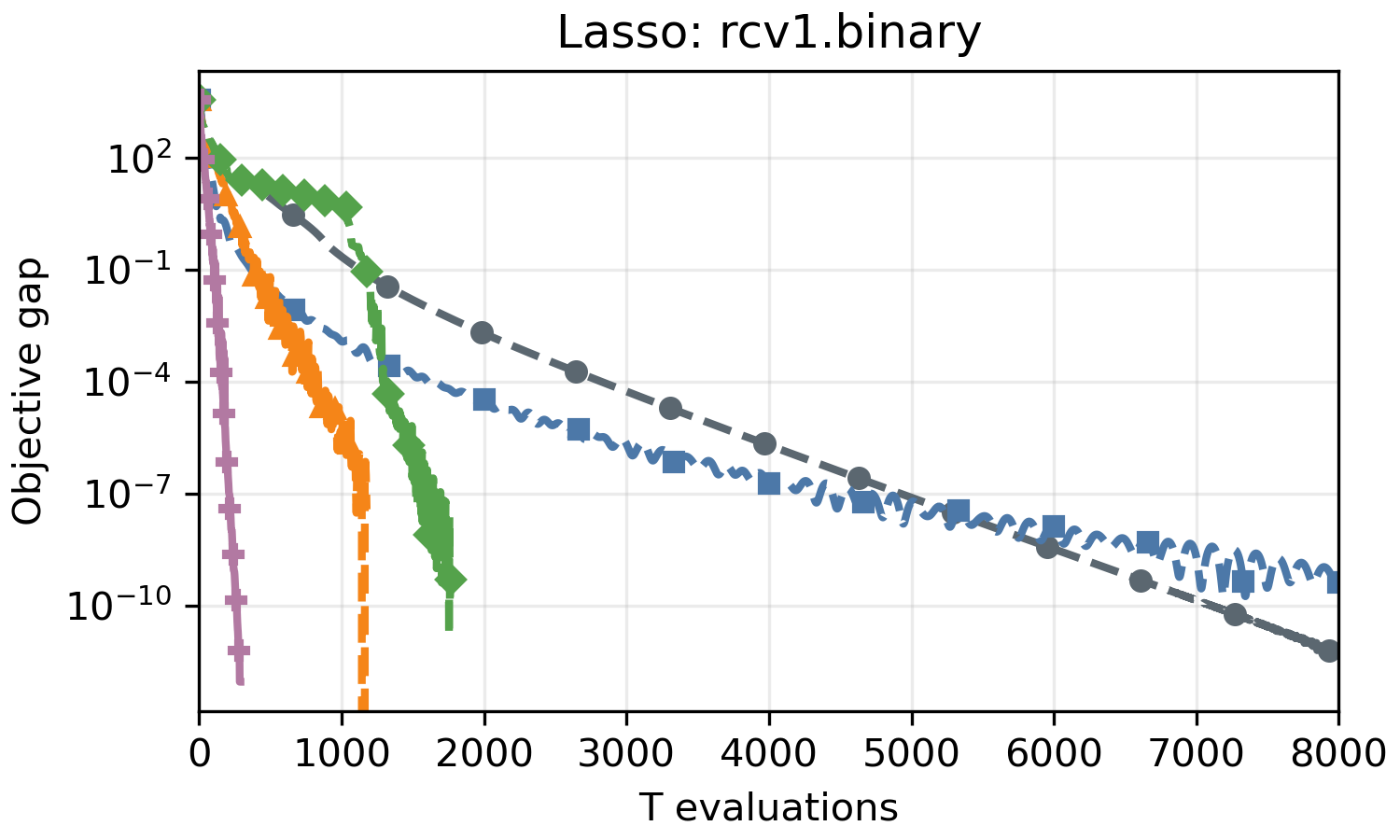}}
    \end{center}\\
  \end{tabularx}
  \caption{\label{fig:lasso}Convergence plots for Lasso problems}
\end{figure}
\begin{figure}[h]
  {\noindent}\begin{tabularx}{1.0\textwidth}{@{}X@{}@{}X@{}@{}X@{}}
    \begin{center}
      \raisebox{0.0\height}{\includegraphics[width=\linewidth]{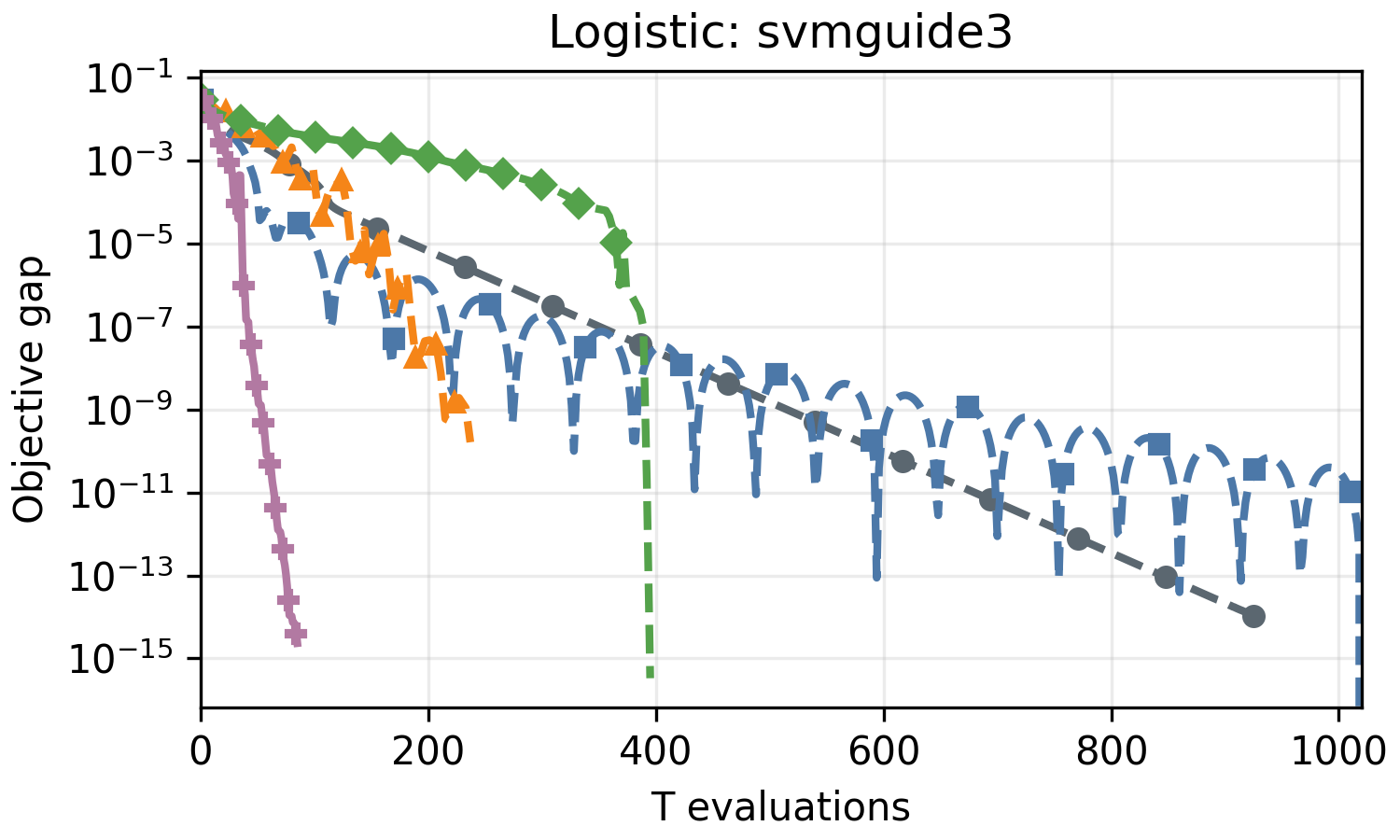}}
    \end{center} & \begin{center}
      \raisebox{0.0\height}{\includegraphics[width=\linewidth]{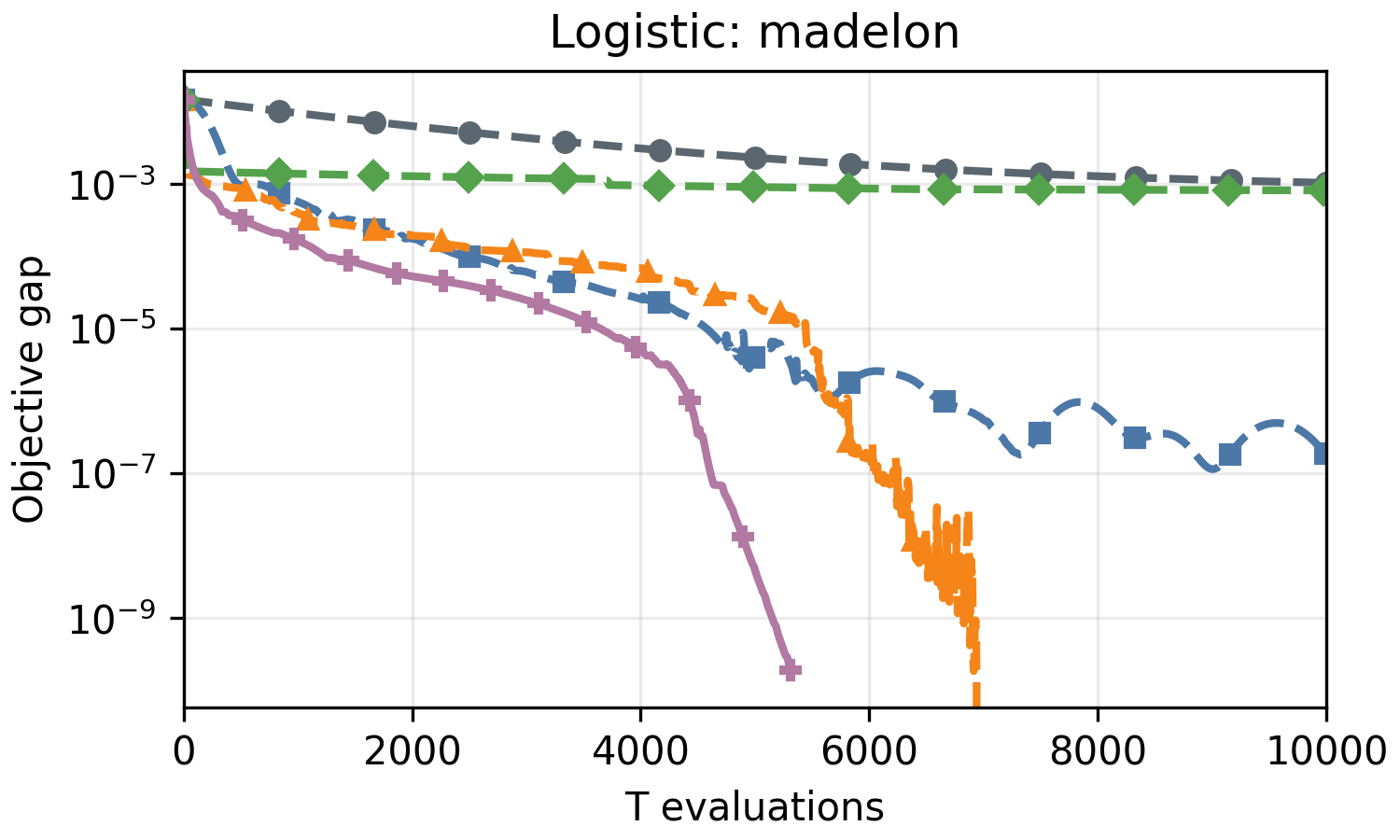}}
    \end{center} & \begin{center}
      \raisebox{0.0\height}{\includegraphics[width=\linewidth]{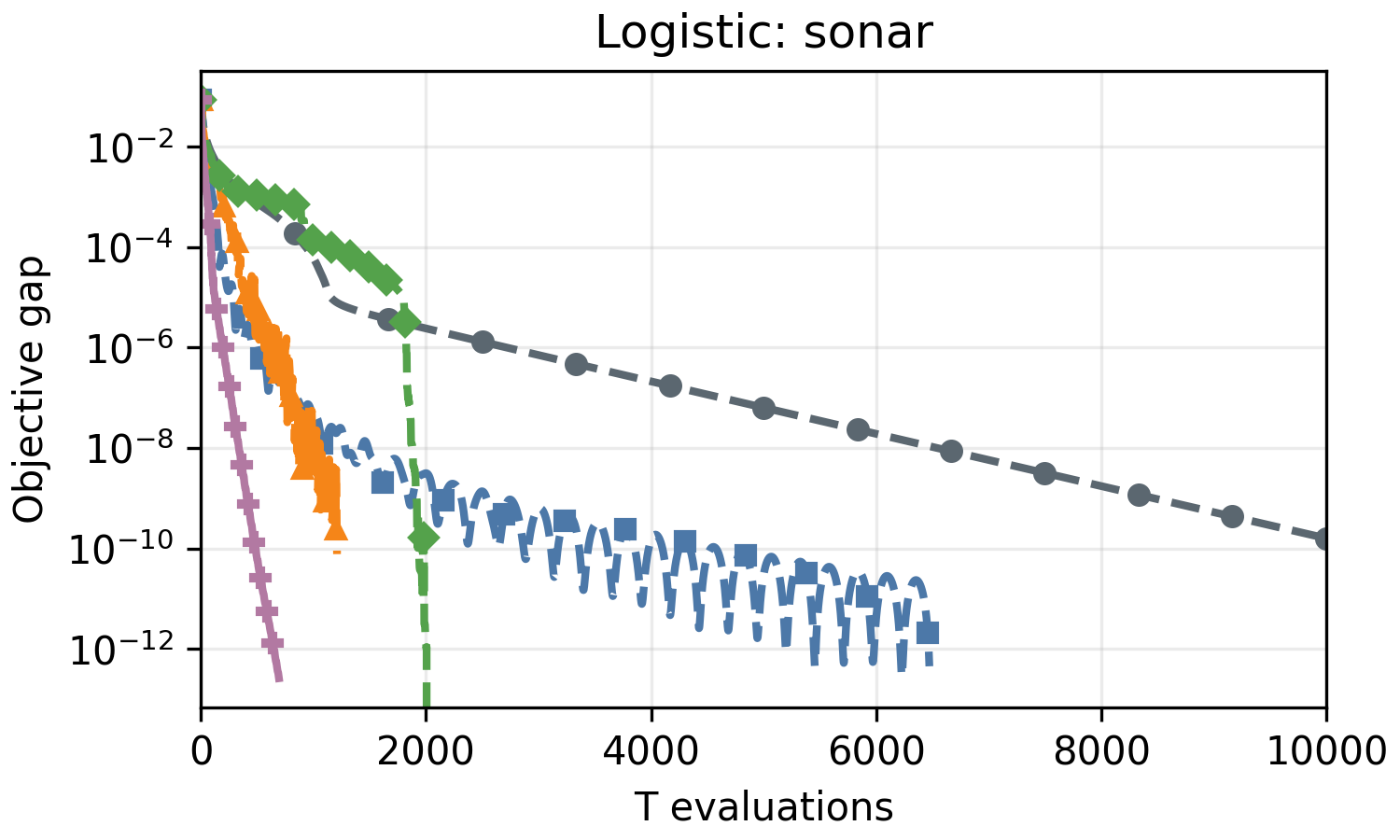}}
    \end{center}\\
    \begin{center}
      \raisebox{0.0\height}{\includegraphics[width=\linewidth]{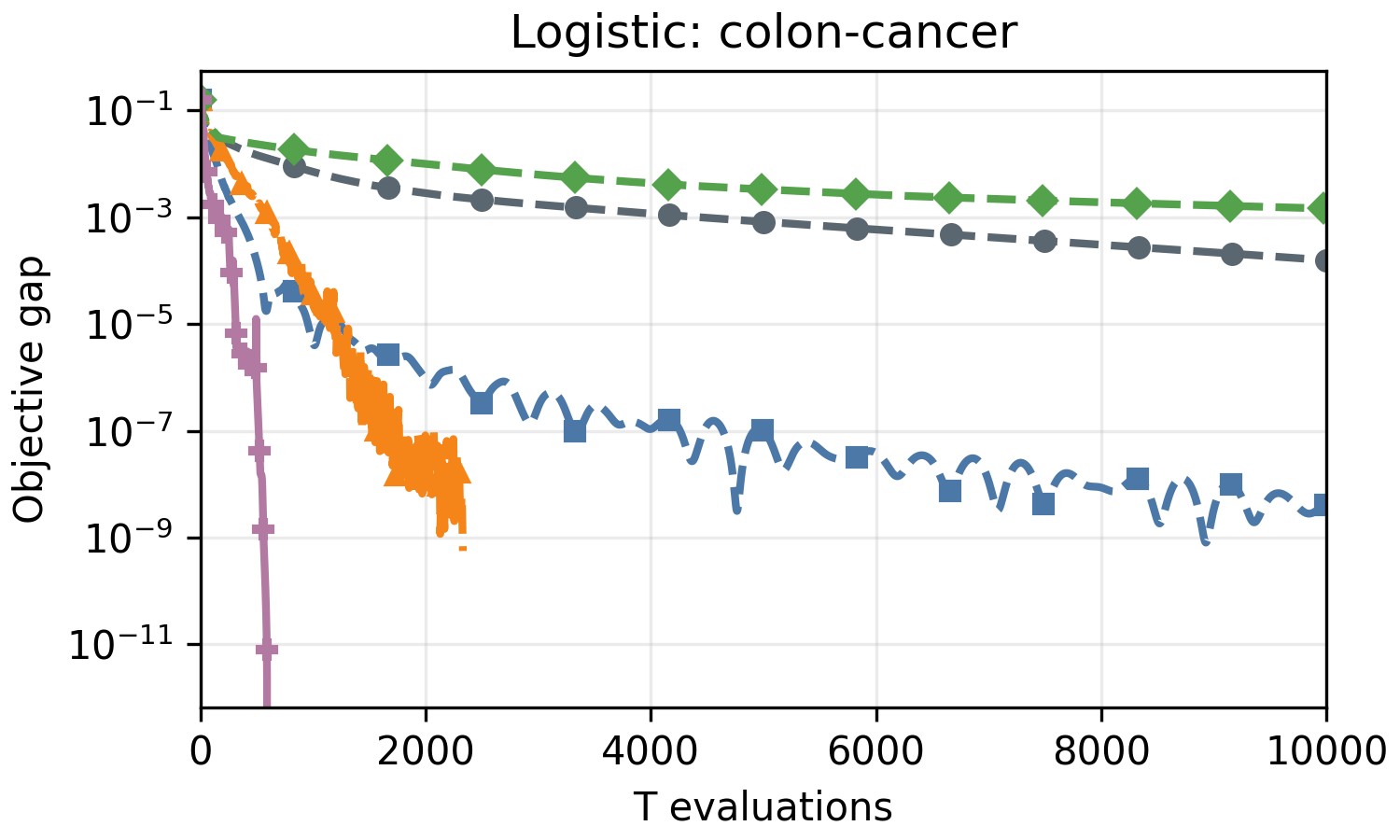}}
    \end{center} & \begin{center}
      \raisebox{0.0\height}{\includegraphics[width=\linewidth]{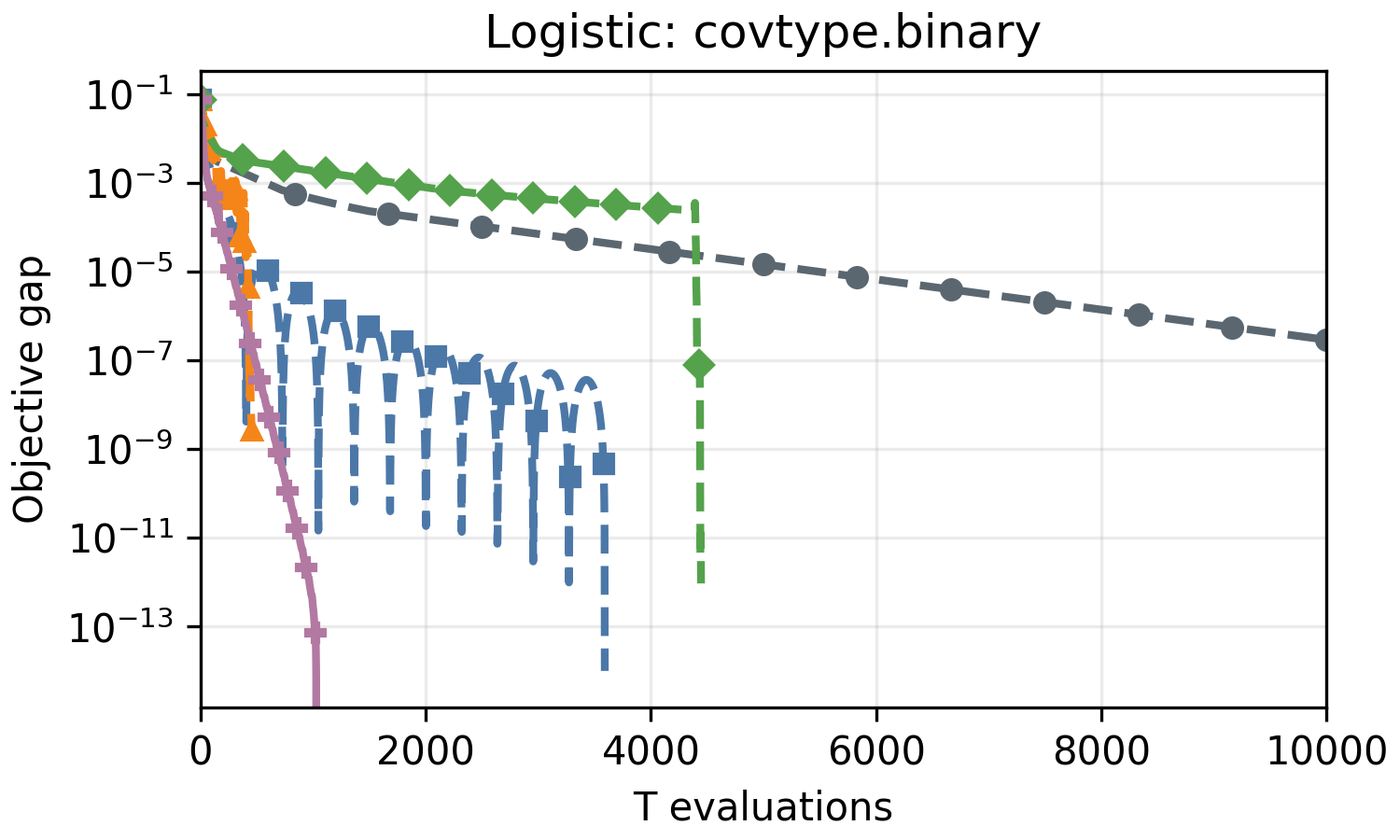}}
    \end{center} & \begin{center}
      \raisebox{0.0\height}{\includegraphics[width=\linewidth]{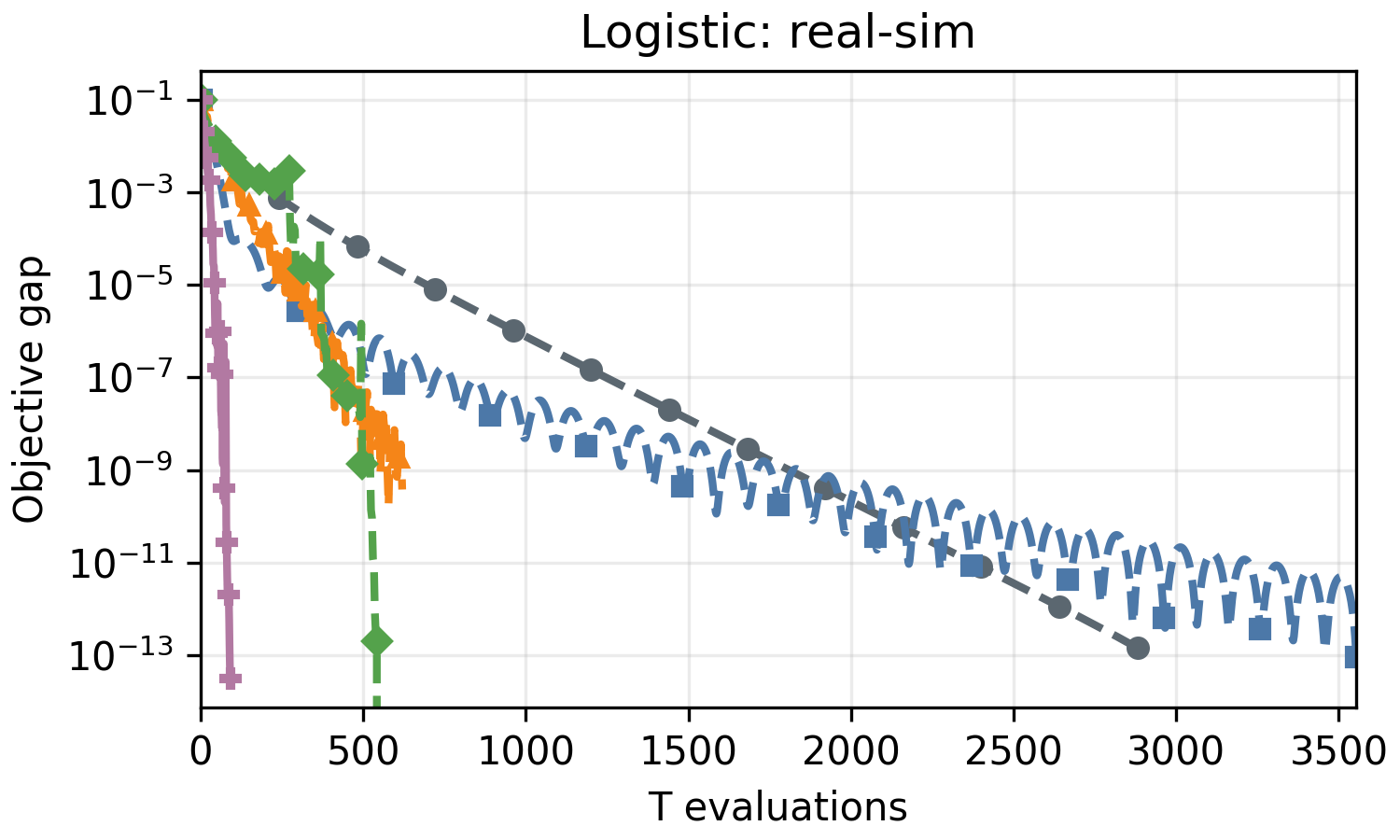}}
    \end{center}\\
    \multicolumn{3}{c}{\raisebox{0.0\height}{\includegraphics[width=0.6\textwidth]{performance-profile-legend.png}}}
  \end{tabularx}
  \caption{\label{fig:logistic}Convergence plots for $\ell_1$-regularized logistic regression
  problems}
\end{figure}

\subsection{Quadratic programming}

We next benchmark OS-ADMM (\Cref{alg:os-admm}) on QPs.
\Cref{app:os-admm-qp} gives pseudocode that specializes
\Cref{alg:os-admm} to QPs. This algorithm solves two linear systems per
iteration: the second solve computes the hypergradient. Consequently,
\Cref{alg:os-admm} can be time-consuming in practice. We therefore use the
efficient OS-ADMM variant in \Cref{app:efficient-os-admm-qp}, which approximates
the preconditioner hypergradient without an additional linear solve. We also
test whether combining online
scaling with other acceleration techniques improves performance.

\textbf{Benchmark algorithms.} We use the following four benchmark algorithms.

\begin{itemizedot}
  \item \tmtextbf{ADMM}. Vanilla ADMM with tuned $\gamma$.

  \item \tmtextbf{Fast-ADMM} {\cite{goldstein2014fast}}. ADMM with Nesterov-type momentum.
  
  \item \tmtextbf{RB-ADMM} {\cite{he2000alternating}}. Residual-balancing ADMM, which adjusts the penalty parameter $\rho$ using the current residuals. We update $\rho$ every 20 iterations.
  
  \item \tmtextbf{AADMM} {\cite{xu2017adaptive}}. Adaptive ADMM with a penalty-adaptation rule motivated by a BB step. We update $\rho$ every 20 iterations.
  
\end{itemizedot}

We apply online scaling to each benchmark algorithm to obtain \textbf{OS-ADMM},
\textbf{OS-Fast-ADMM}, \textbf{OS-RB-ADMM}, and \textbf{OS-AADMM}, which are introduced in \Cref{app:practical-os-admm-variants}. For all OS
variants, we use an AdaGrad learning-rate scheduler and diagonal preconditioners
$P_k$.

\textbf{Datasets.} We use three synthetic QP classes adapted from the OSQP
benchmark suite \cite{stellato2020osqp} and four Maros--M{\'e}sz{\'a}ros
instances. The synthetic classes are random sparse QPs, portfolio optimization
QPs, and constrained optimal-control QPs. Each class contains a $3 \times 3
\times 3$ grid, giving 27 instances. The random class varies the dimension,
constraint-to-variable ratio, and nonzero density. The portfolio class varies
the number of assets, factor ratio, and factor-loading density. The control
class varies the state dimension, horizon length, and input-to-state ratio.
\Cref{app:qp-benchmarks} gives the formulations, random-data distributions, and
grid values.

For Maros--M{\'e}sz{\'a}ros, we report four instances on which RB-ADMM or
AADMM takes between roughly 10 minutes and 1 hour to solve or reaches the 1-hour time
limit: \texttt{cont-050}, \texttt{qship08l}, \texttt{qship08s}, and
\texttt{qshell}.
For the OSQP synthetic profile suite, we use the generated problem scaling
without Ruiz equilibration. The selected Maros--M{\'e}sz{\'a}ros comparison
uses 10 Ruiz equilibration iterations.

We tune the hyperparameters for each QP class on the OSQP benchmark suite and use the best-tuned parameters for each method.
Each run uses a 3600-second time limit and stops when the relative KKT
residual drops below the target tolerance $10^{-8}$.
{The relative KKT residual is defined in
\Cref{app:qp-benchmarks}.}
We report the relative KKT residual against the number of linear solves for the
aggregate profile summary and use wall-clock time for the selected
Maros--M{\'e}sz{\'a}ros instance comparison.

\begin{figure}[htbp]
  {\noindent}\begin{tabularx}{1.0\textwidth}{@{}X@{}@{}X@{}@{}X@{}}
    \begin{center}
      \raisebox{0.0\height}{\includegraphics[width=\linewidth]{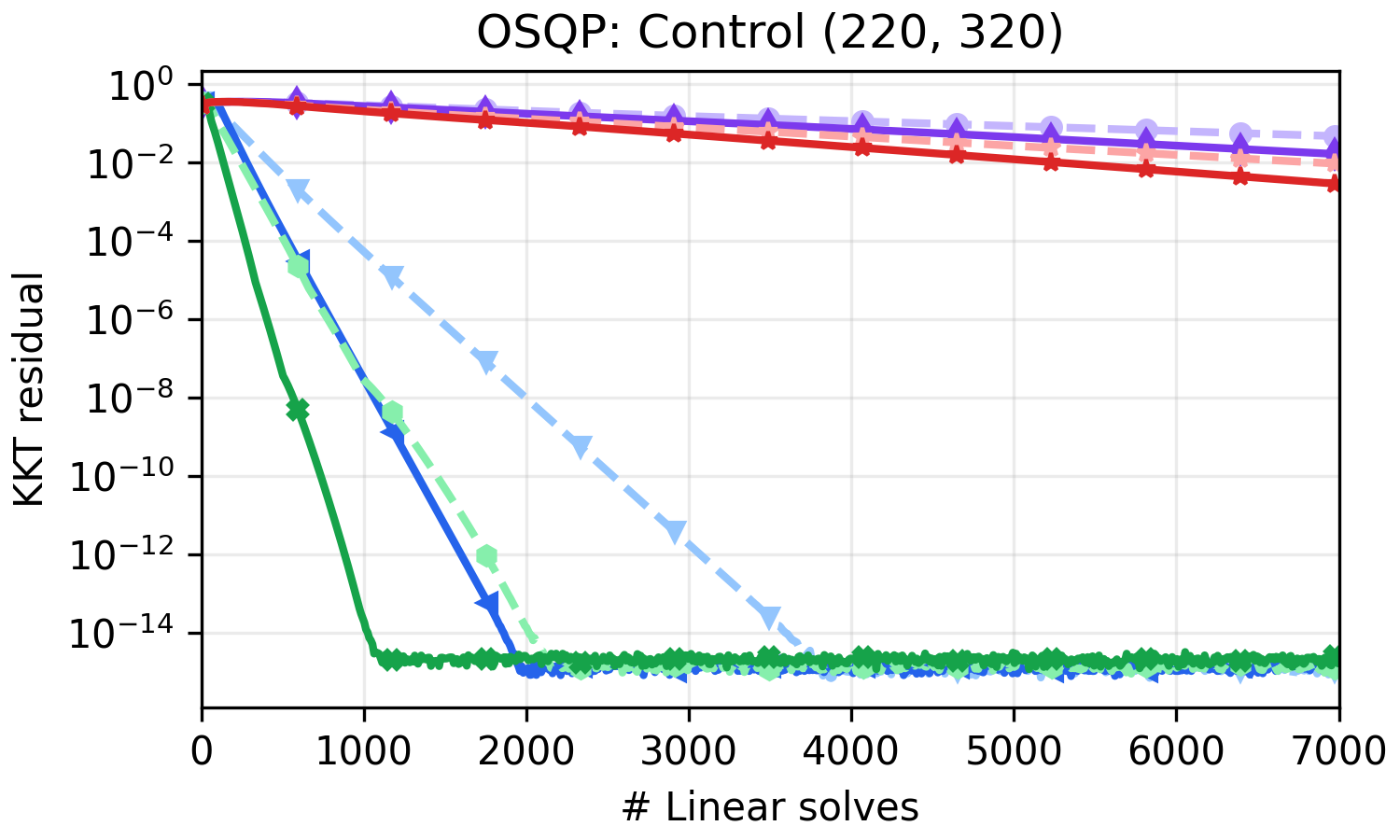}}
    \end{center} & \begin{center}
      \raisebox{0.0\height}{\includegraphics[width=\linewidth]{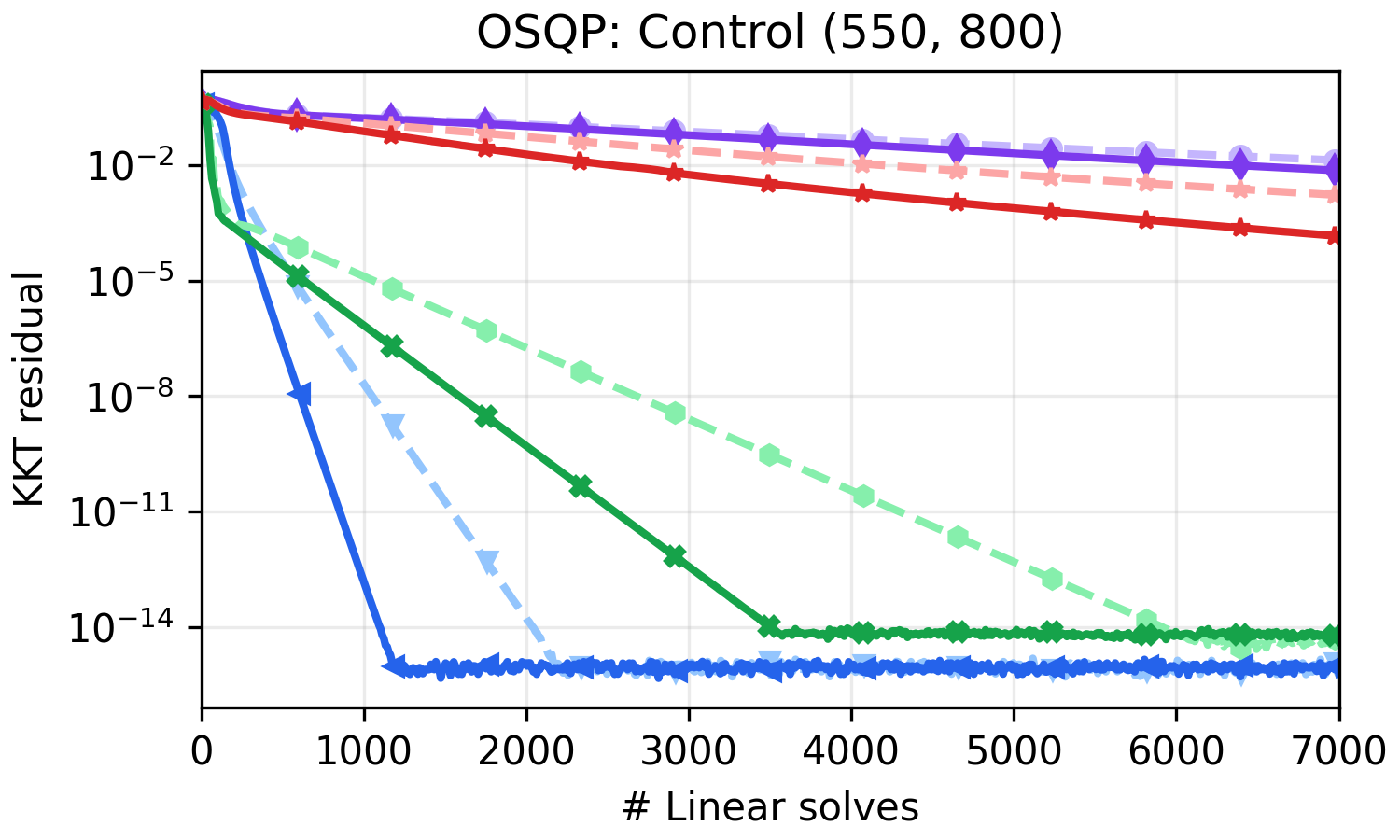}}
    \end{center} & \begin{center}
      \raisebox{0.0\height}{\includegraphics[width=\linewidth]{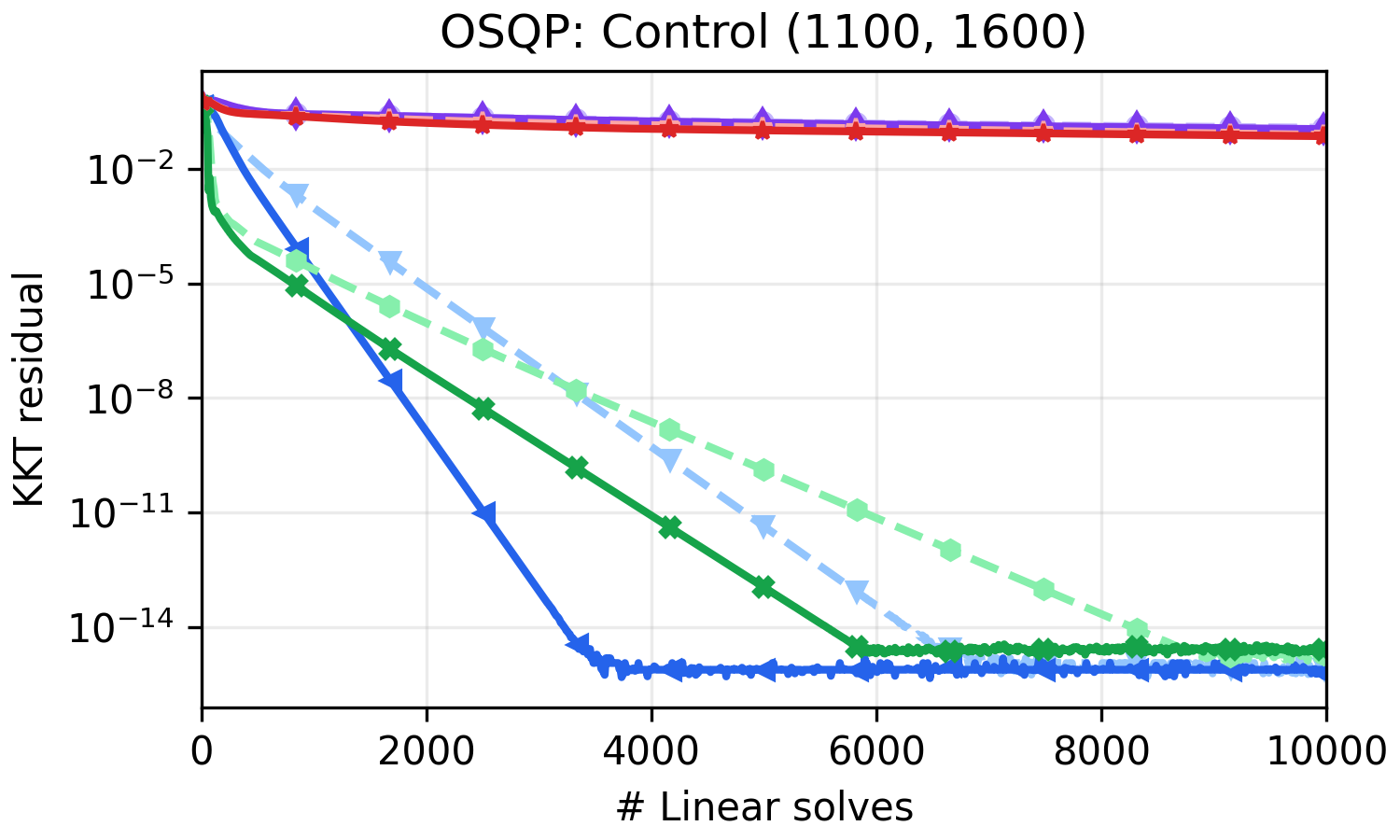}}
    \end{center}\\
    \begin{center}
      \raisebox{0.0\height}{\includegraphics[width=\linewidth]{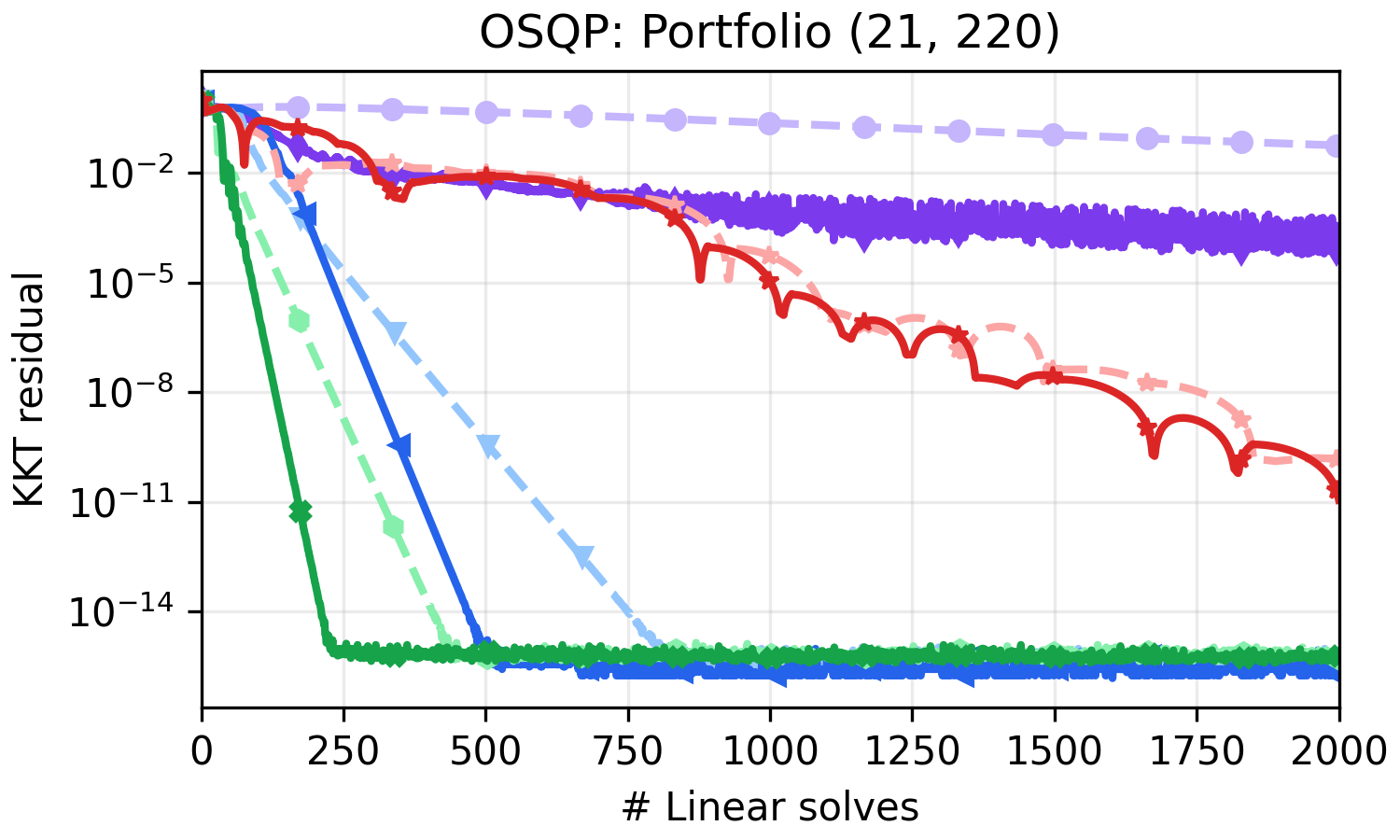}}
    \end{center} & \begin{center}
      \raisebox{0.0\height}{\includegraphics[width=\linewidth]{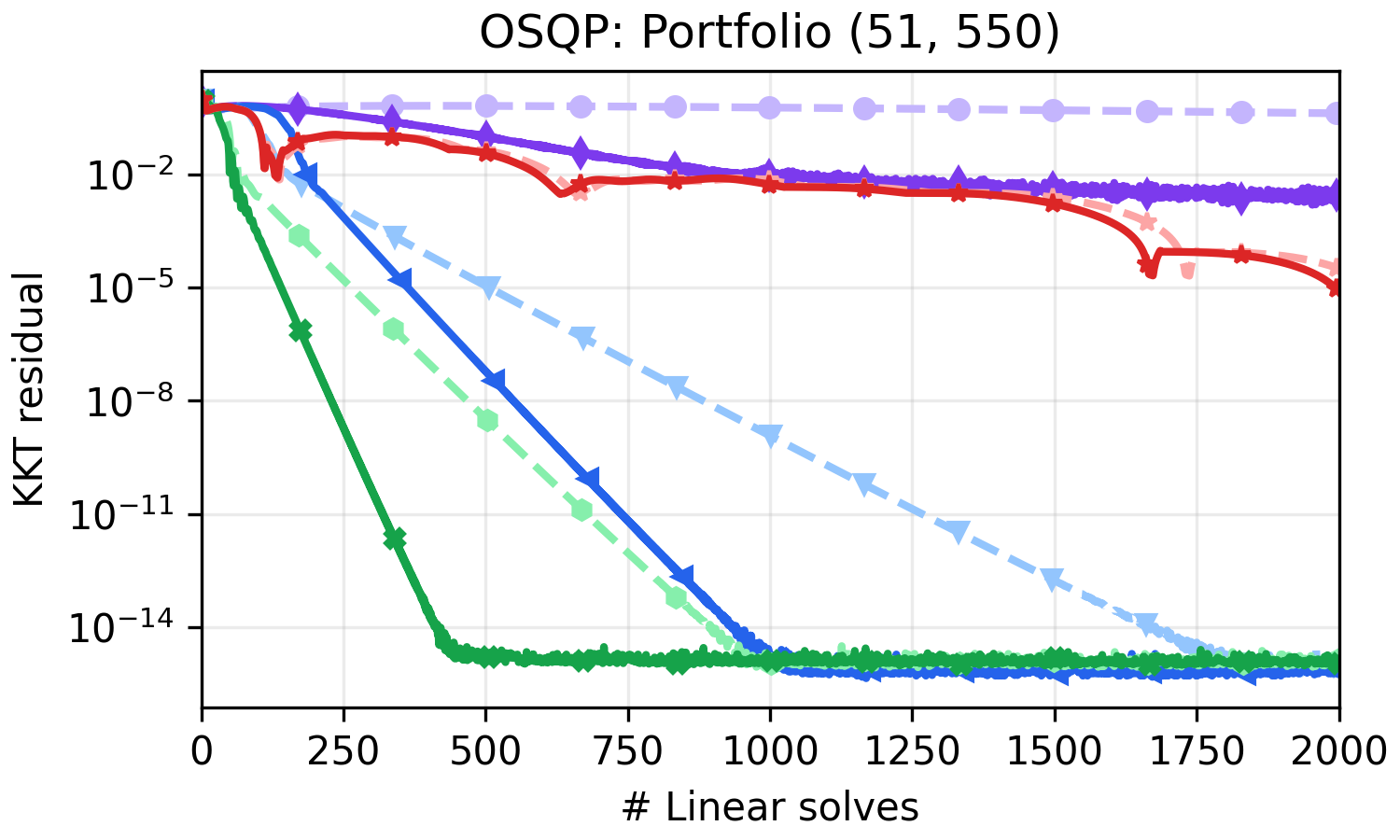}}
    \end{center} & \begin{center}
      \raisebox{0.0\height}{\includegraphics[width=\linewidth]{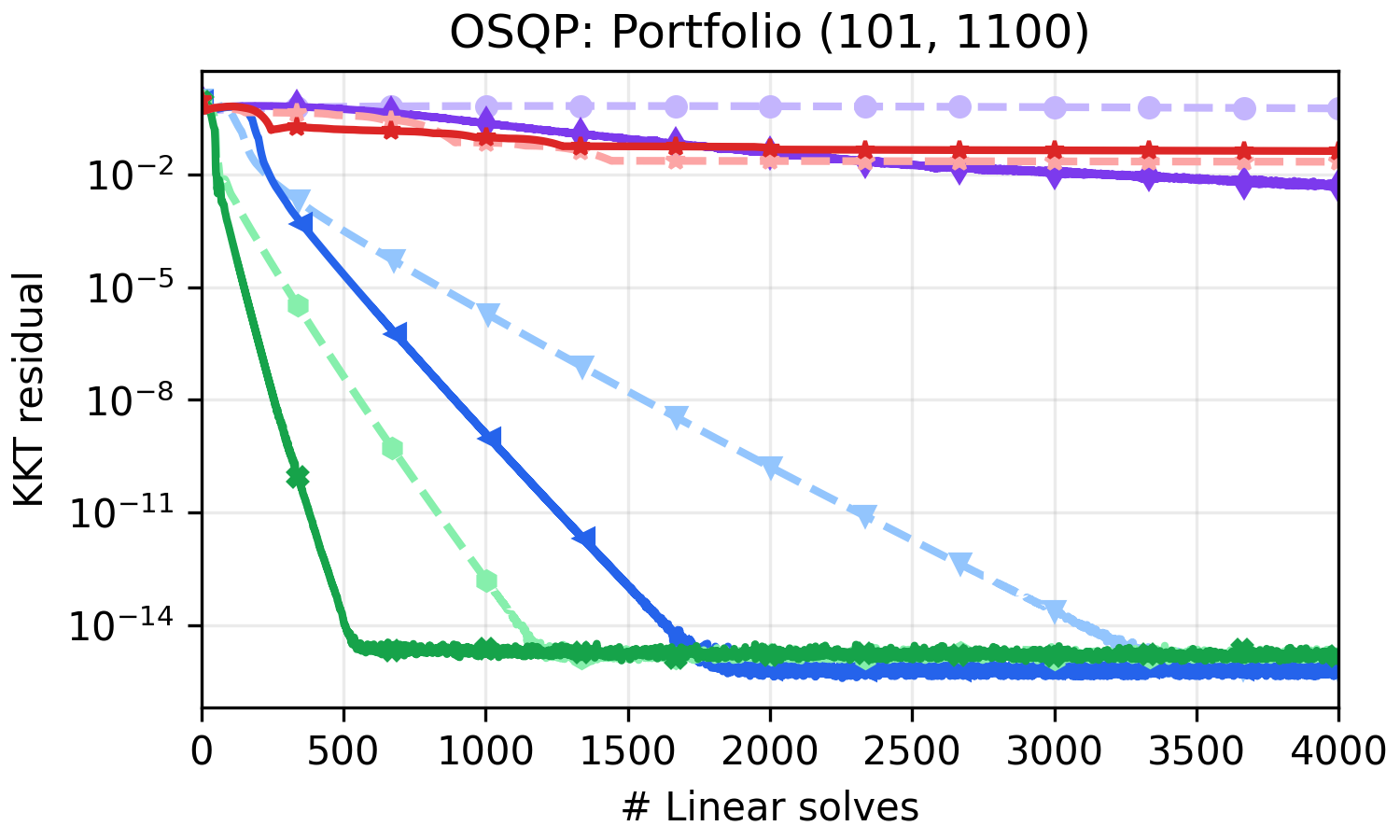}}
    \end{center}\\
    \multicolumn{3}{c}{\raisebox{0.0\height}{\includegraphics[width=0.95\textwidth]{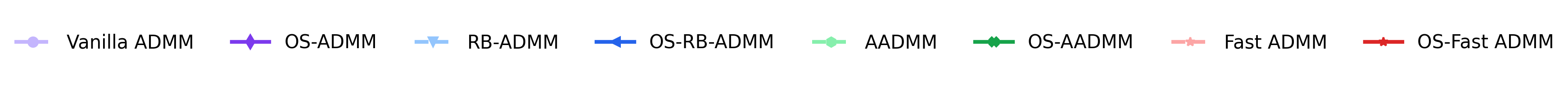}}}
  \end{tabularx}
  \caption{\label{fig:qp-osqp-convergence}Convergence plots for OSQP control
  and portfolio QP benchmark instances. Control uses horizon $N = 10$ and
  input ratio $n_u / n_x = 0.5$. Portfolio uses factor ratio
  $k / n = 0.1$ and factor-loading density $d = 0.5$}
\end{figure}

\begin{table*}[htbp]
\centering
\caption{Summary of QP methods on the OSQP benchmark suite.}
\label{tab:qp-profile-summary}
\resizebox{\textwidth}{!}{\begin{tabular}{ccccc|ccc}
\toprule
\multirow{2}{*}{Problem} & \multirow{2}{*}{Method} & \multicolumn{3}{c}{Tolerance $10^{-6}$} & \multicolumn{3}{c}{Tolerance $10^{-8}$} \\
\cmidrule(lr){3-5}\cmidrule(lr){6-8}
& & Solved & SGM10 solves & SGM10 time (s) & Solved & SGM10 solves & SGM10 time (s) \\
\midrule
\multirow{4}{*}{\shortstack{OSQP\\Random (27)}} & RB-ADMM & 27 & 17836.07 & 87.12 & 24 & 28409.83 & 110.49 \\
 & OS-RB-ADMM & 27 & 8874.29 & 46.04 & 26 & 13901.65 & 56.54 \\
 & AADMM & 27 & 3976.47 & 29.34 & 27 & 6664.83 & 41.69 \\
 & OS-AADMM & 27 & 2198.51 & 17.13 & 27 & 3610.24 & 23.62 \\
\midrule
\multirow{4}{*}{\shortstack{OSQP\\Portfolio (27)}} & RB-ADMM & 27 & 5832.27 & 108.11 & 25 & 7820.31 & 110.82 \\
 & OS-RB-ADMM & 27 & 3187.53 & 55.85 & 27 & 4169.27 & 55.29 \\
 & AADMM & 27 & 1081.07 & 30.14 & 27 & 1464.33 & 30.89 \\
 & OS-AADMM & 27 & 597.53 & 19.60 & 27 & 834.53 & 21.17 \\
\midrule
\multirow{4}{*}{\shortstack{OSQP\\Control (27)}} & RB-ADMM & 27 & 3932.18 & 16.99 & 27 & 5905.00 & 23.11 \\
 & OS-RB-ADMM & 27 & 1886.86 & 10.07 & 27 & 2780.25 & 13.56 \\
 & AADMM & 27 & 11863.81 & 41.83 & 27 & 18687.48 & 60.58 \\
 & OS-AADMM & 27 & 5746.71 & 26.34 & 27 & 8948.08 & 37.77 \\
\bottomrule
\end{tabular}
}
\vspace{2pt}
\end{table*}

\begin{table*}[htbp]
\centering
\caption{\label{tab:qp-maros-large-times}Solution times in seconds for the four
selected Maros--M{\'e}sz{\'a}ros QP benchmark instances. Here, $m$ is the number
of constraints, $n$ is the number of variables, \texttt{t} denotes a timeout,
and numbers in parentheses denote iteration counts.}
\resizebox{\textwidth}{!}{\begin{tabular}{clccccc}
\toprule
Tolerance & Instance & $(m,n)$ & RB-ADMM & OS-RB-ADMM & AADMM & OS-AADMM \\
\midrule
\multirow{4}{*}{$10^{-6}$} & \texttt{cont-050} & $(2401,2597)$ & \texttt{t} & 1633.2 (109010) & 2017.9 (140500) & 1028.8 (68390) \\
 & \texttt{qship08l} & $(778,4283)$ & 2842.3 (82660) & 843.6 (26640) & 1481.1 (44470) & 1390.8 (41130) \\
 & \texttt{qship08s} & $(778,2387)$ & 888.9 (68790) & 448.9 (37200) & 441.6 (35910) & 381.5 (30380) \\
 & \texttt{qshell} & $(536,1775)$ & \texttt{t} & 2946.7 (570950) & 500.9 (89220) & 473.6 (81280) \\
\midrule
\multirow{4}{*}{$10^{-8}$} & \texttt{cont-050} & $(2401,2597)$ & \texttt{t} & 2370.3 (160690) & 2981.3 (207570) & 1516.6 (101020) \\
 & \texttt{qship08l} & $(778,4283)$ & \texttt{t} & 1171.2 (37020) & 1674.9 (50300) & 1569.3 (46400) \\
 & \texttt{qship08s} & $(778,2387)$ & 1333.4 (102990) & 678.9 (55710) & 456.7 (37170) & 395.8 (31520) \\
 & \texttt{qshell} & $(536,1775)$ & \texttt{t} & \texttt{t} & 512.9 (91360) & 486.3 (83510) \\
\bottomrule
\end{tabular}
}
\vspace{2pt}
\end{table*}

\textbf{Results.} 
\Cref{fig:qp-osqp-convergence} shows the convergence plots for all eight
benchmark methods on the OSQP control and portfolio problems. On both problem
classes, the online-scaled variants uniformly outperform their unscaled
counterparts. These results suggest that online scaling can complement
techniques such as momentum and adaptive penalty selection.

Vanilla ADMM and fast ADMM converge much more slowly than RB-ADMM and AADMM and
solve only a small fraction of the benchmark problems. We therefore exclude
vanilla ADMM, fast ADMM, and their online-scaled variants from the subsequent
performance-profile comparison. We defer the convergence plots for the
remaining four methods to the appendix.

\Cref{tab:qp-profile-summary} reports shifted geometric means with a shift of
10 (SGM10) for the number of linear solves and the solution time across the
three OSQP benchmark classes. The table excludes vanilla ADMM, fast ADMM, and
their online-scaled variants because they converge slowly and rarely solve the
benchmark instances within the iteration limit. Among the four reported
methods, the online-scaled variants solve more instances within the budget at
the target tolerance. The OS variants also substantially reduce the SGM10 value
for the number of linear solves, nearly halving this metric on the synthetic
portfolio problems.

\Cref{tab:qp-maros-large-times} reports solution times for the four selected
Maros--M{\'e}sz{\'a}ros instances on which RB-ADMM or AADMM has difficulty
reaching the tolerance $10^{-8}$ within 3600 seconds. On these instances, the
OS variants reduce the solution times of AADMM, roughly halve those of RB-ADMM,
and solve instances on which RB-ADMM fails.

\section{Conclusion}

In this paper, we develop an online scaling technique for accelerating the convergence of FBS, DRS, and ADMM.
We prove trajectory-dependent global convergence on convex problems and local superlinear convergence after active-set identification.
The numerical experiments on Lasso, logistic regression, and QP problems demonstrate the effectiveness of the proposed methods.

\bibliographystyle{plain}
\bibliography{ref}

\newpage
\appendix
\addcontentsline{toc}{section}{Appendix}
\appendixtableofcontents

\crefalias{section}{appendix}
\crefalias{subsection}{appendix}
\appendixsection{Proof of DRS--ADMM translation
(\Cref{lem:drs-admm})}\label{app:drs-admm}

\begin{proof}
  Let $\nu$ and $c$ be any vectors, and suppose
  \[ \tilde{z} \in \arg \min_z \left\{ \psi (z) + \langle \nu, B z \rangle
     + \tfrac{\gamma}{2} \| c + B z - b \|^2 \right\}, \qquad
     \tilde{\nu} = \nu + \gamma (c + B \tilde{z} - b) . \]
  The first-order condition of the $\tilde{z}$-subproblem and the definition of
  $\lambda$ give
  \begin{equation}
    - B^{\top} \tilde{\nu} \in \partial \psi (\tilde{z}), \qquad
    \lambda (\tilde{z}, \tilde{\nu}) = \tilde{\nu} - \gamma B \tilde{z}
    = \nu + \gamma (c - b). \label{eq:admm-z-update}
  \end{equation}
  The initial pair satisfies $- B^{\top} \nu^0 \in \partial \psi (z^0)$ by
  \Cref{alg:os-admm}. Formula~\eqref{eq:admm-z-update} shows that each probe
  pair satisfies the same inclusion. Since the null step retains either the
  current or the probe pair, induction gives
  \[ - B^{\top} \nu \in \partial \psi (z) \]
  for every current and probe pair in the proposition.

  Fix either pair $(z, \nu)$ and abbreviate $\lambda = \lambda (z, \nu)$,
  $x^+ = x^+ (z, \nu)$, and $r_x = r_x (z, \nu)$. The corresponding
  $x$-subproblem attains the minimizer $x^+$ by assumption. Thus $\lambda = \nu
  - \gamma B z$ and $r_x = A x^+ + B z - b$.

  We first prove parts (1) and (2).
  For DRS, write 
  \( p (\lambda) \assign \tmop{prox}_{\gamma f} (\lambda), 
     q (\lambda) \assign \tmop{prox}_{\gamma g} (2 p (\lambda) - \lambda), \)
  and $T (\lambda) = \lambda + q (\lambda) - p (\lambda)$. We use the
  conjugate chain-rule inclusions $- B \partial \psi^{\ast} (- B^{\top} u)
  \subseteq \partial f (u)$ and $- A \partial \phi^{\ast} (- A^{\top} u) + b
  \subseteq \partial g (u)$, which hold without a qualification condition.

  The state inclusion is equivalent to $z \in \partial \psi^{\ast} (-
  B^{\top} \nu)$. Hence
  \[ \tfrac{1}{\gamma} (\lambda - \nu) = - B z \in - B \partial
     \psi^{\ast} (- B^{\top} \nu) \subseteq \partial f (\nu), \]
  which is the proximal optimality condition for $p (\lambda) = \nu$. The
  $x$-subproblem gives $- A^{\top} (\nu + \gamma r_x) \in \partial \phi
  (x^+)$, or equivalently $x^+ \in \partial \phi^{\ast} (- A^{\top} (\nu
  + \gamma r_x))$. Using $p (\lambda) = \nu$ and $r_x = A x^+ + B z - b$,
  \[ \tfrac{1}{\gamma} \big(2 p (\lambda) - \lambda - (\nu + \gamma r_x)\big)
     = B z - r_x = b - A x^+ \in \partial g (\nu + \gamma r_x). \]
  This proximal optimality condition gives
  \begin{equation}
    p (\lambda) = \nu, \qquad q (\lambda) = \nu + \gamma r_x, \qquad
    \lambda - T (\lambda) = p (\lambda) - q (\lambda) = - \gamma r_x.
    \label{eq:admm-prox-identities}
  \end{equation}
  The last identity at $(z^k, \nu^k)$ proves Part (1). Applying
  formula~\eqref{eq:admm-z-update} with $c = c_{P_k} = b - B z^k + P_k r_x^k$
  and then using~\eqref{eq:admm-prox-identities} gives
  \[ \lambda (z^{k + 1 / 2}, \nu^{k + 1 / 2})
     = \nu^k + \gamma (c_{P_k} - b)
     = \lambda^k + \gamma P_k r_x^k
     = \lambda^k - P_k (\lambda^k - T (\lambda^k)), \]
  which proves Part (2).

  We then prove parts (3) and (4). Strong convexity and twice differentiability
  of $\psi$ imply that $\psi^{\ast}$, and hence $f$, is twice differentiable.
  The state inclusion gives $- B^{\top} \nu = \nabla \psi (z)$, so the
  inverse-function identity and~\eqref{eq:admm-prox-identities} yield
  \[ \nabla p (\lambda)
     = (I + \gamma \nabla^2 f (p (\lambda)))^{- 1}
     = (I + \gamma B [\nabla^2 \psi (z)]^{- 1} B^{\top})^{- 1}
     \backassign J_z. \]
  Substituting this formula and~\eqref{eq:admm-prox-identities} into the DRE
  gradient identity~\eqref{eq:dre} gives
  \[ \nabla F_{\gamma}^{\tmop{ADMM}} (\lambda)
     = \tfrac{1}{\gamma} (2 J_z - I) (p (\lambda) - q (\lambda))
     = - (2 J_z - I) r_x, \]
  which proves Part (3).

  For the envelope value, the same proximal identities give $\nabla
  f_{\gamma} (\lambda) = \tfrac{1}{\gamma} (\lambda - p (\lambda)) = - B z$
  and $q (\lambda) - p (\lambda) = \gamma r_x$. Fenchel--Young equality gives
  \[ \begin{aligned}
       f (\nu) &= - \langle \nu, B z \rangle - \psi (z), \\
       g (\nu + \gamma r_x)
       &= - \phi (x^+) - \langle \nu + \gamma r_x, A x^+ - b \rangle .
     \end{aligned} \]
  Substituting these expressions into the DRE value formula~\eqref{eq:dre} and
  using $A x^+ + B z - b = r_x$ gives
  \[ \begin{aligned}
       F_{\gamma}^{\tmop{ADMM}} (\lambda)
       &= - \phi (x^+) - \psi (z)
          - \langle \nu, A x^+ + B z - b \rangle \\
       &\quad {} - \gamma \langle r_x, A x^+ + B z - b \rangle
          + \tfrac{\gamma}{2} \| r_x \|^2 \\
       &= - \phi (x^+) - \psi (z) - \langle \nu, r_x \rangle
          - \tfrac{\gamma}{2} \| r_x \|^2 \\
       &= - \mathcal{L}_{\gamma} (x^+, z, \nu),
     \end{aligned} \]
  which proves Part (4).
\end{proof}

\appendixsection{Envelope smoothness and convexity}\label{app:envelope-proofs}

\appendixsubsection{Proof of \Cref{lem:s-env}}\label{app:proof-s-env}

\begin{proof}
  Suppose first that the smooth term $f$ satisfies the conditions in
  \Cref{ass:smooth}. Since $f$ is real-valued and convex on $\mathbb{R}^n$, it is proper,
  closed, and convex, so the Moreau-envelope facts of \Cref{subsec:moreau}
  apply: $f_{\gamma} : \mathbb{R}^n \to \mathbb{R}$ is convex and
  continuously differentiable with
  \begin{equation}
    \nabla f_{\gamma} (z) = \tfrac{1}{\gamma} (z - \tmop{prox}_{\gamma f}
    (z)) = \nabla f (\tmop{prox}_{\gamma f} (z)) . \label{eq:senv-grad}
  \end{equation}
  We first show that $f_{\gamma}$ is twice differentiable on
  $\mathbb{R}^n$. The map $F \assign \mathrm{id} + \gamma \nabla f$ is
  differentiable with nonsingular Jacobian $I + \gamma \nabla^2 f (x)
  \succeq I$, since convexity of $f$ gives $\nabla^2 f \succeq 0$. Moreover,
  $F$ is a bijection of $\mathbb{R}^n$: for every $z$, the point
  $\tmop{prox}_{\gamma f} (z)$ solves $F (x) = z$, and the strong monotonicity
  $\langle F (x) - F (y), x - y \rangle \geq \| x - y \|^2$ rules out
  further solutions and makes $F^{-1} = \tmop{prox}_{\gamma f}$
  $1$-Lipschitz. Fix $z$ and write $x_z \assign F^{-1} (z)$. For an increment
  $h$, let $k \assign F^{-1} (z + h) - x_z$. Then $\| k \| \leq \| h \|$,
  and differentiability of $F$ at $x_z$ gives
  \[ h = (I + \gamma \nabla^2 f (x_z)) k + o (\| k \|) . \]
  Therefore,
  \[ k = (I + \gamma \nabla^2 f (x_z))^{-1} h + o (\| h \|), \]
  so $F^{-1}$ is differentiable at $z$ with Jacobian $(I + \gamma \nabla^2 f
  (x_z))^{-1}$. Since $z$ is arbitrary, differentiating~\eqref{eq:senv-grad}
  shows that $f_{\gamma}$ is twice differentiable on $\mathbb{R}^n$, with
  \[ \nabla^2 f_{\gamma} (z) = \tfrac{1}{\gamma} (I - (I + \gamma \nabla^2 f
     (x_z))^{-1}) = A (\nabla^2 f (x_z)), \qquad A (M) \assign M (I + \gamma
     M)^{-1} . \]
  It remains to verify the stated bounds on $\Omega$,
  so fix $z, w \in \Omega$. The hypothesis $\tmop{prox}_{\gamma f}
  (\Omega) \subseteq \Omega_f$ places $x_z, x_w \in \Omega_f$.

  \tmtextit{Gradient bound.} By \eqref{eq:senv-grad} and $x_z \in \Omega_f$,
  $\| \nabla f_{\gamma} (z) \| = \| \nabla f (x_z) \| \leq D = \tilde{D}$.

  \tmtextit{Smoothness and convexity.} The hypothesis $\mu I \preceq \nabla^2 f (x)
  \preceq L I$ for $x \in \Omega_f$ combines with the spectral map $\lambda \mapsto
  \lambda / (1 + \gamma \lambda)$ of $A$, increasing on $[0, \infty)$, to give
  \[ \tilde{\mu} I \preceq \nabla^2 f_{\gamma} (z) \preceq \tilde{L} I,
     \qquad z \in \Omega . \]
  Since $\Omega$ is convex and $f_{\gamma}$ is twice differentiable, these
  Hessian bounds imply that $f_{\gamma}$ is $\tilde{L}$-smooth and
  $\tilde{\mu}$-convex on $\Omega$.

  \tmtextit{Hessian Lipschitzness.} For $M, N \succeq \mu I$, the resolvent
  identity
  \[ A (M) - A (N) = (I + \gamma M)^{- 1} (M - N) (I + \gamma N)^{- 1} \]
  and $\| (I + \gamma M)^{- 1} \| \leq \tfrac{1}{1 + \gamma \mu}$ give $\| A
  (M) - A (N) \| \leq \tfrac{1}{(1 + \gamma \mu)^2} \| M - N \|$. Moreover,
  integrating $\nabla^2 f \succeq \mu I$ along the segment
  $[x_w, x_z] \subseteq \Omega_f$ gives $\langle \nabla f (x_z) - \nabla f (x_w), x_z - x_w \rangle
  \geq \mu \| x_z - x_w \|^2$, hence
  \[ (1 + \gamma \mu) \| x_z - x_w \|^2 \leq \langle F (x_z) - F (x_w), x_z
     - x_w \rangle = \langle z - w, x_z - x_w \rangle \leq \| z - w \| \|
     x_z - x_w \|, \]
  {\ie}, $\| x_z - x_w \| \leq \tfrac{1}{1 + \gamma \mu} \| z - w \|$.
  Combining the two displays with the $H$-Lipschitz Hessian of $f$ on
  $\Omega_f$,
  \[ \| \nabla^2 f_{\gamma} (z) - \nabla^2 f_{\gamma} (w) \| \leq \tfrac{1}{(1
     + \gamma \mu)^2} \| \nabla^2 f (x_z) - \nabla^2 f (x_w) \| \leq
     \tfrac{H}{(1 + \gamma \mu)^2} \| x_z - x_w \| \leq \tilde{H} \| z - w \|
     . \]

  Suppose now that \Cref{ass:quad} holds. Then $f_{\gamma}$ is quadratic with Hessian
  $Q (I + \gamma Q)^{-1}$. Its spectrum lies in
  $[\tilde{\mu},\tilde{L}]$, and its Hessian is constant, which gives the
  stated global bounds with $\tilde{H}=0$.
\end{proof}

\appendixsubsection{Proof of \Cref{lem:env-template}}\label{app:proof-env-template}

The main-text calculation uses the pointwise matrix
$\nabla^2 g_{\gamma} (M (x))$ and the third derivative of $s$, neither of
which need exist under the assumptions of \Cref{lem:env-template}. We
replace them by finite differences along a partition of the segment joining
$x$ and $y$. The following lemma constructs the required finite secants of
$\nabla g_{\gamma}$.

\begin{lemma}
  \label{lem:secant}Let $h : \mathbb{R}^n \to \mathbb{R}$ be convex and
  $L_h$-smooth. For any $u, v \in \mathbb{R}^n$, there exists a symmetric
  matrix $B_{u v}$ with $0 \preceq B_{u v} \preceq L_h I$ such that
  \[ \nabla h (u) - \nabla h (v) = B_{u v} (u - v) . \]
\end{lemma}

\begin{proof}
  Convexity and $L_h$-smoothness give the Baillon-Haddad cocoercivity
  \[ \| \nabla h (u) - \nabla h (v) \|^2 \leq L_h \langle \nabla
     h (u) - \nabla h (v), u - v \rangle . \]
  Let $w = u - v$ and $\Delta = \nabla h (u) - \nabla h (v)$.
  If $\Delta = 0$, take $B_{u v} = 0$. Otherwise, cocoercivity gives
  $\langle \Delta, w \rangle > 0$, so set
  \[ B_{u v} = \tfrac{1}{\langle \Delta, w \rangle} \Delta \Delta^{\top} . \]
  Then $B_{u v} w = \Delta$. The matrix $B_{u v}$ is positive semidefinite,
  and its only nonzero eigenvalue is $\| \Delta \|^2 / \langle \Delta, w
  \rangle \leq L_h$ by cocoercivity.
\end{proof}

We now prove \Cref{lem:env-template}.
We first work under the first set of assumptions. The second set of assumptions, corresponding to the quadratic case, is a special case with $H(D+G)$ replaced by $0$.

\textbf{(1)} We first prove the smoothness constant. By
\Cref{subsec:moreau}, $g_{\gamma}$ is convex and $L_h$-smooth with $L_h = 1
  / \gamma$. Proximal optimality gives $\nabla g_{\gamma} (z) \in \partial g
(\tmop{prox}_{\gamma g} (z))$, so the $G$-Lipschitz continuity of $g$ gives
$\|\nabla g_{\gamma} (z)\| \le G$. Since $s$ is twice differentiable and
$g_{\gamma}$ is differentiable, differentiating~\eqref{eq:env-template} gives
\begin{equation}
  \nabla E (x) = (I - \alpha \nabla^2 s (x)) (\nabla s (x) + \nabla
  g_{\gamma} (M (x))) . \label{eq:env-grad}
\end{equation}
Fix $x, y \in \Omega$ and a partition
$0 = t_0 < t_1 < \cdots < t_m = 1$. Set
\[ x_i = x + t_i (y - x), \quad \delta_i = x_{i + 1} - x_i, \quad d = y - x.
\]
Since $\Omega$ is convex, each $x_i \in \Omega$, so the Hessian
and gradient bounds of case 1 of \Cref{lem:env-template} apply at every $x_i$. Let $S_i = \nabla^2 s (x_i)$, $J_i = I - \alpha S_i$, $v_i = \nabla s
(x_i)$, $u_i = M (x_i)$, and $a_i = \nabla g_{\gamma} (u_i)$. Since $\alpha L
\leq 1$, we have $0 \preceq J_i \preceq I$. Telescoping \eqref{eq:env-grad},
\begin{equation}
  \nabla E (y) - \nabla E (x) = \sum_i [J_i ((v_{i + 1} - v_i) + (a_{i + 1} -
  a_i)) + (J_{i + 1} - J_i) (v_{i + 1} + a_{i + 1})] . \label{eq:telescope}
\end{equation}
The second summand accounts for Hessian variation.
Since $J_{i + 1} - J_i = - \alpha (S_{i + 1} - S_i)$ and $\| S_{i + 1} - S_i
\| \leq H \| \delta_i \|$, while $\| v_{i + 1} + a_{i + 1} \| \leq D + G$,
\begin{equation}
  \sum_i \| (J_{i + 1} - J_i) (v_{i + 1} + a_{i + 1}) \| \leq \alpha H (D + G)
  \sum_i \| \delta_i \| = \alpha H (D + G) \| d \| . \label{eq:hl-remainder}
\end{equation}
We next recover the curvature obtained by fixing the Hessian at each $x_i$.
By Hessian Lipschitzness applied to $\nabla s$, there exists $e_i$ with
$\| e_i \| \leq (H / 2) \| \delta_i \|^2$ such that
\[ v_{i + 1} - v_i = S_i \delta_i + e_i . \]
Hence $u_{i + 1} - u_i = \delta_i - \alpha (v_{i + 1} - v_i) = J_i \delta_i -
\alpha e_i$. By \Cref{lem:secant} applied to $g_{\gamma}$ at $u_{i + 1},
u_i$, there exists $B_i$ with $0 \preceq B_i \preceq L_h I$ such that
\[ a_{i + 1} - a_i = B_i (u_{i + 1} - u_i) = B_i J_i \delta_i - \alpha B_i e_i
   . \]
Combining these identities and multiplying by $J_i$, we obtain
\begin{equation}
  J_i ((v_{i + 1} - v_i) + (a_{i + 1} - a_i)) = A_i \delta_i + J_i (I - \alpha
  B_i) e_i , \label{eq:principal}
\end{equation}
where $A_i \assign J_i S_i + J_i B_i J_i$.
The matrix $A_i$ is the finite-difference counterpart of $JS + JBJ$ in the
main-text calculation.
Since $S_i$ and $J_i$ commute and $0 \preceq J_i B_i J_i \preceq L_h J_i^2$,
\[ A_i \preceq S_i J_i + L_h J_i^2, \]
whose eigenvalues are $\lambda (1 - \alpha \lambda) + L_h (1 - \alpha
\lambda)^2$ for $\lambda \in [\mu, L]$. The $e_i$-error in
\eqref{eq:principal} satisfies
  \[ \begin{aligned}
       \sum_i \| J_i (I - \alpha B_i) e_i \|
       &\leq (1 + \alpha L_h) \tfrac{H}{2} \sum_i \| \delta_i \|^2\\
       &= (1 + \alpha L_h) \tfrac{H}{2} \| d \|^2
       \sum_i (t_{i + 1} - t_i)^2\\
       &\leq (1 + \alpha L_h) \tfrac{H}{2} \| d \|^2
     \max_i (t_{i + 1} - t_i).
   \end{aligned} \]
The last quantity vanishes as the partition mesh goes to zero. Combining
\eqref{eq:telescope}, \eqref{eq:hl-remainder}, and \eqref{eq:principal}, and
passing to the limit,
\[ \| \nabla E (y) - \nabla E (x) \| \leq (\sup_{\lambda \in [\mu, L]}
   [\lambda (1 - \alpha \lambda) + L_h (1 - \alpha \lambda)^2] + \alpha H (D
   + G)) \| d \| . \]
\textbf{(2)} We next prove the convexity constant. Take inner products of \eqref{eq:telescope} with $d$. Since $J_i B_i J_i
\succeq 0$, the principal operator satisfies $A_i \succeq J_i S_i$, whose
eigenvalues are $\lambda (1 - \alpha \lambda)$. Thus
\[ \langle A_i \delta_i, d \rangle \geq (\inf_{\lambda \in [\mu, L]} \lambda
   (1 - \alpha \lambda)) \langle \delta_i, d \rangle = (\inf_{\lambda \in
   [\mu, L]} \lambda (1 - \alpha \lambda)) (t_{i + 1} - t_i) \| d \|^2 . \]
Summing and adding the Hessian-Lipschitz contribution, which is at worst $-
\alpha H (D + G) \| d \|^2$ by \eqref{eq:hl-remainder} and Cauchy-Schwarz, and
passing to the limit,
\[ \langle \nabla E (y) - \nabla E (x), y - x \rangle \geq (\inf_{\lambda \in
   [\mu, L]} \lambda (1 - \alpha \lambda) - \alpha H (D + G)) \| y - x \|^2 .
\]
This inequality is the $\mu_E$-convexity bound and completes the proof.

In the quadratic case,
the Hessian $S \assign \nabla^2 s$ is constant, so in the
telescoping argument $S_i \equiv S$, $J_i \equiv I - \alpha S$, and
$v_{i + 1} - v_i = S \delta_i$ exactly. Hence the
remainder~\eqref{eq:hl-remainder} and every error $e_i$ vanish, and the
bounds $\| \nabla s \| \leq D$ and $\| \nabla g_{\gamma} \| \leq G$ are
never invoked. By \Cref{subsec:moreau}, $g_{\gamma}$ is convex and
$\tfrac{1}{\gamma}$-smooth for every proper, closed, and convex $g$, so
\Cref{lem:secant} applies to $g_{\gamma}$ as before. The argument above
therefore runs unchanged and yields the constants of
\Cref{lem:env-template} with $H (D + G)$ replaced by $0$, which completes
the proof.

\appendixsubsection{Proof of \Cref{prop:quad}}\label{app:proof-envelope-tightness}

Throughout the proof, $\sigma (A)$ denotes the set of eigenvalues of a
symmetric matrix $A$.

We first derive a common Hessian formula for both the FBE and DRE.
In either instantiation of~\eqref{eq:env-template}, the smooth block $s$
is quadratic: for the FBE, $s = f$ with Hessian $Q$, and for the DRE, $s =
f_{\gamma}$ with Hessian $\tilde{Q} = Q (I + \gamma Q)^{- 1}$.
Hence $S \assign \nabla^2 s$ is constant.
Whenever $g_\gamma$ is twice differentiable, differentiating the uniform envelope template $E$ in~\eqref{eq:env-template} gives
\[
  \nabla^2 E = J S + J \nabla^2 g_\gamma(M(x)) J,
  \qquad
  J \assign I-\alpha S.
\]
For the FBE, $(S,J)=(Q,I-\gamma Q)$. For the DRE,
$(S,J)=(\tilde{Q},I-2\gamma\tilde{Q})$.

The convexity constants are attained by $g=0$. Then
$\nabla^2g_\gamma=0$. Substitution into the Hessian identity gives
\[
  \sigma(\nabla^2 F_\gamma^{\mathrm{FB}})
  = \left\{ \lambda(1-\gamma\lambda):
    \lambda\in\sigma(Q) \right\},
  \qquad
  \sigma(\nabla^2 F_\gamma^{\mathrm{DR}})
  = \left\{ \tfrac{\lambda(1-\gamma\lambda)}
    {(1+\gamma\lambda)^2}:
    \lambda\in\sigma(Q) \right\}.
\]
The derivatives of the two scalar functions are
$1-2\gamma\lambda$ and
$(1-3\gamma\lambda)/(1+\gamma\lambda)^3$, respectively.
Thus neither function has an interior minimum on
$[0,1/\gamma]$, so their minima over $[\mu, L]$ are attained at an
endpoint. Both endpoints lie in $\sigma (Q)$, so the minima over
$\sigma (Q)$ coincide with the minima over $[\mu, L]$ and equal
$\mu_\gamma^{\mathrm{FB}}$ and $\mu_\gamma^{\mathrm{DR}}$.

The smoothness constants are attained by $g=\delta_{\{0\}}$. Then
$g_\gamma(z)=\tfrac{1}{2\gamma}\|z\|^2$ and
$\nabla^2g_\gamma=\tfrac{1}{\gamma}I$. Substitution into the Hessian identity gives
\[
  \sigma(\nabla^2 F_\gamma^{\mathrm{FB}})
  = \left\{ \tfrac{1-\gamma\lambda}{\gamma}:
    \lambda\in\sigma(Q) \right\},
  \qquad
  \sigma(\nabla^2 F_\gamma^{\mathrm{DR}})
  = \left\{ \tfrac{1-\gamma\lambda}
    {\gamma(1+\gamma\lambda)^2}:
    \lambda\in\sigma(Q) \right\}.
\]
The two scalar functions are decreasing on $[0,1/\gamma]$.
Since $\mu \in \sigma (Q)$, their maxima over $\sigma (Q)$ are
attained at $\mu$ and equal
$L_\gamma^{\mathrm{FB}}$ and $L_\gamma^{\mathrm{DR}}$, respectively.

\appendixsubsection{Proof of \Cref{lem:admm-regularity}}\label{app:proof-admm-regularity}

  Suppose first that \Cref{ass:admm} holds. We verify \Cref{ass:smooth}
  for the dual pair: $f$ satisfies the smooth-term
  conditions on $\Lambda$ with the stated constants $(\mu, L, H, D)$,
  and $g$ is convex and Lipschitz continuous with the stated constant
  $G$.

  Since $\psi$ is $L_\psi$-smooth and $\mu_\psi$-strongly convex, the conjugate
  $\psi^{\ast}$ is $\tfrac{1}{\mu_\psi}$-smooth and $\tfrac{1}{L_\psi}$-strongly
  convex. Composing $\psi^{\ast}$ with the linear map $\lambda \mapsto -
  B^{\top} \lambda$ shows that $f$ is $L$-smooth and $\mu$-convex, with
  $L = \tfrac{\| B \|^2}{\mu_\psi}$ and $\mu =
  \tfrac{\lambda_{\min} (B B^{\top})}{L_\psi}$.
  For $y_i \in \mathbb{R}^{n_2}$, let $z_i \assign \nabla
  \psi^{\ast} (y_i)$. The inverse-Hessian identity, the
  $H_\psi$-Lipschitz continuity of $\nabla^2 \psi$, and the
  $\mu_\psi$-strong convexity of $\psi$ give
  \[ 
       \| \nabla^2 \psi^{\ast} (y_1) - \nabla^2 \psi^{\ast} (y_2) \|
       \leq \tfrac{H_\psi}{\mu_\psi^2} \| z_1 - z_2 \|
       \leq \tfrac{H_\psi}{\mu_\psi^3} \| y_1 - y_2 \|,
    \]
  where the second inequality uses the $\tfrac{1}{\mu_\psi}$-Lipschitz
  continuity of $\nabla \psi^{\ast}$. Since $\nabla^2 f (\lambda) = B
  \nabla^2 \psi^{\ast} (- B^{\top} \lambda) B^{\top}$,
  \[ \| \nabla^2 f (\lambda_1) - \nabla^2 f (\lambda_2) \|
     \leq \tfrac{H_\psi \| B \|^3}{\mu_\psi^3} \| \lambda_1 -
     \lambda_2 \| . \]
  Therefore, the Hessian of $f$ is Lipschitz with constant
  $H_f \leq \tfrac{H_\psi \| B \|^3}{\mu_\psi^3} = H$. The gradient
  bound of \Cref{ass:admm} gives $\| \nabla f (\lambda) \| = \| B
  \nabla \psi^{\ast} (- B^{\top} \lambda) \| \leq D$ on $\Lambda$.

  Since $\tmop{dom} \phi \subseteq \{ x : \| x \| \leq R_\phi \}$, the conjugate
  $\phi^{\ast}$ is $R_\phi$-Lipschitz. Hence $g (\lambda) = \phi^{\ast} (-
  A^{\top} \lambda) + b^{\top} \lambda$ is proper, closed, and convex, and
  Lipschitz with constant $G = \| A \| R_\phi + \| b \|$.

  Suppose now that \Cref{ass:admm-quad} holds. The conjugate is the quadratic $\psi^{\ast} (y)
  = \tfrac{1}{2} (y - c)^{\top} Q^{- 1} (y - c)$, so $f$ is a convex
  quadratic with Hessian $B Q^{- 1} B^{\top}$, whose spectrum lies in
  $[\mu, L]$. The proximable block $g$ is closed and convex since
  $\phi^{\ast}$ is closed and convex. Properness of $g$ follows from the
  subproblem condition: fix any $(z, \nu)$, let $x^+$ minimize the
  $x$-subproblem, and set $\bar{\lambda} \assign \nu + \gamma (A x^+ + B
  z - b)$. The optimality condition of the subproblem gives $- A^{\top}
  \bar{\lambda} \in \partial \phi (x^+)$. By the Fenchel--Young
  equality,
  \[ \phi^{\ast} (- A^{\top} \bar{\lambda}) = \langle - A^{\top}
     \bar{\lambda}, x^+ \rangle - \phi (x^+) < \infty, \]
  so $g (\bar{\lambda}) < \infty$, and since $\phi^{\ast}$ never takes
  $- \infty$, $g$ is proper.
  Thus \Cref{ass:quad} holds for the dual pair $(f, g)$ with the same
  $\mu$ and $L$. Finally, in both cases, applying
  \Cref{thm:envelope-regularity} in the DR case with these constants
  yields the stated conclusions.

\appendixsection{Theoretical analysis of OSOP}\label{app:osop-proofs}

\appendixsubsection{Proof of \Cref{prop:os-lower}}\label{app:proof-os-lower}
  
  For part 1, combining the residual bound~\eqref{eq:res-grad} with the
  strong convexity of $F_{\gamma}$ gives
  \[ \| x - T (x) \|^2 \geq \gamma^2 \| \nabla F_{\gamma} (x) \|^2 \geq
     2 \gamma^2 \mu_{\gamma}  (F_{\gamma} (x) - F^{\star}), \quad
     \forall x \in \mathcal{L}_0 . \]
  Comparing with the definition of the residual PL
  constant~\eqref{eq:residual-pl} yields $\tilde{\mu}_{\gamma} \geq
  \gamma^2 \mu_{\gamma}$.

  For part 2, we first treat the FBE. Set $y = T^{\mathrm{FB}} (x)$ and $d = x - y$. The computable
  value formula~\eqref{eq:fbe-computable-value} and the quadratic identity
  for $f$ give
  \[ F_{\gamma}^{\mathrm{FB}} (x) = F (y) + \tfrac{1}{2 \gamma} d^{\top} (I -
     \gamma Q) d. \]
  The proximal optimality condition gives $\tfrac{1}{\gamma} (I - \gamma Q)
  d \in \partial F (y)$. For the DRE, set $p = \tmop{prox}_{\gamma f} (x)$,
  $q = \tmop{prox}_{\gamma g} (2 p - x)$, $y = q$, and $d = p - q = x -
  T^{\mathrm{DR}} (x)$. Using~\eqref{eq:dre-computable-value}, the quadratic
  identity for $f$, and the optimality conditions for the two proximal
  steps, we get the same two relations:
  \[ F_{\gamma}^{\mathrm{DR}} (x) = F (y) + \tfrac{1}{2 \gamma} d^{\top} (I -
     \gamma Q) d, \qquad \tfrac{1}{\gamma} (I - \gamma Q) d \in \partial F
     (y). \]

  Thus both cases reduce to the following bound. Let $D_Q \assign I -
  \gamma Q$ and $v \assign \tfrac{1}{\gamma} D_Q d \in \partial F (y)$.
  Since $F$ is $\mu$-strongly convex,
  $F (y) - F^{\star} \leq \tfrac{1}{2 \mu} \| v \|^2$, and therefore
  \[ F_{\gamma} (x) - F^{\star} \leq \tfrac{1}{2 \mu \gamma^2} d^{\top}
     (D_Q^2 + \mu \gamma D_Q) d. \]
  If $\lambda \in [\mu, L]$ and $\delta = 1 - \gamma \lambda$, then
  $\delta \in [0, 1 - \gamma \mu]$ and $\delta + \gamma \mu \leq 1$. Hence
  every eigenvalue of $D_Q^2 + \mu \gamma D_Q$ is bounded by $1 - \gamma
  \mu$, so
  \[ F_{\gamma} (x) - F^{\star} \leq \tfrac{1 - \gamma \mu}{2 \mu \gamma^2}
     \| d \|^2, \]
  which proves the two bounds.

  For part 3, fix $x \in \mathcal{L}_0$ and let $x^{\star}$ denote its projection
  onto $\mathcal{X}^{\star}$. The segment $[x, x^{\star}]$ lies in the
  convex region $\Omega$, on which $F_{\gamma}$ is $L_{\gamma}$-smooth by
  hypothesis, and $\nabla F_{\gamma}$ vanishes on $\mathcal{X}^{\star}$,
  so the descent lemma gives
  \[ F_{\gamma} (x) - F^{\star} \leq \tfrac{L_{\gamma}}{2} \| x -
     x^{\star} \|^2 = \tfrac{L_{\gamma}}{2} \tmop{dist} (x,
     \mathcal{X}^{\star})^2 . \]
  The error bound converts the distance into the residual: $\tmop{dist}
  (x, \mathcal{X}^{\star})^2 \leq \kappa^2 \| x - T (x) \|^2$, so
  $F_{\gamma} (x) - F^{\star} \leq \tfrac{L_{\gamma} \kappa^2}{2} \| x
  - T (x) \|^2$. Comparing with the definition of the residual PL
  constant~\eqref{eq:residual-pl} yields $\tilde{\mu}_{\gamma} \geq
  \tfrac{1}{L_{\gamma} \kappa^2}$.

\appendixsubsection{Proof of \Cref{lem:regret}}\label{app:proof-regret}

  We first show three properties of the feedback: $\ell_x$ is $\min \{
  \mu_{\gamma}, 0 \}$-convex, $\ell_x^{\beta}$ is $(L_{\gamma} +
  \beta)$-smooth, and $\ell_x^{\beta}$ is $G_{\ell}$-Lipschitz on
  $\mathcal{P}$. Write $d = x - T (x)$ and, for $P \in \mathcal{P}$, $u_P
  = x - P d$. By the construction of $\Omega$, $u_P \in \Omega$.
  From~\eqref{eq:osop-feedback}, $\nabla \ell_x (P) = - \nabla F_{\gamma}
  (u_P) d^{\top} / \| d \|^2$.

  (1) Let $P, Q \in \mathcal{P}$. The $\mu_{\gamma}$-convexity of
  $F_{\gamma}$ on $\Omega$ gives
  \[ F_{\gamma} (u_Q) \geq F_{\gamma} (u_P) + \langle \nabla F_{\gamma}
     (u_P), u_Q - u_P \rangle + \tfrac{\mu_{\gamma}}{2} \| u_Q - u_P \|^2 .
  \]
  Subtract $F_{\gamma} (x)$ from both sides and divide by $\| d \|^2$. Since
  $u_Q - u_P = (P - Q) d$, the inner-product term equals $\langle \nabla \ell_x
  (P), Q - P \rangle$, and $\| (P - Q) d \|^2 \leq \| P - Q \|_F^2 \| d
  \|^2$ implies $\tfrac{\mu_{\gamma}}{2} \| (P - Q) d \|^2 / \| d \|^2 \geq
  \tfrac{\min \{ \mu_{\gamma}, 0 \}}{2} \| P - Q \|_F^2$. Hence
  \[ \ell_x (Q) \geq \ell_x (P) + \langle \nabla \ell_x (P), Q - P \rangle +
     \tfrac{\min \{ \mu_{\gamma}, 0 \}}{2} \| Q - P \|_F^2, \]
  so $\ell_x$ is $\min \{ \mu_{\gamma}, 0 \}$-convex, and adding the
  $\beta$-strongly convex ridge with $\beta \geq - \min \{ \mu_{\gamma}, 0
  \}$ makes $\ell_x^{\beta}$ convex.

  (2) Using $\| u v^{\top} \|_F = \| u \| \| v \|$ and the
  $L_{\gamma}$-smoothness of $F_{\gamma}$,
  \[ \| \nabla \ell_x (P) - \nabla \ell_x (Q) \|_F = \tfrac{\| \nabla F_{\gamma}
     (u_P) - \nabla F_{\gamma} (u_Q) \|}{\| d \|} \leq \tfrac{L_{\gamma} \|
     (P - Q) d \|}{\| d \|} \leq L_{\gamma} \| P - Q \|_F, \]
  and the ridge contributes $\beta \| P - Q \|_F$.

  (3) Since $u_I = T (x)$ and the ridge gradient vanishes at $I$, we have
  $\nabla \ell_x^{\beta} (I) = \nabla \ell_x (I) = - \nabla F_{\gamma} (T
  (x)) d^{\top} / \| d \|^2$.
  The segment $[x, T (x)]$ joins two points of the convex set $\Omega$, so the
  $L_{\gamma}$-smoothness of $F_{\gamma}$ and the residual
  bound~\eqref{eq:res-grad} give
  \[ \| \nabla F_{\gamma} (T (x)) \| \leq \| \nabla F_{\gamma} (x) \| +
     L_{\gamma} \| d \| \leq \left( \tfrac{1}{\gamma} + L_{\gamma} \right)
     \| d \|, \]
  and at the anchor
  $\| \nabla \ell_x^{\beta} (I) \|_F = \| \nabla \ell_x (I) \|_F =
  \| \nabla F_{\gamma} (T (x)) \| / \| d \| \leq
  \tfrac{1}{\gamma} + L_{\gamma}$. Combining this with the
  $(L_{\gamma} + \beta)$-smoothness from (2) and
  $\| P - I \|_F \leq \Delta_{\mathcal{P}}$ gives, for every
  $P \in \mathcal{P}$,
  \[
    \| \nabla \ell_x^{\beta} (P) \|_F \leq
    \| \nabla \ell_x^{\beta} (I) \|_F +
    (L_{\gamma} + \beta) \| P - I \|_F \leq G_{\ell},
  \]
  which proves the gradient bound.

  We now bound the regret. Fix $\hat{P} \in \mathcal{P}$ and write $g_k \assign \nabla
  \ell_{x^k}^{\beta} (P_k)$. Nonexpansiveness of the projection gives $\|
  P_{k + 1} - \hat{P} \|_F^2 \leq \| P_k - \hat{P} \|_F^2 - 2 \eta \langle
  g_k, P_k - \hat{P} \rangle + \eta^2 \| g_k \|_F^2$. Summing from $k = 0$
  to $K - 1$, discarding $- \| P_K - \hat{P} \|_F^2 \leq 0$, and using
  convexity of the losses (1) together with the gradient
  bound (3),
  \[
    \begin{aligned}
      \sum_{k = 0}^{K - 1} \left[ \ell_{x^k}^{\beta} (P_k) -
      \ell_{x^k}^{\beta} (\hat{P}) \right]
      &\leq \sum_{k = 0}^{K - 1} \langle g_k, P_k - \hat{P} \rangle \\
      &\leq \tfrac{\| P_0 - \hat{P} \|_F^2}{2 \eta}
      + \tfrac{\eta}{2} \sum_{k = 0}^{K - 1} \| g_k \|_F^2 \\
      &\leq \tfrac{\Delta_{\mathcal{P}}^2}{2 \eta}
      + \tfrac{\eta K G_{\ell}^2}{2}
      = \Delta_{\mathcal{P}} G_{\ell} \sqrt{K} .
    \end{aligned}
  \]
  Since $\hat{P} \in \mathcal{P}$ is arbitrary, the bound holds for the
  maximum defining $R_K$.

\appendixsection{Experimental details}

\appendixsubsection{OS-ADMM for QP}\label{app:os-admm-qp}

We specialize OS-ADMM (\Cref{alg:os-admm}) to the standard-form QP
$\min_x \, \{ \tfrac{1}{2} x^{\top} Q x + q^{\top} x : A x = b, \, x \geq
0 \}$ with $Q \succ 0$, cast as the splitting of \Cref{ex:qp-split}:
\[ \min_{x, z} \quad \underbrace{\delta_{\mathbb{R}_+^n} (x)}_{\phi (x)}
   + \underbrace{\tfrac{1}{2} z^{\top} Q z + q^{\top} z}_{\psi (z)}
   \quad \text{subject to} \quad \left[ \begin{array}{c} 0 \\ I \end{array} \right] x
   + \left[ \begin{array}{c} A \\ - I \end{array} \right] z
   = \left[ \begin{array}{c} b \\ 0 \end{array} \right] . \]
Partition the multiplier $\nu \in \mathbb{R}^{m + n}$ as $\nu = (\nu_1,
\nu_2)$, where $\nu_1$ corresponds to the block $A z = b$ and $\nu_2$ to
$x - z = 0$. The augmented Lagrangian $\mathcal{L}_{\gamma}$ of the
split specializes to
\[ \mathcal{L}_{\gamma}^{\tmop{QP}} (x, z, \nu) =
   \delta_{\mathbb{R}_+^n} (x) + \tfrac{1}{2} z^{\top} Q z + q^{\top} z
   + \nu_1^{\top} (A z - b) + \nu_2^{\top} (x - z)
   + \tfrac{\gamma}{2} \| A z - b \|^2 + \tfrac{\gamma}{2} \| x - z \|^2 . \]

The two subproblems of \Cref{alg:os-admm} have closed forms. The
$x$-update is the projection $x^+ (z, \nu) = \left[ z - \tfrac{1}{\gamma}
\nu_2 \right]_+$. The $z$-update and the hypergradient direction $G^{k +
1}$ are linear solves with the fixed matrix $H_{\gamma} \assign Q + \gamma
(A^{\top} A + I)$, factorized once. \Cref{alg:os-admm-qp} states the
specialized method, matching \Cref{alg:os-admm} step for step with the
preconditioner $P_k \in \mathbb{R}^{(m + n) \times (m + n)}$ acting on the
stacked residual $r_x^k$. In the preconditioner update, the vector $G^{k +
1}$ equals $(2 J_{z^{k + 1 / 2}} - I) \, r_x^{k + 1 / 2}$ of
\Cref{alg:os-admm} by the Woodbury identity, at the cost of one additional
$H_{\gamma}$-solve.

\begin{algorithm}[htbp]
  \DontPrintSemicolon
  \caption{OS-ADMM for QP}
  \label{alg:os-admm-qp}
  \KwInit{$(z^0, \nu^0)$ with $\nu_2^0 - A^{\top} \nu_1^0 = Q z^0 + q$;
  penalty $\gamma > 0$; compact
  convex $\mathcal{P} \ni I$; $P_0 = I$; online stepsize $\eta > 0$;
  regularization $\beta = \max \{ - \mu^{\tmop{ADMM}}_{\gamma}, 0 \}$}
  \For{$k = 0, 1, \ldots, K - 1$}{
    \textbf{Current $x$-update:}
    $x^{k +} = \left[ z^k - \tfrac{1}{\gamma} \nu_2^k \right]_+$,
    $r_x^k = \left[ \begin{array}{c}
      r_1^k\\
      r_2^k
    \end{array} \right] = \left[ \begin{array}{c}
      A z^k - b\\
      x^{k +} - z^k
    \end{array} \right]$; \textbf{stop} if $r_x^k = 0$\;
    \textbf{Probe step:}
    $c_{P_k} = \left[ \begin{array}{c}
      c_1^k\\
      c_2^k
    \end{array} \right] = \left[ \begin{array}{c}
      b - A z^k\\
      z^k
    \end{array} \right] + P_k r_x^k$\;
    \qquad $\begin{aligned}
      z^{k + 1 / 2} = & H_{\gamma}^{- 1} [\nu_2^k - A^{\top} \nu_1^k - q
      + \gamma A^{\top} (b - c_1^k) + \gamma c_2^k]\\
      \nu_1^{k + 1 / 2} = & \nu_1^k + \gamma (c_1^k + A z^{k + 1 / 2} -
      b)\\
      \nu_2^{k + 1 / 2} = & \nu_2^k + \gamma (c_2^k - z^{k + 1 / 2})\\
      x^{k + 1 / 2} = & \left[ z^{k + 1 / 2} - \tfrac{1}{\gamma}
      \nu_2^{k + 1 / 2} \right]_+, \quad
      r_x^{k + 1 / 2} = \left[ \begin{array}{c}
        A z^{k + 1 / 2} - b\\
        x^{k + 1 / 2} - z^{k + 1 / 2}
      \end{array} \right]
    \end{aligned}$\;
    \textbf{Preconditioner update:}
    $P_{k + 1} = \Pi_{\mathcal{P}} \big[ P_k + \tfrac{\eta}{\gamma}
    \tfrac{G^{k + 1} (r_x^k)^{\top}}{\| r_x^k \|^2} - \eta \beta (P_k - I)
    \big]$ with
    $G^{k + 1} = \left[ \begin{array}{c}
      r_1^{k + 1 / 2} - 2 \gamma A w^{k + 1}\\
      r_2^{k + 1 / 2} + 2 \gamma w^{k + 1}
    \end{array} \right]$,
    $w^{k + 1} = H_{\gamma}^{- 1} (A^{\top} r_1^{k + 1 / 2} -
    r_2^{k + 1 / 2})$\;
    {\color{blue}$\triangleright$ \textbf{Efficient variant:} use
    $G^{k + 1} = r_x^{k + 1 / 2}$, equivalently $w^{k + 1} = 0$.}\;
    \textbf{Null step:}
    $(z^{k + 1}, \nu^{k + 1}) = \arg \max \{ \mathcal{L}^{\tmop{QP}}_{\gamma}
    (x^+ (z, \nu), z, \nu) \mid (z, \nu) \in \{ (z^k, \nu^k),
    (z^{k + 1 / 2}, \nu^{k + 1 / 2}) \} \}$\;
  }
  \KwOutput{$(x^{K +}, z^K, q^K)$ with $x^{K +} = \left[ z^K -
  \tfrac{1}{\gamma} \nu_2^K \right]_+$, $q^K = \nu^K + \gamma r_x^K$}
\end{algorithm}

\begin{remark}[Envelope constants for QP]\label{rem:qp-envelope-constants}
  Suppose $Q \succ 0$, so that \Cref{ass:admm-quad} holds with $\mu_\psi =
  \lambda_{\min} (Q)$ and $L_\psi = \lambda_{\max} (Q)$: step 1 solves
  every $x$-subproblem in closed form.
  For the QP split, the constraint matrix $B =
  \bigl[ \begin{smallmatrix} A \\ - I \end{smallmatrix} \bigr]$ of
  problem~\eqref{eq:admm} satisfies $\| B \|^2 = \| A \|^2 + 1$ and
  $\lambda_{\min} (B B^{\top}) = 0$, so the constants of
  \Cref{lem:admm-regularity}
  are $L = \tfrac{\| A \|^2 + 1}{\lambda_{\min} (Q)}$ and $\mu = 0$.
  For $0 < \gamma < \tfrac{1}{L}$, the quadratic case of
  \Cref{thm:envelope-regularity} shows that
  $F^{\tmop{ADMM}}_{\gamma}$ is $\tfrac{1}{\gamma}$-smooth and convex on the
  whole dual space. In particular, $\mu^{\tmop{ADMM}}_{\gamma} = 0$, so the
  regularization term in the preconditioner update vanishes for QP.
\end{remark}

\begin{remark}[Per-iteration cost]
  Each iteration of \Cref{alg:os-admm-qp} makes two linear solves with the
  factorized matrix $H_{\gamma}$: one for the probe iterate $z^{k + 1 / 2}$
  and one for the vector $w^{k + 1}$ in the preconditioner update. All
  other operations are matrix-vector products and projections. Vanilla ADMM
  makes one solve per iteration, so online scaling doubles the number of
  solves per iteration.
\end{remark}

\appendixsubsection{Efficient OS-ADMM variant for QP}
\label{app:efficient-os-admm-qp}

Since \Cref{alg:os-admm-qp} requires two linear solves per iteration, we also test a more efficient variant that uses only one solve per iteration. The
variant approximates the hypergradient $G^{k + 1}$ in the preconditioner update by setting $w^{k + 1} = 0$,
and hence replaces the exact direction $G^{k + 1}$ by the probe residual $r_x^{k + 1 / 2}$.
This approximation removes the second
$H_{\gamma}$-solve, so each iteration has the same one-solve cost as
vanilla ADMM. We use residual feedback when online scaling is combined with
Fast-ADMM, RB-ADMM, and AADMM; the precise residual used by each variant is
described in \Cref{app:practical-os-admm-variants}.

The omitted correction has a controlled size in the stepsize regime of
\Cref{rem:qp-envelope-constants}. To see this, write $B =
\bigl[\begin{smallmatrix} A \\ -I \end{smallmatrix}\bigr]$ and $r \assign
r_x^{k + 1 / 2}$. Since $H_{\gamma} = Q + \gamma B^{\top} B$, the exact
direction satisfies
\[
  G^{k + 1} = (I - 2 \gamma B H_{\gamma}^{-1} B^{\top}) r, \qquad
  0 \preceq \gamma B H_{\gamma}^{-1} B^{\top}
  \preceq \tfrac{\gamma \|B\|^2}{\lambda_{\min} (Q) + \gamma \|B\|^2} I.
\]
For $0 < \gamma < \lambda_{\min} (Q) / \|B\|^2$, which is the condition in
the QP envelope analysis, these bounds give
\[
  \langle r, G^{k + 1} \rangle
  \geq \tfrac{\lambda_{\min} (Q) - \gamma \|B\|^2}
  {\lambda_{\min} (Q) + \gamma \|B\|^2} \|r\|^2 > 0,
  \qquad
  \|G^{k + 1} - r\|
  \leq \tfrac{2 \gamma \|B\|^2}
  {\lambda_{\min} (Q) + \gamma \|B\|^2} \|r\|.
\]
Thus the residual is positively aligned with the exact direction, and the
absolute approximation error vanishes with the fixed-point residual. The
approximation is most accurate when $\gamma \|B\|^2$ is small relative to
$\lambda_{\min} (Q)$. It is a computational surrogate rather than the exact
hypergradient, so the convergence guarantees for \Cref{alg:os-admm} do not
directly apply to this efficient variant.

\appendixsubsection{Practical OS-ADMM variants}
\label{app:practical-os-admm-variants}

All practical variants use the same efficient online-scaled ADMM step and
differ in the pair and penalty supplied to this step or in the update applied
afterward. \Cref{alg:practical-os-admm-step} states the shared step.
\Cref{tab:practical-os-admm-composition} summarizes how each method wraps it. 
The practical variants
commit the probe unconditionally, whereas the efficient OS-ADMM baseline
retains the null step in \Cref{alg:os-admm-qp}. 
All variants initialize
$(z^0,\nu^0)$ as in \Cref{alg:os-admm-qp} and set $P_0=I$. Holding $P_k=I$
recovers the underlying ADMM, Fast-ADMM, RB-ADMM, or AADMM step.
When the penalty $\gamma$ varies, $x^+(z,\nu;\gamma)$ and
$r_x(z,\nu;\gamma)$ denote the $x$-update and residual in
\Cref{alg:os-admm-qp} with displayed $\gamma$.

\begin{algorithm}[htbp]
  \DontPrintSemicolon
  \caption{Shared efficient online-scaled ADMM step for QP}
  \label{alg:practical-os-admm-step}
  \KwInit{input pair $(\bar z^k,\bar\nu^k)$; penalty $\gamma_k>0$;
  compact convex $\mathcal{P}\ni I$; preconditioner $P_k\in\mathcal{P}$;
  online update rule}
  $x^{k+}=x^+(\bar z^k,\bar\nu^k;\gamma_k)$,
  $r_x^k=r_x(\bar z^k,\bar\nu^k;\gamma_k)$\;
  $c_{P_k}^k=\left[\begin{array}{c}b-A\bar z^k\\\bar z^k\end{array}\right]
  +P_kr_x^k$\;
  Compute $(z^{k+1/2},\nu^{k+1/2},x^{k+1/2},r_x^{k+1/2})$ by the
  probe step of \Cref{alg:os-admm-qp} from $(\bar z^k,\bar\nu^k)$ with
  $\gamma=\gamma_k$ and $c_{P_k}=c_{P_k}^k$\;
  Update $P_{k+1}$ by the residual-feedback preconditioner rule of
  \Cref{alg:os-admm-qp} with $G^{k+1}=r_x^{k+1/2}$ and the chosen online
  update rule\;
  \KwOutput{$(z^{k+1/2},\nu^{k+1/2},P_{k+1})$}
\end{algorithm}

\begin{table}[htbp]
  \centering
  \caption{Composition of efficient OS-ADMM and its practical variants.}
  \label{tab:practical-os-admm-composition}
  \begin{tabularx}{\textwidth}{@{}l l l X@{}}
    \toprule
    Method & Input pair & Penalty & Update after the shared step \\
    \midrule
    Efficient OS-ADMM & $(z^k,\nu^k)$ & fixed $\gamma$ & Apply the null step \\
    OS-RB-ADMM & $(z^k,\nu^k)$ & $\gamma_k$ & Apply residual balancing \\
    OS-AADMM & $(z^k,\nu^k)$ & $\gamma_k$ & Apply the safeguarded spectral update \\
    OS-Fast-ADMM & $(\widehat z^k,\widehat\nu^k)$ & fixed $\gamma$ & Extrapolate or restart \\
    \bottomrule
  \end{tabularx}
\end{table}

\paragraph{Penalty-based variants.}
Since RB-ADMM and AADMM are both penalty-adaptive algorithms,
OS-RB-ADMM and OS-AADMM call \Cref{alg:practical-os-admm-step} with the
current pair $(z^k,\nu^k)$ and penalty $\gamma_k$. The probe residual supplies
the feedback for $P_{k+1}$, and the probe pair becomes
$(z^{k+1},\nu^{k+1})$. The two methods differ only in the subsequent update
of the penalty $\gamma_k$.

\textbf{OS-RB-ADMM.}
Residual balancing compares primal and dual residuals and changes the penalty
when their magnitudes are substantially different \cite{he2000alternating}.
For the diagonal preconditioners used in our experiments, we make this
comparison in the online-scaled metric:
\[
  R_{\tmop{pri}}^k \assign
  \|P_{k+1}^{1/2}r_x^{k+1/2}\|, \qquad
  R_{\tmop{dual}}^k \assign \gamma_k
  \|\left[\begin{array}{cc}0&I\end{array}\right]
  P_{k+1}^{-1/2}B(z^{k+1/2}-z^k)\|.
\]
The RB rule increases $\gamma_k$ when $R_{\tmop{pri}}^k$ dominates
$R_{\tmop{dual}}^k$, decreases it when the reverse inequality holds, and
otherwise leaves it unchanged.

\textbf{OS-AADMM.}
AADMM estimates the local curvatures of the two dual components from
differences between ADMM iterates, then uses the resulting spectral
Barzilai--Borwein estimates to update the penalty \cite{xu2017adaptive}. In
the online-scaled variant, the same safeguarded spectral rule is applied to
the iterates produced by \Cref{alg:practical-os-admm-step}. If both
curvature estimates pass the correlation safeguard, the new penalty is their
geometric mean. If only one estimate passes, that estimate sets the penalty,
and if neither passes, the penalty remains unchanged.

\paragraph{OS-Fast-ADMM.}
Fast-ADMM evaluates the ADMM step at an extrapolated pair and updates this pair
by Nesterov-type momentum with a restart rule \cite{goldstein2014fast}. Write
the extrapolated pair as $(\widehat z^k,\widehat\nu^k)$ and initialize it by
$(\widehat z^0,\widehat\nu^0)=(z^0,\nu^0)$. Let $\alpha_0=1$, $E^0=\infty$,
and choose a restart threshold $\eta_{\tmop{F}}\in(0,1)$. After the ADMM step
produces $(z^{k+1},\nu^{k+1})$, define the restart measure
\[
  E^{k+1}\assign \tfrac{1}{\gamma}
  \|\nu^{k+1}-\widehat\nu^k\|^2
  +\gamma\|B(z^{k+1}-\widehat z^k)\|^2.
\]
If $E^{k+1}<\eta_{\tmop{F}}E^k$, Fast-ADMM updates the momentum parameter and
extrapolates both components:
\begin{equation}
  \label{eq:fast-admm-extrapolation}
  \begin{aligned}
    \alpha_{k+1}&=\tfrac{1+\sqrt{1+4\alpha_k^2}}{2},\\
    (\widehat z^{k+1},\widehat\nu^{k+1})
    &=(z^{k+1},\nu^{k+1})
    +\tfrac{\alpha_k-1}{\alpha_{k+1}}
    (z^{k+1}-z^k,\nu^{k+1}-\nu^k).
  \end{aligned}
\end{equation}
Otherwise, it sets $\alpha_{k+1}=1$,
$(\widehat z^{k+1},\widehat\nu^{k+1})=(z^k,\nu^k)$, and
$E^{k+1}=E^k/\eta_{\tmop{F}}$. Online scaling changes only the ADMM step
evaluated at the extrapolated pair. The scaled step uses the affine point
$c_{P_k}$ from \Cref{alg:os-admm-qp}, with $(z^k,\nu^k)$ replaced by
$(\widehat z^k,\widehat\nu^k)$.
At iteration $k$, OS-Fast-ADMM calls \Cref{alg:practical-os-admm-step} with
input pair $(\widehat z^k,\widehat\nu^k)$ and fixed penalty $\gamma$, commits
the returned probe pair as $(z^{k+1},\nu^{k+1})$, and then applies the
extrapolation~\eqref{eq:fast-admm-extrapolation} or the restart rule.

\appendixsubsection{Details on QP experiments}\label{app:qp-benchmarks}

\paragraph{Synthetic benchmark classes.}
We adapt the following three QP classes from \cite{stellato2020osqp}. Each
class varies three structural parameters over a $3 \times 3 \times 3$ grid,
and every instance is generated with random seed zero.

\paragraph{Random sparse QPs.}
For $n$ variables, $m$ constraints, and nonzero density $d$, the problem is
\[
  \begin{aligned}
    \tmop{minimize}_{x \in \mathbb{R}^n} \quad &
    \tfrac{1}{2} x^{\top} (M M^{\top} + 10^{-2} I) x + q^{\top} x \\
    \text{subject to} \quad & A x \leq A \bar{x} + s .
  \end{aligned}
\]
The vectors $q, \bar{x} \in \mathbb{R}^n$ have independent standard Gaussian
entries, and the entries $s_i$ are independent uniform random variables on
$[0,1]$. The nonzero entries of $M \in \mathbb{R}^{n \times n}$ and $A \in
\mathbb{R}^{m \times n}$ are independent standard Gaussian random variables,
with independent supports of density $d$. The construction makes $\bar{x}$
feasible. We use the grid
\[
  n \in \{500,1000,1500\}, \qquad
  \tfrac{m}{n} \in \{2,5,10\}, \qquad
  d \in \{0.005,0.05,0.15\}.
\]

\paragraph{Portfolio QPs.}
For $n$ assets and $k$ factors, we use asset weights $w \in \mathbb{R}^n$ and
factor exposures $y \in \mathbb{R}^k$ to solve
\[
  \begin{aligned}
    \tmop{minimize}_{w,y} \quad &
    w^{\top} D w + \|y\|^2 - \mu^{\top} w \\
    \text{subject to} \quad & y = F^{\top} w, \qquad
    \boldsymbol{1}^{\top} w = 1, \qquad w \geq 0 .
  \end{aligned}
\]
This lifted formulation corresponds to the factor covariance model $F
F^{\top} + D$. The factor-loading matrix $F \in \mathbb{R}^{n \times k}$
has independent standard Gaussian nonzeros on a support of density $d$, the
diagonal entries satisfy $D_{ii} \sim \mathcal{U}(0,\sqrt{k})$, and $\mu \sim
\mathcal{N}(0,I_n)$. We use the grid
\[
  n \in \{500,2000,6000\}, \qquad
  \tfrac{k}{n} \in \{0.1,0.2,0.4\}, \qquad
  d \in \{0.3,0.5,0.8\}.
\]

\paragraph{Constrained optimal-control QPs.}
For state dimension $n_x$, input dimension $n_u$, and horizon length $N$, the
problem is
\[
  \begin{aligned}
    \tmop{minimize}_{x_0,\ldots,x_N,u_0,\ldots,u_{N-1}} \quad &
    x_N^{\top} Q_N x_N
    + \sum_{t=0}^{N-1} \left(x_t^{\top} Q x_t + u_t^{\top} R u_t\right) \\
    \text{subject to} \quad & x_{t+1} = A_{\tmop{dyn}} x_t
    + B_{\tmop{dyn}} u_t, \quad t=0,\ldots,N-1, \\
    & -\bar{x} \leq x_t \leq \bar{x}, \quad t=0,\ldots,N, \\
    & -\bar{u} \leq u_t \leq \bar{u}, \quad t=0,\ldots,N-1, \\
    & x_0 = x_{\tmop{init}} .
  \end{aligned}
\]
We generate $A_{\tmop{dyn}}$ from $I+0.1G$, where $G$ has independent
standard Gaussian entries, and rescale eigenvalues of modulus at least $0.99$
to obtain stable dynamics. The entries of $B_{\tmop{dyn}}$ are independent
standard Gaussian random variables. We set $Q = \tmop{diag}(r)$, where 70\%
of the entries $r_i$ are nonzero and are uniform on $[0,1]$, and set $R =
0.1I$. If $S$ is the solution of the discrete algebraic Riccati equation for
$(A_{\tmop{dyn}},B_{\tmop{dyn}},Q,R)$, then the terminal cost is $Q_N = S
S^{\top}$. Finally, $\bar{x}_i \sim \mathcal{U}(1,2)$, $\bar{u}_i \sim
\mathcal{U}(0,1)$, and $x_{\tmop{init}}$ is uniform on
$[-\bar{x}/2,\bar{x}/2]$. We use the grid
\[
  n_x \in \{10,20,50\}, \qquad
  N \in \{10,20,40\}, \qquad
  \tfrac{n_u}{n_x} \in \{0.2,0.5,1\}.
\]
    
\paragraph{Relative KKT residual.} 
The QP formulation we use in the code is
\begin{align}
  \tmop{minimize}\quad & \tfrac{1}{2} x^{\top} Q x + q^{\top} x \nonumber\\
  \text{subject to} \quad & l_{\tmop{con}} \leq A x \leq u_{\tmop{con}} \nonumber\\
  & l_x \leq x \leq u_x . \nonumber
\end{align}
The relative KKT residual is defined as follows:
\[{\mathrm{relKKT} (x, y) = \max \{ r_{\mathrm{prim}}^{\mathrm{rel}},
   r_{\mathrm{dual}}^{\mathrm{rel}}, r_{\mathrm{gap}}^{\mathrm{rel}} \},} \]
where
\[ {
   r_{\mathrm{prim}}^{\mathrm{rel}} = \tfrac{\| \bar{A} x - \Pi_{[\bar{l},
   \bar{u}]} (\bar{A} x) \|_{\infty}}{1 + \max \left\{ \| \bar{A}
   x\|_{\infty}, \hspace{0.17em} \| \Pi_{[\bar{l}, \bar{u}]} (\bar{A}
   x)\|_{\infty} \right\}}, \qquad
   r_{\mathrm{dual}}^{\mathrm{rel}} = \tfrac{\| Q x + q + A^{\top} \lambda_z +
   \lambda_w \|_{\infty}}{1 + \max \left\{ \|Q x\|_{\infty}, \hspace{0.17em}
   \|q\|_{\infty}, \hspace{0.17em} \|A^{\top} \lambda_z \|_{\infty},
   \hspace{0.17em} \| \lambda_w \|_{\infty} \right\}} .
   } \]
\[ r_{\mathrm{gap}}^{\mathrm{rel}} = \tfrac{| x^{\top} Q x + q^{\top} x +
   \bar{u}^{\top} y_+ + \bar{l}^{\top} y_- |}{1 + \max \left\{ \left|
   \frac{1}{2} x^{\top} Q x + q^{\top} x \right|, \left| - \frac{1}{2}
   x^{\top} Q x - \bar{u}^{\top} y_+ - \bar{l}^{\top} y_- \right| \right\}} .
\]

\end{document}